\documentclass[reqno,oneside,11pt]{amsart}

\usepackage{amsmath, amssymb,epsfig,verbatim}
\usepackage[all,dvips,arc,curve,frame]{xy}
\usepackage{epic,eepic}
\usepackage{enumerate}

\theoremstyle{plain}

\newtheorem{thm}{Theorem}[section]
\newtheorem{cor}[thm]{Corollary}
\newtheorem{pro}[thm]{Proposition}
\newtheorem{lem}[thm]{Lemma}
\newtheorem{fact}[thm]{Fact}

\theoremstyle{definition}

\newtheorem{que}[thm]{Question}
\newtheorem{eg}[thm]{Example}
\newtheorem{rem}[thm]{Remark}

\newenvironment{thm-0}
{{\vv \noindent \bf Main Theorem.$\,$}\it}{\vv}

\newenvironment{thm-A}
{{\vv \noindent \bf Theorem A.$\,$}\it}{\vv}

\newenvironment{thm-*}
{{\vv \noindent \bf Theorem.$\,$}\it}{\vv}

\newenvironment{thm-B}
{{\vv \noindent \bf Theorem B.$\,$}\it}{\vv}

\newenvironment{thm-C}
{{\vv \noindent \bf Theorem C.$\,$}\it}{\vv}

\newenvironment{thm-D}
{{\vv \noindent \bf Theorem D.$\,$}\it}{\vv}

\def\vv{\vspace{0.2cm}}

\newcommand{\mygraph}[1]{\xybox{\xygraph{#1}}}
\newcommand{\sizex}{1cm}
\newcommand{\sizey}{0.6cm}

\def\C{\mathbf{C}}
\def\R{\mathbf{R}}

\def\Z{\mathbf{Z}}

\def\P{\mathbb{P}}

\def\Hilb{{H}}
\def\Hyp{{\mathbb{H}}}
\def\GH{{\mathcal{H}}}
\def\hd{{\sf{dist}}}
\def\dist{{\sf{d}}}
\def\Vect{{\sf{Vect}}}
\def\Tub{{\text{Tub}}}
\def\axe{{\text{Ax}}}
\def\End{{\text{End}}}
\def\Min{{{\rm Min}}}

\def\J{{\sf{J}}}
\def\V{{\sf{V}}}

\def\da{\dasharrow}

\def\dd{{\lambda}}
\def\lgt{L}

\def\eps{{\varepsilon}}
\def\cpd{{\mathbb{P}^2_\C}}

\def\NS{{\sf{N}^1}}
\def\ten{_\R} 
\def\mod{\mathcal{B}}
\def\pts{\mathcal{V}}
\def\PM{\mathcal{Z}}
\def\man{\bar{\mathcal{Z}}}
\def\hypman{{\mathbb{H}}_{\man}}
\def\Isom{{\text{Isom}}}
\def\good{{tight}}
\def\Good{{Tight}}

\def\Aut{{\sf{Aut}}}
\def\Bir{{\sf{Bir}}}

\def\PGL{{\sf{PGL}}}
\def\PSL{{\sf{PSL}}}
\def\GL{{\sf{GL}}}
\def\SL{{\sf{SL}}}

\def\diam{{\text{\rm diam}}}
\def\Ind{{\text{\rm Ind}}}
\def\Exc{{\text{\rm Exc}}}

\def\lld{\left\langle\!\left\langle}
\def\rrd{\right\rangle\!\right\rangle}

\addtocounter{section}{0}             
\numberwithin{equation}{section}       

\usepackage[colorlinks, linktocpage, 
citecolor = blue, linkcolor = blue]{hyperref}

\makeatletter 
\let\@wraptoccontribs\wraptoccontribs
\makeatother

\begin{document}

\setlength{\baselineskip}{0.52cm}        
%
%

\title[Normal Subgroups in the Cremona Group]
{Normal Subgroups in the Cremona Group \\ (long version)}
\date{July 2012}
\contrib[with an appendix by]{Yves de Cornulier}
\author{Serge Cantat}
\address{D\'epartement de math\'ematiques \\
        Universit\'e de Rennes\\
        Rennes\\
        France}
\email{serge.cantat@univ-rennes1.fr}
\author{St\'ephane Lamy}
\address{Mathematics Institute\\
        University of Warwick \\
        Coventry \\
        United Kingdom }
\curraddr{Universit\'e Paul Sabatier  \\
Institut de Math\'ematiques de Toulouse \\
118 route de Narbonne, 31062 Toulouse Cedex 9 \\
France}
\email{slamy@math.univ-toulouse.fr}
\thanks{The second author was supported by a Marie Curie Intra European Fellowship, on leave from the Institut Camille Jordan, Universit\'e Lyon 1.}
\thanks{Main difference with the submitted version: proofs of Corollary \ref{cor:TGromov} and Lemmas \ref{lem:obtuse}, \ref{lem:closer}, \ref{lem:weak}, \ref{lem:canoeing}, \ref{lem:serge} are given}

%
%

%
%

%
%

\begin{abstract} 
Let ${\mathbf{k}}$ be an algebraically closed field.
We show that the Cremona group of all birational transformations of
the  projective plane $\P^2_{\mathbf{k}}$ is not a simple group. The strategy makes use 
of hyperbolic geometry, geometric group theory, and algebraic geometry 
to produce elements in the Cremona group that generate non trivial normal subgroups.

\vv

\noindent{\sc{R\'esum\'e.}} Soit ${\mathbf{k}}$ un corps alg\'ebriquement clos.
Nous montrons que le groupe de Cremona, 
form\'e des transformations birationnelles du plan projectif $\P^2_{\mathbf{k}}$, 
n'est pas un groupe simple. Pour cela nous produisons des \'el\'ements
engendrant des sous-groupes normaux non triviaux \`a l'aide de 
th\'eorie g\'eom\'etrique des groupes, de g\'eom\'etrie hyperbolique et
de g\'eom\'etrie alg\'ebrique. 
\end{abstract}

\maketitle



\setcounter{tocdepth}{1}
\tableofcontents

\section{Introduction}


\subsection{The Cremona group}

The Cremona group in $n$ variables over a field $\mathbf{k}$ is the
group $\Bir(\P^n_\mathbf{k})$ of birational transformations of the projective space $\P^n_{\mathbf{k}}$
or, equivalently, the group of ${\mathbf{k}}$-automorphisms of the field ${\mathbf{k}}(x_1,\cdots,x_n)$ of 
rational functions of $\P^n_{\mathbf{k}}$. In dimension $n=1,$ the Cremona group coincides
with the group $\Aut(\P^1_{\mathbf{k}})$ of regular automorphisms of $\P^1_{\mathbf{k}}$, that is with the
group $\PGL_2({\mathbf{k}})$ of linear projective transformations. 
When $n\geq 2,$ the group $\Bir(\P^n_{\mathbf{k}})$ is infinite dimensional; it contains algebraic 
groups of arbitrarily large dimensions (see \cite{Blanc:2009,Serre:Bourbaki}). In this article, we prove the following theorem.

\begin{thm-0}
If $\mathbf{k}$ is an algebraically closed field,
the Cremona group $\Bir(\P^2_{\mathbf{k}})$  is not a simple group.
\end{thm-0}

{\noindent}This answers one of the main open problems concerning 
the algebraic structure of this group: According to Dolgachev, Manin asked whether
$\Bir(\P^2_\C)$ is a simple group in the sixties; Mumford mentionned also this question
in the early seventies (see \cite{Mumford:1974}); in fact, one can trace this problem back 
to 1894, since Enriques already mentions it in \cite{Enriques:1894}: 
{\sl{
``l'importante questione se il gruppo Cremona contenga alcun sottogruppo invariante (questione alla quale sembra probabile si debba rispondere negativamente)"}}.

The proof of this theorem makes use of geometric group theory and isometric actions 
on infinite dimensional hyperbolic spaces. It provides stronger results when one
works over the field of complex numbers $\C$; before stating these results, a few remarks 
are in order. 

\begin{rem}{\bf{a.--}}
There is a natural topology on the Cremona
group which induces the Zariski topology on the spaces $\Bir_d(\P^2_{\mathbf{k}})$
of birational transformations of degree $d$ (see \cite{Serre:Bourbaki, Blanc:2009} and references therein).
Blanc proved in \cite{Blanc:2010} that $\Bir(\P^2_{\mathbf{k}})$ has only two \textit{closed} normal subgroups
for this topology, 
namely $\{ {\text{Id}} \}$ and $\Bir(\P^2_{\mathbf{k}})$ itself. 
D\'eserti proved that $\Bir(\cpd)$ is perfect, hopfian, and co-hopfian, 
and that its automorphism group is generated
by inner automorphisms and the action of automorphisms of the field of complex numbers
(see \cite{Deserti:2006bis,Deserti:2006,Deserti:2007}). In particular, there is no obvious algebraic reason which explains 
why $\Bir(\cpd)$ is not simple. 

\noindent{\bf{b.--}} An interesting example is given by the group $\Aut[{\mathbb{A}}^2_{\mathbf{k}}]_1$ of polynomial automorphisms of the 
affine plane ${\mathbb{A}}^2_{\mathbf{k}}$ with Jacobian determinant equal to $1$. From Jung's
theorem, one knows that this group is the amalgamated product of the group $\SL_2({\mathbf{k}})\ltimes {\mathbf{k}}^2$ of
affine transformations  with the group of elementary transformations
\[
f(x,y)= (ax+ p(y), y/a), \; a\in \mathbf{k}^*,\; p(y) \in \mathbf{k}[y],
\]
over their intersection (see \cite{Jung:1942,Lamy:2002} and references therein); as such, $\Aut[{\mathbb{A}}^2_{\mathbf{k}}]_1$ acts faithfully on a simplicial tree without fixing any edge or vertex.
Danilov used this action, together with small cancellation theory, to prove that  $\Aut[{\mathbb{A}}^2_{\mathbf{k}}]_1$ is not simple
(see \cite{Danilov:1974} and  \cite{Furter-Lamy:Preprint}).

\noindent{\bf{c.--}} The group $\Bir(\P^2_\C)$ contains a linear group with Kazhdan property 
(T), namely $\PGL_3(\C)$, and it cannot be written as a non trivial free product with amalgamation: see Appendix \ref{deCornulier} of this paper.
Thus, even if our strategy is similar to Danilov's proof, it has to be more involved (see \S \ref{par:Danilovbad}). 
\end{rem}

\subsection{General elements}\label{intro:generic}

Let $[x:y:z]$ be homogeneous coordinates for the projective plane $\P^2_{\mathbf{k}}$.
Let $f$ be a birational transformation of $\P^2_\mathbf{k}$. There
are three homogeneous polynomials $P,$ $Q,$ and $R\in {\mathbf{k}}[x,y,z]$ of the same degree $d,$ 
and without common factor of degree $\geq 1$, such that 
\[
f[x:y:z]=[P:Q:R].
\]
The degree $d$ is called the {\bf{degree}} of $f,$ and is denoted by $\deg(f)$. The space $\Bir_d(\P^2_{\mathbf{k}})$
of birational transformations of degree $d$ is a quasi-projective variety. 
For instance, Cremona transformations of degree $1$ correspond to the automorphism group $\Aut(\P^2_{\mathbf{k}})=\PGL_3({\mathbf{k}})$
and Cremona transformations
of degree  $2$ form an irreducible variety of dimension $14$ (see \cite{Cerveau-Deserti:2009}). 

We define the {\bf{Jonqui\`eres group}} $\J$ as the group of birational transformations
of the plane $\P^2_{\mathbf{k}}$ that preserve the pencil of lines through the point $[1:0:0]$. Let $\J_d$ 
be the subset of $\J$ made of birational transformations of degree $d$, and let $\V_d$ 
be the subset of $\Bir_d(\P^2_{\mathbf{k}})$ whose elements are compositions $h_1\circ f\circ h_2$ 
where $h_1$ and $h_2$ are in $\Aut(\P^2_{\mathbf{k}})$ and $f$ is in $\J_d$. The dimension of
$\Bir_d (\P^2_{\mathbf{k}})$ is equal to $4d+6$ and $\V_d$ is its unique irreducible component
of maximal dimension (see~\cite{Nguyen:these}).

On an algebraic variety $W$, a property is said to be {\bf{generic}} if it is satisfied on the complement 
of a Zariski closed subset of $W$ of codimension~$\geq1$, and is said to be {\bf{general}} if it is satisfied
on the complement of countably many  Zariski closed subsets of $W$ of codimension~$\geq1$.
 
\begin{thm-A}
There exists a positive integer $k$  with the following property. 
Let $d \ge 2$ be an integer.
If $g$ is a general
element of $\Bir_d(\cpd)$, and $n$ is an integer with $n\geq k$, then $g^n$ generates
a normal subgroup of the Cremona group $\Bir(\cpd)$ that does not contain any element
$f$ of $\Bir(\cpd)\setminus \{\rm Id\}$ with $\deg(f)<\deg(g)^n$. 
\end{thm-A}

To prove Theorem A we choose  $g$  in the unique irreducible component of maximal dimension $\V_d$. 
Thus, the proof does not provide any information concerning general elements of the remaining components
of the variety $\Bir_d(\cpd)$. 
As a corollary of Theorem A, {\sl{the group $\Bir(\cpd)$ contains an uncountable number of normal subgroups}} (see \S \ref{par:SQ}).

\begin{rem}{\bf{a.--}}
M. Gizatullin proved that any non trivial element $g \in \Bir(\cpd)$ preserving a pencil of lines generates a normal subgroup equal to $\Bir(\cpd)$; the same conclusion holds for any $g \in \Bir(\cpd)$ with $\deg(g) \le 7$ (see \cite{Gizatullin:1994}, or \cite{Cerveau-Deserti:2009} when $\deg(g)\leq 2$).

\noindent{\bf{b.--}} 
We provide an explicit upper bound for the best possible constant $k$ in Theorem A, namely $k\leq 201021$. 
We do so to help the reader follow the proof. According to Gizatullin's result, if
$g$ is an element of $\Bir(\cpd)\setminus\{\rm Id\}$ with $\deg(g)\leq 7$, then the smallest normal
subgroup of $\Bir(\cpd)$ containing $g$ coincides with $\Bir(\cpd)$. Thus, the best constant
$k$ is larger than $1$. 
Our method gives a rather small constant for large enough degrees, namely $k\leq 375$, 
but does not answer the following question: Does
a general element of $\Bir_d(\P^2_\C)$ generates a strict normal subgroup of $\Bir(\P^2_\C)$
for all large enough degrees $d$ ? 
Note that $g^n$ is not an element of the maximal dimensional component $\V_{d^n}$  
if $n\geq 2$ and $g\in \Bir_d(\cpd)$.
\end{rem}


\subsection{Automorphisms of Kummer and Coble surfaces}

If $X$ is a rational surface, then the group of automorphisms $\Aut(X)$ is conjugate to a subgroup of $\Bir(\P^2_\C)$. 
In \S \ref{par:goodauto}, we study two classes of examples. 
The first one is a (generalized) Kummer surface. 
Let $\Z[i]\subset \C$ be the lattice of Gaussian integers.
We start with the abelian surface $Y=E\times E$ where $E$ is the
elliptic curve $\C/\Z[i]$. The group $\SL_2(\Z[i])$ acts by linear automorphisms on $Y,$ 
and commutes with the order $4$ homothety $\eta(x,y)=(i x, iy)$. The quotient 
$Y/\eta$ is a (singular) rational surface on which $\PSL_2(\Z[i])$ acts by automorphisms.
Since $Y/\eta$ is rational, the conjugacy by any birational map $\phi\colon Y/\eta\da \cpd$ 
provides an isomorphism between $\Bir(Y/\eta)$ and $\Bir(\cpd)$, and therefore an embedding of
$\PSL_2(\Z[i])$  into $\Bir(\cpd)$. In the following statement, we restrict this homomorphism to the
subgroup $\PSL_2(\Z)$ of $\PSL_2(\Z[i])$.

\begin{thm-B} There is an integer $k\geq 1$ with the following property.
Let $M$ be an element of $\SL_2(\Z)$, the  trace of which satisfies $\vert {\text{tr}} (M) \vert \geq 3$. Let $g_M$ 
be the automorphism of the rational Kummer surface $Y/\eta$ which 
is induced by $M$. Then,  ${g_M}^k$ generates 
a proper normal subgroup of  the Cremona group $\Bir(\cpd) \simeq \Bir(Y/\eta)$.
\end{thm-B}

Theorems A and B provide examples of normal subgroups coming respectively 
from general and from highly non generic elements in the Cremona group. In \S 
\ref{par:coble}, we also describe  automorphisms of Coble surfaces that generate non trivial 
normal subgroups of the Cremona group; Coble surfaces are quotients of K3 surfaces, 
while Kummer surfaces are quotient of abelian surfaces. Thanks to a result due to
Coble and Dolgachev, for any algebraically closed field ${\mathbf{k}}$
 we obtain automorphisms of Coble surfaces that generate
proper normal subgroups in the Cremona group $\Bir(\P^2_{\mathbf{k}})$.
The Main Theorem follows from this construction (see Theorem \ref{thm:coblegood}).

\subsection{An infinite dimensional hyperbolic space (see \S \ref{sec:cremona-PM})}

The Cremona group $\Bir(\P^2_{\mathbf{k}})$ acts faithfully on an infinite dimensional hyperbolic space
which is the exact analogue of the classical finite dimensional hyperbolic spaces $\Hyp^n$.
This action is at the heart of the proofs of Theorems A and B.
To describe it, let us consider all rational surfaces $\pi\colon X\to \P^2_{\mathbf{k}}$ obtained from the projective plane 
by successive blow-ups. If $\pi'\colon X'\to \P^2_{\mathbf{k}}$ is obtained from $\pi\colon X\to \P^2_{\mathbf{k}} $ by blowing up more points, 
then there is a natural birational morphism $\varphi\colon X'\to X$, defined by $\varphi=\pi^{-1}\circ\pi'$,  
and the pullback operator $\varphi^*$ embeds the N\'eron-Severi group $\NS(X)\otimes \R$  into $\NS(X')\otimes \R$
(note that $\NS(X)\otimes \R$ can be identified with the second cohomology group $H^2(X,\R)$ when ${\mathbf{k}}=\C$).  
The direct limit of all these  groups $\NS(X)\otimes \R$ is called
the {\bf{Picard-Manin space}} of $\P^2_{\mathbf{k}}$. This infinite dimensional vector space comes together with an
intersection form of signature $(1,\infty)$, induced by the intersection form on divisors; we shall denote this quadratic form by
\[
([\alpha], [\beta])\mapsto [\alpha]\cdot [\beta].
\]
 We obtain in this way an infinite dimensional 
Minkowski space. The set of elements $[\alpha]$ in this space with self intersection $[\alpha]\cdot [\alpha]=1$ 
is a hyperboloid with two sheets, one of them containing classes
of ample divisors of rational surfaces; this connected component is an infinite 
hyperbolic space for the distance $\hd$ defined in terms of the intersection form by
\[
\cosh(\hd([\alpha],[\beta]))= [\alpha]\cdot [\beta].
\]
Taking the completion of this metric space, we get an infinite dimensional, complete, hyperbolic space~$\hypman$.
The Cremona group acts faithfully on the Picard-Manin space, preserves the intersection form, 
and acts by isometries on the hyperbolic space $\hypman$.

\subsection{Normal subgroups in isometry groups}

In Part \ref{part:hyperbolic}, we study the general situation of a group $G$ acting by isometries on a $\delta$-hyperbolic space $\GH$. 
Let us explain the content of the central result of Part \ref{part:hyperbolic}, namely Theorem \ref{thm:criterion}, in the particular case of the Cremona 
group $\Bir(\P^2_{\mathbf{k}})$ acting by isometries on  the hyperbolic space $\hypman$. Isometries of hyperbolic spaces fall into three types: elliptic, parabolic, and hyperbolic. A Cremona transformation $g\in \Bir(\P^2_{\mathbf{k}})$ determines a hyperbolic 
isometry $g_*$ of $\hypman$ if and only if the following equivalent properties are
satisfied: 
\begin{itemize}
\item The sequence of degrees $\deg(g^n)$ grows exponentially fast: 
\[
\dd(g):=\limsup_{n\to \infty} \left( \deg(g^n) ^{1/n} \right) > 1;
\]
\item There is a $g_*$-invariant plane $V_g$ in the Picard-Manin space that intersects $\hypman$ on a curve $\axe(g_*)$ (a geodesic line) on which $g_*$ acts by a non-trivial translation; more precisely, 
$\hd(x, g_*(x))= \log(\dd(g))$ for all $x$ in $\axe(g_*)$.
\end{itemize}
The curve $\axe(g_*)$ is uniquely determined and is called the {\bf{axis}} of $g_*$. 
We shall say that an element $g$ of the Cremona group is {\bf{\good}} if it satisfies the following three properties:
\begin{itemize}
\item The isometry $g_*\colon \hypman\to \hypman$ is hyperbolic;
\item There exists a positive number $B$ such that: If $f$ is an element of $\Bir(\P^2_{\mathbf{k}})$ and $f_*(\axe(g_*))$ contains two points at distance $B$ 
which are at distance at most $1$ from $\axe(g_*)$  then $f_*(\axe(g_*))=\axe(g_*)$;
\item If $f$ is in $\Bir(\P^2_{\mathbf{k}})$ and $f_*(\axe(g_*))=\axe(g_*)$, then $fgf^{-1}=g$ or $g^{-1}$.
\end{itemize}
The second property is a rigidity property of  $\axe(g)$ with respect to isometries $f_*$, for $f$ in $\Bir(\P^2_{\mathbf{k}})$. 
 The third property means that the stabilizer of $\axe(g)$ coincides with the normalizer of the cyclic group $g^\Z$. Applied to the Cremona group, Theorem~\ref{thm:criterion} gives
the following statement. 

\begin{thm-C}
Let $g$ be an element of the Cremona group $\Bir(\P^2_{\mathbf{k}})$. If the isometry
$g_*\colon \hypman\to \hypman$ is tight, there exists a positive integer $k$ such that $g^k$ generates a non trivial normal subgroup of $\Bir(\P^2_{\mathbf{k}})$.
\end{thm-C}

Theorems~A and B are then deduced from the fact that general elements of $\Bir_d(\cpd)$ (resp. automorphisms $g_M$ on rational Kummer surfaces
with $\vert {\text{tr}} (M) \vert \geq 3$) are \good\ elements of the Cremona group
 (see Theorems \ref{thm:kummergood} and \ref{thm:coblegood}).

\begin{rem}
Theorem~C will be proved in the context of groups of isometries of Gromov's $\delta$-hyperbolic spaces ; 
the strategy is similar to the one used by Delzant in \cite{Delzant:1996} to construct normal subgroups in Gromov 
hyperbolic groups. As such, Theorem~C is part of the so-called small cancellation theory; we shall explain and comment on
this in Section \ref{sec:GromovHyp}.
\end{rem}

\subsection{Description of the paper}

The paper starts with the proof of Theorem~C in the general context of $\delta$-hyperbolic 
spaces (section \ref{sec:GromovHyp}), and explains how this general statement can be used in the case of isometry groups of spaces with constant negative curvature (section \ref{sec:NegCurvature}). 

Section \ref{sec:cremona-PM} provides an overview on the Picard-Manin space, the associated hyperbolic space $\hypman$,  and the isometric action of $\Bir(\cpd)$ on this space. Algebraic geometry and geometric group theory are then put together to prove Theorem A (\S \ref{par:generic}), Theorem B, and the Main Theorem (\S \ref{par:goodauto}). At the end of the paper we list a few remarks and comments.
An appendix by de Cornulier proves that the Cremona group cannot be written as a non-trivial amalgamated product.

\subsection{Acknowledgment}

Thanks to Thomas Delzant, Koji Fujiwara, and Vincent Guirardel for interesting discussions
regarding hyperbolic groups and geometric group theory, and to J\'er\'emy Blanc,  Dominique Cerveau, Julie D\'eserti, Igor Dolgachev,
and Jean-Pierre Serre for nice discussions on the Cremona group and Coble surfaces. Marat Gizatullin and the anonymous referees made important comments on the first version of this article, which helped clarify the proof and the exposition.




%
%


\part{Hyperbolic Geometry and Normal Subgroups}
\label{part:hyperbolic}

%
%

\section{Gromov hyperbolic spaces}
\label{sec:GromovHyp}
%
%

This section is devoted to the proof of Theorem~\ref{thm:criterion}, a result similar to Theorem~C but in the general 
context of isometries of Gromov hyperbolic spaces. 

\subsection{Basic definitions}\label{par:basicdef}

\subsubsection{Hyperbolicity}
Let $(\GH, \dist)$ be a metric space. If $x$ is a base point of $\GH,$
the Gromov product of two points  $y, z\in \GH$ is
\[
(y\vert z)_x = \frac{1}{2} \left\{\dist(y,x)+\dist(z,x)-\dist(y,z)\right\}.
\]
The triangle inequality implies that  $(y\vert z)_x$ is  non negative. Let $\delta$ 
be a non negative real number.
The metric space $(\GH, \dist)$ is {\bf{$\delta$-hyperbolic}} in the sense of Gromov  if
\[
(x\vert z)_w \geq \min\{ (x\vert y)_w, \, (y\vert z)_w\}-\delta
\]
for all $x,y,w,$ and $z$ in $\GH$. Equivalently, the space $\GH$ is $\delta$-hyperbolic if, for any $w,x,y,z \in \GH$, we have
\begin{equation}
\label{eq:defequiv}
\dist(w,x) + \dist(y,z) \le \max \left\lbrace \dist(w,y) + \dist(x,z), \dist(w,z) + \dist(x,y) \right\rbrace  + 2\delta.
\end{equation}
Two fundamental examples of $\delta$-hyperbolic spaces are given by 
simplicial metric trees, or more generally by  $\R$-trees, since they are  0-hyperbolic, 
and by the $n$-dimensional hyperbolic space $\Hyp^n$, or more generally {\sc{cat}}(-1) spaces, since both are  $\delta$-hyperbolic for $\delta = \log(3)$ (see \cite[chapter 2]{Ghys-delaHarpe:Book}, \cite{CDP:Book}).

\subsubsection{Geodesics}
A {\textbf{geodesic segment}} from $x$ to $y$ in $\GH$ is an isometry
$\gamma$ from an interval $[s,t]\subset \R$ into $\GH$ such that $\gamma(s)=x$ and $\gamma(t)=y$.
The metric space $(\GH, \dist)$ is a {\bf{geodesic space}} if any pair of points in $\GH$ can be joined by a geodesic segment.
Let $[x,y]\subset \GH$ be a geodesic segment, and let $\gamma\colon[0, l]\to \GH$
be an arc length parametrization of $[x,y]$, with $\gamma(0) = x$ and $\gamma(l) = y$. By convention, for $0 \le s \leq l-t \le l$, we denote by $[x+s,y-t]$ the geodesic segment $\gamma([s,l-t])$.

A  {\textbf{geodesic line}} is an isometry $\gamma$ from $\R$ to its image $\Gamma \subset \GH$.
If $x$ is a point of $\GH$ and $\Gamma$ is (the image of) a geodesic line, a {\textbf{projection}} of $x$ onto
$\Gamma$ is a point $z\in \Gamma$ which realizes the distance from $x$ to $\Gamma$.
Projections exist but are not necessarily unique (in a tree, or in $\Hyp^n$, the projection is unique). 
If $\GH$ is $\delta$-hyperbolic and if $z$ and $z'$ are two projections of a point $x$ on a geodesic $\Gamma,$ the
distance between $z$ and $z'$ is at most $4\delta$.  In what follows, we shall denote
by $\pi_\Gamma(x)$ any projection of $x$ on $\Gamma$. Since all small errors will be written in terms of 
$\theta=8\delta$, we note that the projection $\pi_\Gamma(x)$ is well defined up to an error of $\theta/2$.

\subsection{Classical results from hyperbolic geometry}
\label{par:classical}


In what follows, $\GH$ will be a $\delta$-hyperbolic and geodesic metric space.

\subsubsection{Approximation by trees }\label{par:approxbytrees}
In a $\delta$-hyperbolic space, the geometry of finite subsets is well
approximated by finite subsets of metric trees; we refer to \cite{Ghys-delaHarpe:Book} chapter 2, or \cite{CDP:Book}, for a proof of this basic
and crucial fact.

\begin{lem}[Approximation by trees]
Let $\GH$ be a $\delta$-hyperbolic space, and $(x_0$, $x_1$, $\cdots, x_n)$ a finite list of points in $\GH$.
Let $X$ be the union of the $n$ geodesic segments $[x_0, x_i]$ from the base point $x_0$ to each other $x_i,$ $1\leq i\leq n$.
Let $k$ be the smallest integer such that $2n\leq 2^k+1$.
Then there exist a metric tree $T$ and a map $\Phi: X \to T$ such that:
\begin{enumerate}
\item $\Phi$ is an isometry from $[x_0,x_i]$ to $[\Phi(x_0),\Phi(x_i)]$ for all $1\leq i \leq n$;
\item for all $x,y \in X$
\[
\dist(x,y) - 2k\delta \leq \dist(\Phi(x),\Phi(y)) \leq \dist(x,y).
\]
\end{enumerate}
\end{lem}
The map $\Phi \colon X\to T$ is called an {\textbf{approximation tree}}. 
For the sake of simplicity, the distance in the tree is also denoted 
$\dist(\cdot, \cdot)$. However, to avoid any confusion, we stick to the convention that a point in the tree will always be written under the form $\Phi(x)$, with $x \in X$. \\

\noindent{\textbf{Convention.}} When $n \leq 4$ we can choose $k=3$.  Let $\theta$ be the positive real number 
defined by 
\[
\theta = 8\delta.
\] 
we have
\begin{equation}
\dist(x,y) - \theta \leq \dist(\Phi(x),\Phi(y)) \leq \dist(x,y).
\end{equation}
From now on we fix such a $\theta$, and we always use the approximation lemma in this way
for at most $5$ points $(x_0, \cdots, x_4)$. 
The first point in the list will always be taken as the base point and will be denoted by a white circle in the pictures.
Since two segments with the same extremities are $\delta$-close, the specific choice of the segments between $x_0$ and the $x_i$ is not
important. We may therefore forget which segments are chosen, and refer to an approximation tree as a pair $(\Phi, T)$ associated
to $(x_0, \cdots, x_n)$. In general, the choice of the segments between $x_0$ and the $x_i$ is either clear, or irrelevant.\\

In the remaining of \S \ref{par:classical}, well known facts from hyperbolic geometry are listed; complete proofs can also be found in \cite{Ghys-delaHarpe:Book}, \cite{CDP:Book}. First, the following corollary is an immediate consequence of the approximation lemma, and Lemma \ref{lem:obtuse} follows
from the corollary.

\begin{cor}
\label{cor:TGromov}
Let $\Phi\colon X\to T$ be an approximation tree for at most $5$ points. Then
\[
(\Phi(x)\vert \Phi(y) )_{\Phi(z)} - \frac{\theta}{2}
\leq
( x \vert y )_z
\leq
(\Phi (x)\vert \Phi (y) )_{\Phi(z)} + \theta
\]
for all $x,$ $y,$ and $z$ in $X$.
In particular $( x\vert y )_z \leq  \theta$ as soon as $\Phi(z) \in [\Phi(x),\Phi(y)]$.
\end{cor}

\begin{proof} By definition of the Gromov product we have
\begin{eqnarray*}
2 (x\vert y)_z & = &  \dist(x,z) + \dist(y,z) - \dist(x,y)   \\
& \ge &  \dist(\Phi(x),\Phi(z)) + \dist(\Phi(y),\Phi(z)) - \dist(\Phi(x),\Phi(y)) - \theta   \\
& = & 2 (\Phi(x)\vert \Phi(y) )_{\Phi(z)} -\theta.
\end{eqnarray*}
On the other hand
\begin{eqnarray*}
2( x\vert y )_z & = & \dist(x,z) + \dist(y,z) - \dist(x,y)   \\
& \leq &  \dist(\Phi(x),\Phi(z))+\theta + \dist(\Phi(y),\Phi(z))+\theta - \dist(\Phi(x),\Phi(y))  \\
& = &  2(\Phi(x)\vert\Phi(y) )_{\Phi(z)} + 2\theta.
\end{eqnarray*}
\end{proof}

\begin{lem}[Obtuse angle implies thinness]
\label{lem:obtuse}
Let $\Gamma \subset \GH$ be a  geodesic line. Let $x$ be a point of $\GH$ and  $a \in \Gamma$ be a projection of $x$ onto $\Gamma$.  For all
$b$ in $\Gamma$ and  $c$ in the segment $[a,b]\subset \Gamma$, we have
$( x\vert b)_{c} \leq 2\theta$.
\end{lem}

\begin{proof}
Let $y$ and $z$ be two points of $\Gamma$ such that $a$ is the middle of $[y,z]\subset \Gamma$,  $b$ is contained in $[a,z]$, and $\dist(y,z) \ge 10 \theta$.
Let $\Phi\colon X\mapsto T$ be an approximation tree of $(y,z,x)$, where $X$ is the union of the segment $[y,z]\subset \Gamma$ and a segment $[y,x]$.
Let $p$ be the point of $\Gamma$ which is mapped onto the branch point of the tripode $T$ by $\Phi.$
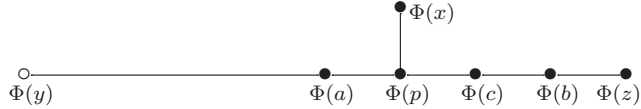
$$\mygraph{
!{<0cm,0cm>;<\sizex,0cm>:<0cm,\sizey>::}
!{(0,0)}*{\circ}="y" !{(4,0)}*{\bullet}="a"
!{(5,0)}*{\bullet}="p" !{(5,1.5)}*{\bullet}="x"
!{(6,0)}*{\bullet}="c" !{(7,0)}*{\bullet}="b"
!{(8,0)}*{\bullet}="z"
"x"-^<{\Phi(x)}"p"
"y"-_<{\Phi(y)}"a" "a"-_<{\Phi(a)}"p" "p"-_<{\Phi(p)}"c"
"c"-_<{\Phi(c)}"b" "b"-_<{\Phi(b)}_>{\Phi(z)}"z"
}$$
We have
\begin{eqnarray*}
\dist(\Phi(x),\Phi (p)) + \theta & \geq & \dist(x,p) \\
&\geq & \dist(x,a) \\
& \ge & \dist(\Phi(x),\Phi(a)) \\
& = & \dist(\Phi(x),\Phi(p)) + \dist(\Phi(p),\Phi(a)),
\end{eqnarray*}
because $\Phi(a)$ is contained in the segment $\Phi([y,z]).$ In particular
$ \dist(\Phi(p),\Phi(a))$ is bounded by $\theta$;
from this and Corollary \ref{cor:TGromov} we get
\begin{eqnarray*}
(x\vert b)_{c} & \leq & (\Phi(x),\Phi(b))_{\Phi(c)} +\theta \\
& \leq & \dist(\Phi(a),\Phi(p)) + \theta,
\end{eqnarray*}
so that $(x\vert b)_{c}$ is bounded by $2\theta.$
\end{proof}

\subsubsection{Shortening and weak convexity}

The following lemma says that two geodesic segments are close as soon as there
extremities are not too far. Lemma \ref{lem:serge} makes this statement much more precise
in the context of {\sc{cat}}(-1) spaces.

\begin{lem}[A little shorter, much closer. See 1.3.3 in \cite{Delzant:1996}]
\label{lem:closer}
Let $[x,y]$ and $[x',y']$ be two geodesic segments of $\GH$ such that
\begin{enumerate}[(i)]
\item $\dist(x,x') \leq \beta,$ and $ \dist(y,y') \leq \beta$
\item  $\dist(x,y) \ge 2\beta + 4 \theta$.
\end{enumerate}
Then, the geodesic segment $[x'+\beta +\theta,y'-\beta -\theta]$ is in the $2\theta$-neighborhood of~$[x,y]$.
\end{lem}

\begin{proof}
Let $z'$ be a point of $[x',y'].$ Consider the approximation
tree $\Phi \colon X \mapsto T$ of $(x,y,x',z',y').$ Note $p_1$ (resp. $p_2$, resp. $p_3$) the points in $[x,y]$ such that $\Phi(p_1)$ (resp. $\Phi(p_2)$, resp. $\Phi(p_3)$) is the branch point of the tripode $(\Phi(x),\Phi(x'),\Phi(y))\subset T$
(resp.  $(\Phi(x),\Phi(y'),\Phi(y))$, resp. $(\Phi(x),\Phi(z'),\Phi(y))).$ One easily checks that $p_3$ is
contained in $[p_1,p_2]$  as soon as
\[
\min \left( \dist(z',x'), \dist(z',y') \right) \geq \frac{\theta}{2} +\beta.
\]
When $z'$ is in $[x'+\beta +\theta,y'-\beta -\theta]$ this inequality is satisfied. In particular,
the tree $T$ looks as follows:

$$\mygraph{
!{<0cm,0cm>;<\sizex,0cm>:<0cm,\sizey>::}
!{(7,1)}*{\bullet}="y'" !{(0,0)}*{\circ}="x"
!{(1,1)}*{\bullet}="x'" !{(4,1)}*{\bullet}="z'"
!{(4,0)}*{\bullet}="fork" !{(8,0)}*{\bullet}="y"
!{(2,0)}*{\bullet}="fork-1" !{(6,0)}*{\bullet}="fork+1"
"z'"-^<{\Phi(z')}"fork"
"x'"-^<{\Phi(x')}"fork-1"
"y'"-_<{\Phi(y')}"fork+1"
"fork-1"-_>{\Phi(p_3)}"fork" "fork"-"fork+1"
"x"-_<{\Phi(x)}_>{\Phi(p_1)}"fork-1"
"fork+1"-_<{\Phi(p_2)}_>{\Phi(y)}"y"
}$$

Therefore we have $\dist(\Phi(z'), \Phi(p_3)) = (\Phi(x') \vert \Phi(y'))_{\Phi(z')}$. By Corollary \ref{cor:TGromov} we have
$(\Phi(x') \vert \Phi(y'))_{\Phi(z')} \le \frac{\theta}{2}$. We obtain
\[
\dist(z',[x,y]) \le \dist(z',p_3) \le  d(\Phi(z'), \Phi(p_3)) + \theta \leq 2\theta,
\]
and $[x'+\beta +\theta,y'-\beta -\theta]$ is in the $2\theta$-neighborhood of~$[x,y]$.
\end{proof}

The following lemma is obvious in a {\sc{cat}}(0) space (by convexity of the distance, see \cite[p. 176]{Bridson-Haefliger:Book}).

\begin{lem}[Weak convexity]
\label{lem:weak}
Let $[x,y]$ and $[x',y']$ be two geodesic segments of $\GH$ such that $\dist(x,x')$ and $\dist(y,y')$ are bounded from above by $\beta$.
Then $[x',y']$ is in the $(\beta+2\theta)$-neighborhood of $[x,y]$.
\end{lem}

\begin{proof}
Pick $z' \in [x',y']$, and consider an approximation tree of $(x,y, x',y',z')$. Note $p_1 \in [x,y]$ (resp. $p_2$) the point such that $\Phi(p_1)$ (resp. $\Phi(p_2)$) is the branch point of the tripod $(\Phi(x),\Phi(y),\Phi(x'))$ (resp. $(\Phi(x),\Phi(y),\Phi(y'))$).
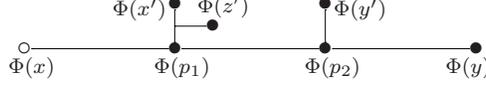
$$\mygraph{
!{<0cm,0cm>;<\sizex,0cm>:<0cm,\sizey>::}
!{(6,0)}*{\bullet}="y" !{(0,0)}*{\circ}="x"
!{(2,1)}*{\bullet}="x'" !{(4,1)}*{\bullet}="y'"
!{(4,0)}*{\bullet}="p2" !{(2.5,0.5)}*{\bullet}="z'"
!{(2,0)}*{\bullet}="p1" !{(2,0.5)}="fork"
"x"-_<{\Phi(x)}"p1"
"p1"-_<{\Phi(p_1)}"p2"
"p2"-_<{\Phi(p_2)}_>{\Phi(y)}"y"
"x'"-_<{\Phi(x')}"p1"
"y'"-^<{\Phi(y')}"p2"
"z'"-_<(-0.6){\Phi(z')}"fork"
}$$
We have
\[
\dist(z', [x,y]) \leq \dist(\Phi(z'), \Phi(p_1)) + \theta \leq
( \Phi(x') \vert \Phi(y') )_{\Phi(z')} + \beta +\theta \leq \beta + 2\theta
\]
and the lemma is proved.
\end{proof}

\subsubsection{Canoeing to infinity}

The following lemma comes from \cite{Delzant:1996} (see \S 1.3.4), and says that  hyperbolic canoes never make loops  (see \S 2.3.13 in \cite{Hubbard:Book}).
This lemma is at the heart of our proof of Theorem~C. 

\begin{lem}[Canoeing]
\label{lem:canoeing}

Let $y_0, \cdots, y_n$ be a finite sequence of points in $\GH,$ such that
\begin{enumerate}[(i)]
\item $\dist(y_i, y_{i-1}) \ge 10 \theta$ for all $1 \leq i \leq n$;
\item $( y_{i+1} \vert y_{i-1} )_{y_i} \leq 3\theta$ for all $1 \leq i \leq n-1$.
\end{enumerate}
Then, for all $1 \leq j \leq n,$
\begin{enumerate}
\item $\dist(y_0,y_j) \ge \dist(y_0, y_{j-1}) + 2\theta$ if $j\ge 1$;
\item $\dist(y_0,y_j) \ge \sum_{i=1}^j ( \dist(y_i,y_{i-1}) - 7\theta )$;
\item $y_j$ is $5\theta$-close to any geodesic segment $[y_0,y_n]$.
\end{enumerate}
\end{lem}

\begin{proof}
The second assumption implies
\begin{equation}\label{eq:star1}
 \dist(y_{i+1},y_{i-1}) \ge  \dist(y_{i+1},y_{i}) + \dist(y_{i-1},y_{i}) - 6\theta 
\end{equation}
for all $1\leq i\leq n-1$. Together with the first assumption, we get
\begin{equation}\label{eq:star2}
 \dist(y_{i+1},y_{i-1})
 \geq \max \left\lbrace \dist(y_i,y_{i-1}),\, \dist(y_{i+1},y_i) \right\rbrace +4\theta .
 \end{equation}

When applied to $i=1,$ equation (\ref{eq:star1}) implies that
\[
\dist(y_2,y_0)\geq \dist(y_2,y_1) -7\theta + \dist(y_1,y_0) -7\theta + 8\theta.
\]
In particular, properties (1) and (2) are proved for $j=2.$ Now we prove (1) and (2) by
induction, assuming that these inequalities have been proved for all indices up to $j$
(included).
By property (1) and equation (\ref{eq:star2}) (for $i=j$) we have
\begin{eqnarray*}
\dist(y_0,y_j)+\dist(y_{j-1},y_{j+1}) & > & \dist(y_0,y_{j-1}) +2\theta +\dist(y_j,y_{j+1})+4\theta \\
& \geq & \dist(y_0,y_{j-1})   +\dist(y_j,y_{j+1})+6\theta.
\end{eqnarray*}
This, and the hyperbolic inequality (\ref{eq:defequiv}), imply
\[
\dist(y_0,y_j)+\dist(y_{j-1},y_{j+1}) \leq \dist(y_0,y_{j+1})+\dist(y_j,y_{j-1}) + \delta.
\]
Hence
\[
\dist(y_0,y_{j+1})  \geq \dist(y_0,y_j) +\dist(y_{j-1}, y_{j+1})-\dist(y_j,y_{j-1})-\delta.
\]
From equation (\ref{eq:star2}) with $i=j$ and the inequality $4\theta-\delta \ge 2\theta$ we get property (1);  then
equation (\ref{eq:star1}) with $i=j$ implies
\[
\dist(y_0,y_{j+1})\geq \dist(y_0,y_j) + \dist(y_j,y_{j+1}) -(6\theta + \delta).
\]
The induction hypothesis concludes the proof of (1) and (2).\\

We  now  prove assertion (3). First note that, reversing the labeling of the points, we get 
\begin{enumerate}
\item[(2')] $  \dist(y_n,y_j) \ge \sum_{i=j+1}^n ( \dist(y_i,y_{i-1}) - 7\theta ).$
\end{enumerate}
Let $[y_0,y_n]$ be a geodesic segment from $y_0$ to $y_n$ and let $\Phi\colon X\mapsto T$ be an approximation tree for 
$(y_0, y_{j-1}, y_j, y_{j+1}, y_n)$, where $X$ is the union of $[y_0,y_n]$ and three segments from $y_0$ to $y_{j-1}$, $y_j$, and $y_{j+1}$.
First, we prove that
\begin{equation}\label{eq:star3}
 (\Phi(y_n)\vert \Phi(y_{j+1}))_{\Phi(y_0)} > (\Phi(y_n)\vert \Phi(y_j))_{\Phi(y_0)}.
\end{equation}
By Corollary \ref{cor:TGromov} it is sufficient to prove
\[
(y_n\vert y_{j+1})_{y_0} >  (y_n\vert y_j)_{y_0} + \frac{3\theta}{2}.
\]
This inequality is equivalent to
\[
\dist(y_n, y_0)+\dist(y_{j+1},y_0)-\dist(y_n,y_{j+1}) > \dist(y_n,y_0)+\dist(y_j,y_0)-\dist(y_n,y_j)+3\theta,
\]
that is, to
\[
\dist(y_{j+1},y_0) +\dist(y_n,y_j) > \dist(y_j,y_0) + \dist(y_n,y_{j+1}) + 3 \theta,
\]
and therefore follows from property (1) applied to $\dist(y_0,y_{j+1})$ and to
$\dist(y_n, y_j)$ by reversing the labeling.

The inequality (\ref{eq:star3}) shows that the approximation tree $T$ has the following form:
$$\mygraph{
!{<0cm,0cm>;<\sizex,0cm>:<0cm,\sizey>::}
!{(6,1)}*{\bullet}="xj+1" !{(0,0)}*{\circ}="x0"
!{(2,1)}*{\bullet}="xj-1" !{(4,1)}*{\bullet}="xj"
!{(4,0)}*{\bullet}="fork" !{(8,0)}*{\bullet}="xn"
!{(2,0)}*{\bullet}="fork-1" !{(6,0)}*{\bullet}="fork+1"
"xj"-^<{\Phi(y_{j})}"fork"
"xj-1"-^<{\Phi(y_{j-1})}"fork-1"
"xj+1"-^<{\Phi(y_{j+1})}"fork+1"
"x0"-_<{\Phi(y_{0})}"fork"
"fork"-_<{\Phi(p)}_>{\Phi(y_{n})}"xn"
}$$
Choose $p \in [y_0,y_n]$ such that $\Phi(p)$ is the branch point of the tripod generated by $\Phi(y_0)$, $\Phi(y_n)$, and $\Phi(y_j)$ in~$T$. We have
\begin{align*}
\dist(y_j - p)
&\leq \dist(\Phi(y_j),\Phi(p)) + \theta \\
&=  (\Phi(y_{j+1})\vert \Phi(y_{j-1}) )_{\Phi(y_j)} + \theta \\
&\leq ( y_{j+1}\vert y_{j-1} )_{y_j} + \tfrac{\theta}{2} + \theta \\
&\leq  3\theta + \tfrac{\theta}{2} + \theta.
\end{align*}
Thus the point $y_j$ is $5\theta$-close to the segment $[y_0,y_n]$.
\end{proof}

\subsection{Rigidity of axis and non simplicity criterion}


Let $G$ be a group of isometries of a $\delta$-hyperbolic space $(\GH,\dist)$, and $g$ be an element of $G$.
Our goal is to provide a criterion which implies that the normal subgroup $\lld g \rrd$ generated by $g$ 
in $G$ does not coincide with $G$.

\subsubsection{Isometries}\label{par:isometriesGromov}

If $f \in \Isom(\GH)$ is an isometry of a $\delta$-hyperbolic space $(\GH,\dist)$, we define the {\bf{translation length}} $\lgt(f)$ as the 
limit $\lim_{n \to \infty} \frac1n\dist(x, f^n(x))$ where $x$ is an arbitrary  point of $\GH$ (see \cite[p. 117]{CDP:Book}).
We denote by $\Min(f)$ the set of points $y \in \GH$ such that $\dist(y,f(y)) = \lgt(f)$.


There are three types of isometries  $f \in \Isom(\GH)$, respectively termed \textbf{elliptic}, \textbf{parabolic}, \textbf{hyperbolic}. 
An isometry $f$ is elliptic if it admits a bounded orbit (hence all orbits are bounded).
If the orbits of $f$ are unbounded, we say that $f$ is hyperbolic if $\lgt(f) > 0$, and otherwise that $f$ is parabolic. 
 

 Say that $f$ has an {\textbf{invariant axis}} if there is a geodesic line $\Gamma$ such that $f(\Gamma)=\Gamma$.
If $f$ has an invariant axis, then $f$ is either elliptic or hyperbolic,
since the restriction $f|_\Gamma$ is either the identity, a symmetry, or a translation. 

Suppose that $f$ is hyperbolic. If it has an invariant axis $\Gamma$, this axis is contained in $\Min(f)$ and $f$ acts as a translation of length $\lgt(f)$ along $\Gamma$. If $\Gamma$ and $\Gamma'$ are two invariant axis, each of them is in the $2\theta$-tubular neighborhood of the other (apply Lemma \ref{lem:closer}). 
When $f$ is hyperbolic and has an invariant axis, we denote by $\axe(f)$ any invariant geodesic line, even if such a line is a priori not unique. 

\begin{rem}
\label{rem:axe}
If $(\GH,\dist)$ is complete, {\sc{cat}}(0) and $\delta$-hyperbolic, every hyperbolic isometry has an invariant axis 
(see \cite[chap. II, theorem 6.8]{Bridson-Haefliger:Book}).
When $\GH$ is a tree or a hyperbolic space $\Hyp^n$, the set $\Min(g)$ coincides with the unique geodesic line it contains.
If $G$ is a hyperbolic group acting on its Cayley graph, any hyperbolic $g \in G$ admits a periodic geodesic line\footnote{ Sketch of proof: Consider $x, y \in \Min(g)$. If $\Gamma$ is an infinite  geodesic in $\Min(g)$, we have $x,y \in \Tub_{2\theta}(\Gamma)$. By local compacity, there exists $N >0$ such that all balls of radius $2\theta$ in the Cayley graph $\GH$ contains at most $N$ vertices. Thus there is at most $N^2$ segments from the ball $B(x,2\theta)$ to the ball $B(y,2\theta)$. Take $n = (2N)!$, we obtain that $g^n$ fixes all these segments, and therefore also any geodesic $\Gamma \in \Min(g)$.};
thus, there is a non trivial power
$g^k$ of $g$ which has an invariant axis. 
\end{rem}

\subsubsection{Rigidity of geodesic lines}\label{par:rigidityoflines}
If $A$ and $A'$ are two subsets of $\GH$, the {\textbf{intersection with precision $\alpha$}} of $A$ and $A'$  is
the intersection of the tubular neighborhoods $\Tub_\alpha(A)$ and $\Tub_\alpha(A')$:
\[
A \cap_\alpha A' = \left\lbrace x \in \GH; \dist(x,A) \leq \alpha \mbox{ and } \dist(x,A') \leq \alpha \right\rbrace.
\]
Let $\eps$ and $B$ be positive real numbers. A subset $A$ of $\GH$ is \textbf{$(\eps,B)$-rigid} if  
$f(A)=A$ as soon as $f\in G$ satisfies
\[
\diam(A \cap_{\eps} f(A)) \ge B.
\]
This rigidity property involves both the set $A$ and the group $G$.
The set $A$ is {\textbf{$\eps$-rigid}} is there exists a positive number $B>0$ such that $A$ is $(\eps, B)$-rigid.
If $\eps' < \eps$ and $A \subset \GH$ is $\eps$-rigid, then $A$ is also $\eps'$-rigid (for the same constant $B$). The converse holds for geodesic lines (or convex sets) 
when $\eps'$ is not too small:

\begin{lem} \label{lem:rigidGH}
Let $\eps > 2\theta$. If a geodesic line $\Gamma \subset \GH$ is $(2\theta,B)$-rigid, then $\Gamma$ is also $(\eps, B + 6\eps + 4\theta)$-rigid.
\end{lem}

\begin{proof} Let $B'= B + 6 \eps + 4 \theta$.
Suppose that $\diam(\Gamma \cap_{\eps} f(\Gamma)) \ge B'$. We want to show that $f(\Gamma) = \Gamma$. There exist $x,y \in \Gamma$, $x',y'\in f(\Gamma)$ such that $\dist(x,x') \le 2\eps$, $\dist(y,y') \le 2\eps$, $\dist(x,y) \ge B'-2\eps$ and $\dist(x',y') \ge B'-2\eps$. By Lemma \ref{lem:closer}, the segment 
\[
[u,v] := [x' + 2\eps + \theta, y' - 2\eps-\theta]\subset \Gamma
\]
is $2\theta$-close to $[x,y]$, and we have $\dist(u,v) \ge B$. Thus $\Gamma \cap_{2\theta} f(\Gamma)$ contains the two points $u$ and $v$ and  has diameter greater than $B$. Since $\Gamma$ is $(2\theta, B)$-rigid, we conclude that $\Gamma = f(\Gamma)$.
\end{proof}

\subsubsection{\Good\ elements, small cancellation, and the normal subgroup theorem}

We shall say that an element $g \in G$ is \textbf{\good} if
\begin{itemize}
\item $g$ is hyperbolic and admits an invariant axis $\axe(g) \subset \Min(g)$;
\item the geodesic line $\axe(g)$ is $2\theta$-rigid;
\item for all $f \in G$, if $f(\axe(g)) = \axe(g)$ then $fgf^{-1} = g$ or $g^{-1}$.
\end{itemize}
Note that if $g$ is \good, then any iterate $g^n$, $n > 0$, is also \good.

We shall say that $g \in G$ satisfies the \textbf{small cancellation property} if 
$g \in G$ is \good, with a $(14\theta, B)$-rigid axis $\axe(g)$
and 
\[
\frac{\lgt(g)}{20} \geq 60\,  \theta + 2B.
\]
Thus, the small cancellation property requires both tightness and a large enough translation length.
We shall comment on this definition in Section \ref{par:small-cancellation} and explain how it is related to
classical small cancellation properties.

\begin{rem}
If $g$ is \good, then $g^n$ is \good~for all $n\geq 1$, and $g^n$ satisfies the small cancellation property as soon as
$n\lgt(g)\geq 1200\, \theta + 40 B$. 
\end{rem}

We now state the main theorem of Part \ref{part:hyperbolic}; this result implies
Theorem~C from the introduction.

\begin{thm}[Normal subgroup theorem]
\label{thm:criterion}
Let $G$ be a group acting by isometries on a hyperbolic space $\GH$.
Suppose that $g \in G$ satisfies the small cancellation property.
Then any element $h \neq {\rm Id}$ in the normal subgroup $\lld g \rrd \subset G$ satisfies the following
alternative: Either $h$ is a conjugate of $g$, or $h$ is a hyperbolic  isometry with translation length $\lgt(h) > \lgt(g)$.
In particular, if $n\geq 2$,  the normal subgroup $\lld g^n \rrd$ does not contain $g$.
\end{thm}

\subsubsection{Small cancellation properties and complements}\label{par:small-cancellation}
Assume that $g$ is a hyperbolic element of $G$ and $\axe(g) \subset \Min(g)$ is an axis of $g$. 
Let $\lambda \in\R_+^*$ be a positive real number. 
One says that $g$ satisfies the small cancellation property with parameter $\lambda$, or the $\lambda$-small
cancellation property, if the following is satisfied: {\sl{For all elements $f\in G$ with $f(\axe(g))\neq \axe(g)$, the set of points of 
$f(\axe(g))$ at distance $\leq 4 \theta=32 \delta$ from $\axe(g)$
has diameter at most $\lambda \lgt(g)$}}. 

When $g$ is tight, with a $(14\theta, B)$-rigid axis, then $g$ satisfies
this small cancellation property as soon as $\lambda \lgt(g)\geq B$. In particular, $g^n$ satisfies the $\lambda$-small
cancellation property when $n\geq B/(\lambda \lgt(g))$. Assume now that $g$ satisfies the small cancellation 
property with parameter $\lambda$; then $g$ is tight with a $(2\theta, \lambda \lgt(g))$-rigid axis. By Lemma \ref{lem:rigidGH},
the axis of $g$ is $(14\theta, B)$-rigid with $B=\lambda \lgt(g) + 88\theta$. 
This shows that the less precise cancellation property used in Theorem \ref{thm:criterion} corresponds to the following, 
more classical, pair of assumptions
\begin{itemize}
\item $g$ satisfies the $\lambda$-small cancellation property for some $0< \lambda < 1/40$;
\item $\lgt(g)\geq (3640\, \theta)/(1-40\lambda)$.
\end{itemize}
Such hypotheses are well known in geometric group theory (see \cite{Ol}, \cite{Delzant:1996} or \cite{Osin:2010} for example).

\begin{rem}{\bf{a.--}} \label{rem:criterion}
Theorem \ref{thm:criterion} is similar to the main result of \cite{Delzant:1996}, which concerns the case of a hyperbolic group $G$ acting on its Cayley graph; our strategy of proof, presented in paragraph \ref{par:proof}, follows the same lines (see \cite{Delzant:1996}, and also the recent 
complementary article \cite{Chaynikov}). 
There are several results of this type in the literature but none of them seems to contain Theorem \ref{thm:criterion} as a 
corollary; in our setting, the space $\GH$ is not assumed to be locally compact, and the action of $G$ is not assumed to 
be proper. 
This is crucial for  application to  the Cremona group: In this case, the space $\GH$ is locally
homeomorphic to a non-separable Hilbert space, and the stabilizers of points in $G$ are algebraic groups, like $\PGL_3({\mathbf{k}})=\Aut({\mathbb{P}}^2_{\mathbf{k}})$. 

\noindent {\bf{b.--}}  Theorem \ref{thm:criterion} applies in particular to groups acting by isometries on trees: This includes the situations of a free group acting on its Cayley graph, or of an amalgamated product over two factors acting on its Bass-Serre tree. It is a useful exercise, which is done in \cite{Lamy:HDR}, to write down the proof in this particular case ; the proofs of the technical lemmas \ref{lem:presentation} and \ref{lem:infernal} become much more transparent in this setting because one does not need approximations by trees. 

\noindent {\bf{c.--}}  In the definition of a \good\ element, we could impose a weaker list of hypothesis. The main point would be to replace $\axe(g)$ by a long quasi-geodesic segment (obtained, for example, by taking an orbit of a point $x$ in $\Min(g)$); this would give a similar statement, without assuming that $g$ has an invariant axis. Here, we take 
this slightly more restrictive definition because it turns out to be sufficient for our purpose, and it makes the proof less technical; 
we refer to  \cite{Delzant:1996}  or to the recent work \cite{DahGui} for more general viewpoints.  We refer to Section \ref{par:SQ}
for extensions and improvements of Theorem \ref{thm:criterion}.
\end{rem}

We now prove Theorem~\ref{thm:criterion}. First we need to introduce the notion of an admissible presentation.

\subsection{Pieces, neutral segments and admissible presentations}
\label{par:presentation}

Let $g\in G$ be a hyperbolic isometry with an invariant axis; let $\axe(g)$ be such an axis (see Remark \ref{rem:axe}), and  $\lgt = \lgt(g)$
be the translation length of $g$.
All isometries that are conjugate to $g$ have an invariant axis: If $f=sgs^{-1}$, then $s(\axe(g))$ 
is an invariant axis for $f$.

\subsubsection{Axis, pieces, and neutral segments}
Let $[x,y]\subset \GH$ be an oriented geodesic segment, the length of which is at least $20\theta$. The segment $[x,y]$ is \textbf{a piece}
if there exists an element $s$ in $G$ such that $[x,y]$ is contained in the tubular neighborhood $\Tub_{7\theta}(s(\axe(g))$. If $[x,y]$ 
is a piece, the conjugates $f = s g s ^{-1}$ of $g$ and  $f^{-1}=s g^{-1} s^{-1}$ of $g^{-1}$ have the same invariant axis $\Gamma=s(\axe(g))$, and this axis almost contains $[x,y]$.
Changing $f$ to $f^{-1}$ if necessary, we can assume that $\pi_\Gamma(y)$ and $f(\pi_\Gamma(x))$ lie on the same side\footnote{
Since $\dist(x,y)\geq 20\theta$, this assumption is meaningful: It does not depend on the choices of the projections of $x$ and $y$ on $\axe(f)$.} of $\pi_\Gamma(x)$ in $\Gamma$.
This assumption made, the isometry $f$ is called a \textbf{support} of $[x,y]$ (and so $f^{-1}$ is a support of the piece $[y,x]$).
The segment $[x,y]$ is \textbf{a piece of size $p/q$} if furthermore $\dist(x,y)\geq \frac{p}{q} \lgt$.
We say that $[x,y]$ {\textbf{contains a piece of size $p/q$}} if there is a segment $[x',y']\subset [x,y]$ which is such a piece.

A pair of points $(x,y)$ is \textbf{neutral} if none of the geodesic segments $[x,y]$ between $x$ and $y$ contains a piece of size $11/20$; by a slight abuse of notation, we also say that the segment $[x,y]$ is neutral if this property holds 
(even if there are a priori several segments from $x$ to $y$). If $[x,y]$ is neutral and $f$ is an element of the group
$G$, then $f([x,y])$ is also neutral; in other words, being neutral is invariant by translation under isometries $f\in G$.

Our choice of working with ${L}/{20}$ as a unit has no particular significance; we use it for convenience.

\subsubsection{Admissible presentations}

Let $h$ be an element of the normal subgroup $\lld g \rrd \subset G$. We can write $h$ as a product
$h=h_kh_{k-1} \cdots h_1$ where each $h_j$ is a conjugate of $g$ or its inverse:
\[
\forall\,  1\leq i\leq k, \quad \exists\,  s_i\in G, \quad h_i=s_i g s_i^{-1} \text{ or } s_i g^{-1} s_i^{-1}.
\]
By convention, each of the $h_i$ comes with $s_i\in G$ and thus with an invariant axis $\axe(h_i)=s_i(\axe(g))$;
thus, the choice of $s_i$ is part of the data ($s_i$ could be changed into $s_i t$ with $t$ in the centralizer
of $g$).

We fix a base point $x_0\in \GH$. Let us associate three sequences of points $(a_i),$  $(b_i),$ and $(x_i),$
$1\leq i\leq k,$ to the given base point $x_0$ and the factorization of $h$ into the product of the $h_i$. 
Namely, for all $1\leq i \leq k,$ we first set: 
\begin{itemize}
\item $x_i$ is equal to $h_i(x_{i-1})$; in particular $x_{k}=h(x_0)$;
\item $a_i$ is the projection of $x_{i-1}$ on the geodesic line $\axe(h_i)$;
\item $b_i$ is equal to $h_i (a_i)$; in particular, both $a_i$ and $b_i$ are on $\axe(h_i)$.
\end{itemize}
However we do not want the distances $\dist(x_{i-1},a_i) = \dist(b_i, x_i)$ to be too small. 
Thus, if $\dist(x_{i-1},a_i) \le 10\theta$, we translate $a_i$ towards $b_i$ on a distance of $12\theta$, and similarly we translate $b_i$ towards $a_i$. This defines the three sequences $(x_i)$, $(a_i)$ and $(b_i)$.

\begin{lem}
\label{lem:properties}
If the translation length $\lgt$ of $g$ is larger than $480 \theta$ then the following properties hold:
\begin{enumerate}
\item \label{point:1} Each $[a_i,b_i]$ is a subsegment of $\axe(h_i)$ of length at least $\frac{19}{20} \lgt$;
\item \label{point:2} $(x_{i-1}\vert b_i )_{a_i} \leq 2\theta$ and $(a_{i}\vert x_i )_{b_i} \leq 2\theta$;
\item \label{point:3} all segments $[x_{i-1},a_i]$ and $[b_i,x_i]$ have length at least $10 \theta$ (for $1\leq i \leq k$).
\end{enumerate}
\end{lem}

\begin{proof}
Since we performed some translations when $[x_{i-1},a_i]$ was too small, we have 
Property (\ref{point:3}). 
The segment $[a_i,b_i]$ was initially of length $\lgt$, so the first property remains true as long as $\frac{\lgt}{20} \ge 24\theta$, which is our assumption. 
Property (\ref{point:2}) follows from Lemma~\ref{lem:obtuse}.
\end{proof}

We say that $h_k \cdots h_1$ is an \textbf{admissible presentation} of $h$ (with respect to the base point $x_0$) if the three sequences $(x_i)$, $(a_i)$, and $(b_i)$ have the additional property
 that all pairs $(x_{i-1},a_i)$, $1\leq i \leq k$, are neutral. It follows that all pairs $(b_j, x_j)$ are also neutral, because the property of being neutral is stable under translation by isometries.

\begin{lem}
\label{lem:presentation}
Let $x_0$ be a base point in $\GH$.
Let $g$ be an element of $G,$ and $h$ be an element of the normal subgroup generated by $g$. 
Assume that the translation length of $g$ satisfies $28 \theta< \frac{\lgt(g)}{20}$. Then $h$ admits at least one admissible presentation.
\end{lem}

\begin{rem} \label{rem:property4}
Thanks to this lemma, the following property can be added to the first three properties listed
in Lemma \ref{lem:properties}:
\begin{enumerate}
\setcounter{enumi}{3}
\item \label{point:4} \textit{All pairs $(x_{i-1},a_i)$ and $(b_i,x_i)$ are neutral}.
\end{enumerate}
\end{rem}

Lemma \ref{lem:presentation} corresponds to Lemma 2.2 in \cite{Delzant:1996}. We give a complete
proof because it is more involved in our case. 

\begin{proof}[Proof of Lemma \ref{lem:presentation}] As before, denote by $\lgt$ the translation length $\lgt(g)$. 
Start with a decomposition $h= h_k \cdots h_1$ with $h_i=s_i g^{\pm 1} s_i^{-1}$
and construct the sequences of points $(a_i)$, $(b_i)$, and $(x_i)$ as above.
Let $\mathcal{I}$ be the set of indices $1 \le i \le k$ such that $(x_{i-1},a_i)$ is not neutral. 
Suppose $\mathcal{I}$ is not empty (otherwise we are done). Our goal is to modify the
construction in order to get a new decomposition of $h$ which is admissible. \\

{\sl{Changing the decomposition.--}} 
Pick $i \in \mathcal{I}$. Then there are two geodesic  segments $[y,z] \subset [x_{i-1},a_i]$ (with $y \in [x_{i-1},z]$) and a conjugate $f=s g s^{-1}$  of $g$ such that $[y,z]$ is $7\theta$-close to $\axe(f)=s(\axe(g))$ and $\dist(y,z) \ge \frac{11}{20} \lgt$; we fix such a segment $[y,z]$ with a maximal length. 
Let $y',z'$ be projections of $y$ and $z$ on  $\axe(f)$; we have $\dist(y,y') \le 7\theta$ and $\dist(z,z') \le 7\theta$. If $y'$ is contained in the segment $[f(y'),z']\subset \axe(f)$, we change $f$ into $f^{-1}$. 
Replace $h_i$ by the product of three conjugates of $g$ or $g^{-1}$
\[
(h_i f^{-1}h_i^{-1}) h_i f.
\]
This gives rise to a new decomposition of $h$ as a product of conjugates of $g^{\pm 1}$, hence to new sequences of points. 
Let $a'_i$ be a projection of $x_{i-1}$ on $\axe(f)$.
Concerning the sequence $(x_i),$ we have two new points $x'_i = f(x_{i-1})$ and $x'_{i+1} = h_i(x'_i)$; the point $x'_{i+2} = h_i f^{-1} h_i^{-1}(x'_{i+1})$ is equal to the old point $x_i = h_i(x_{i-1})$.
Thus, the neutral pair $(x_{i-1},a_i)$ disappears and is replaced by three new pairs $(x_{i-1},a'_i)$, $(x'_i,a'_{i+1})$, and $(x'_{i+1}, a'_{i+2})$ (see Figure \ref{fig:admissible}). Note that if any of these segments is small then it is automatically neutral; so without loss of generality we can assume that there was no need to move  $a'_i$, $a'_{i+1}$ or $a'_{i+2}$ as in Lemma \ref{lem:properties}.\\

{\sl{Two estimates.--}} Our first  claim is 
\begin{enumerate}
\item $\dist(x_{i-1},a'_i) = \dist(x'_{i+1},a'_{i+2}) \le \dist(x_{i-1},a_i) - \frac{10}{20}\lgt$.
\end{enumerate}

\begin{figure}[ht]
\begin{center}
\setlength{\unitlength}{0.00083333in}
\begingroup\makeatletter\ifx\SetFigFont\undefined%
\gdef\SetFigFont#1#2#3#4#5{%
  \reset@font\fontsize{#1}{#2pt}%
  \fontfamily{#3}\fontseries{#4}\fontshape{#5}%
  \selectfont}%
\fi\endgroup%
{\renewcommand{\dashlinestretch}{30}
\begin{picture}(4270,3339)(0,-10)
\put(1761.760,4820.541){\arc{4823.613}{0.7814}{2.3178}}
\whiten\thicklines
\path(3443.277,3071.847)(3474.000,3122.000)(3423.295,3092.197)(3445.500,3094.016)(3443.277,3071.847)
\thinlines
\put(3402,269){\blacken\ellipse{72}{72}}
\put(3402,269){\ellipse{72}{72}}
\put(2724,2623){\blacken\ellipse{72}{72}}
\put(2724,2623){\ellipse{72}{72}}
\put(1227,1376){\blacken\ellipse{72}{72}}
\put(1227,1376){\ellipse{72}{72}}
\put(3438,2623){\blacken\ellipse{72}{72}}
\put(3438,2623){\ellipse{72}{72}}
\put(3652,1875){\blacken\ellipse{72}{72}}
\put(3652,1875){\ellipse{72}{72}}
\put(123,269){\blacken\ellipse{72}{72}}
\put(123,269){\ellipse{72}{72}}
\put(835,2088){\blacken\ellipse{72}{72}}
\put(835,2088){\ellipse{72}{72}}
\put(728,1981){\blacken\ellipse{72}{72}}
\put(728,1981){\ellipse{72}{72}}
\put(658,2088){\blacken\ellipse{72}{72}}
\put(658,2088){\ellipse{72}{72}}
\put(1156,2231){\blacken\ellipse{72}{72}}
\put(1156,2231){\ellipse{72}{72}}
\put(336,983){\blacken\ellipse{72}{72}}
\put(336,983){\ellipse{72}{72}}
\put(453,1089){\blacken\ellipse{72}{72}}
\put(453,1089){\ellipse{72}{72}}
\put(464,841){\blacken\ellipse{72}{72}}
\put(464,841){\ellipse{72}{72}}
\put(799,2616){\blacken\ellipse{72}{72}}
\put(799,2616){\ellipse{72}{72}}
\put(988,2544){\blacken\ellipse{72}{72}}
\put(988,2544){\ellipse{72}{72}}
\path(2724,2587)(3402,269)
\path(799,2587)(123,269)
\path(1156,2231)(1227,1410)
\path(3438,2623)(3652,1875)
\path(3794,2730)(3793,2730)(3790,2729)
	(3785,2728)(3777,2727)(3766,2725)
	(3752,2722)(3734,2718)(3713,2713)
	(3689,2707)(3661,2700)(3632,2692)
	(3600,2683)(3566,2673)(3532,2661)
	(3497,2649)(3461,2635)(3426,2619)
	(3390,2602)(3356,2583)(3321,2562)
	(3288,2540)(3256,2514)(3225,2486)
	(3195,2454)(3167,2420)(3141,2381)
	(3116,2338)(3094,2291)(3074,2239)
	(3058,2184)(3046,2124)(3039,2069)
	(3035,2012)(3034,1957)(3035,1902)
	(3037,1850)(3042,1800)(3047,1753)
	(3053,1709)(3060,1668)(3067,1630)
	(3075,1595)(3082,1563)(3090,1533)
	(3098,1505)(3107,1478)(3115,1453)
	(3124,1429)(3132,1405)(3141,1381)
	(3151,1357)(3161,1332)(3172,1307)
	(3184,1280)(3197,1251)(3210,1221)
	(3225,1189)(3242,1154)(3260,1117)
	(3280,1077)(3302,1036)(3326,992)
	(3352,947)(3380,902)(3409,856)
	(3441,811)(3474,769)(3512,726)
	(3551,688)(3590,655)(3628,628)
	(3666,605)(3703,587)(3738,573)
	(3773,563)(3807,556)(3841,552)
	(3874,550)(3906,551)(3937,554)
	(3969,559)(3999,565)(4029,572)
	(4057,580)(4085,589)(4111,598)
	(4136,607)(4159,616)(4180,625)
	(4198,633)(4214,640)(4227,647)
	(4238,652)(4246,656)(4258,662)
\whiten\thicklines
\path(4213.341,623.727)(4258.000,662.000)(4200.587,649.237)(4222.275,644.137)(4213.341,623.727)
\thinlines
\path(1156,55)(1155,55)(1153,56)
	(1148,58)(1142,60)(1132,63)
	(1120,68)(1104,74)(1086,81)
	(1064,89)(1041,99)(1015,110)
	(987,123)(958,136)(928,151)
	(897,168)(865,185)(834,205)
	(803,225)(772,248)(741,272)
	(712,298)(683,327)(654,358)
	(627,392)(601,429)(576,470)
	(552,515)(530,564)(510,616)
	(493,673)(479,733)(470,788)
	(463,844)(458,899)(455,952)
	(454,1003)(454,1052)(455,1097)
	(457,1139)(460,1178)(462,1213)
	(466,1246)(469,1277)(473,1305)
	(476,1331)(480,1355)(484,1378)
	(489,1401)(493,1423)(498,1444)
	(503,1466)(509,1489)(516,1512)
	(523,1537)(532,1563)(541,1591)
	(552,1621)(564,1654)(578,1688)
	(594,1725)(611,1764)(631,1805)
	(653,1848)(677,1891)(704,1934)
	(732,1977)(763,2017)(802,2061)
	(843,2099)(884,2131)(925,2158)
	(966,2179)(1007,2196)(1047,2209)
	(1087,2218)(1126,2224)(1165,2226)
	(1203,2226)(1241,2224)(1279,2220)
	(1316,2215)(1352,2208)(1387,2200)
	(1421,2191)(1454,2182)(1484,2173)
	(1512,2164)(1537,2155)(1558,2148)
	(1577,2141)(1591,2135)(1603,2131)(1620,2124)
\whiten\thicklines
\path(1561.808,2132.540)(1620.000,2124.000)(1572.667,2158.911)(1583.067,2139.208)(1561.808,2132.540)
\put(3196,3201){\makebox(0,0)[lb]{\smash{{\SetFigFont{9}{10.8}{\rmdefault}{\mddefault}{\updefault}${\rm Ax}(h_i)$}}}}
\put(1227,1197){\makebox(0,0)[lb]{\smash{{\SetFigFont{9}{10.8}{\rmdefault}{\mddefault}{\updefault}$x'_i = f(x_{i-1})$}}}}
\put(442,2053){\makebox(0,0)[lb]{\smash{{\SetFigFont{9}{10.8}{\rmdefault}{\mddefault}{\updefault}$z$}}}}
\put(728,1803){\makebox(0,0)[lb]{\smash{{\SetFigFont{9}{10.8}{\rmdefault}{\mddefault}{\updefault}$z'$}}}}
\put(763,2196){\makebox(0,0)[lb]{\smash{{\SetFigFont{9}{10.8}{\rmdefault}{\mddefault}{\updefault}$b$}}}}
\put(15,55){\makebox(0,0)[lb]{\smash{{\SetFigFont{9}{10.8}{\rmdefault}{\mddefault}{\updefault}$x_{i-1}$}}}}
\put(514,1054){\makebox(0,0)[lb]{\smash{{\SetFigFont{9}{10.8}{\rmdefault}{\mddefault}{\updefault}$y'$}}}}
\put(550,769){\makebox(0,0)[lb]{\smash{{\SetFigFont{9}{10.8}{\rmdefault}{\mddefault}{\updefault}$a'_i$}}}}
\put(123,983){\makebox(0,0)[lb]{\smash{{\SetFigFont{9}{10.8}{\rmdefault}{\mddefault}{\updefault}$y$}}}}
\put(3844,799){\makebox(0,0)[lb]{\smash{{\SetFigFont{9}{10.8}{\rmdefault}{\mddefault}{\updefault}${\rm Ax}(h_if^{-1}h_i^{-1})$}}}}
\put(1443,1902){\makebox(0,0)[lb]{\smash{{\SetFigFont{9}{10.8}{\rmdefault}{\mddefault}{\updefault}${\rm Ax}(f)$}}}}
\put(3196,2746){\makebox(0,0)[lb]{\smash{{\SetFigFont{9}{10.8}{\rmdefault}{\mddefault}{\updefault}$a'_{i+2}$}}}}
\put(2547,2695){\makebox(0,0)[lb]{\smash{{\SetFigFont{9}{10.8}{\rmdefault}{\mddefault}{\updefault}$b_i$}}}}
\put(1223,2267){\makebox(0,0)[lb]{\smash{{\SetFigFont{9}{10.8}{\rmdefault}{\mddefault}{\updefault}$f(a'_{i})$}}}}
\put(3092,55){\makebox(0,0)[lb]{\smash{{\SetFigFont{9}{10.8}{\rmdefault}{\mddefault}{\updefault}$x_i = x'_{i+2} = h_if^{-1}h_i^{-1}(x'_{i+1})$}}}}
\put(3510,1630){\makebox(0,0)[lb]{\smash{{\SetFigFont{9}{10.8}{\rmdefault}{\mddefault}{\updefault}$x'_{i+1} = h_i(x'_i)$}}}}
\put(728,2695){\makebox(0,0)[lb]{\smash{{\SetFigFont{9}{10.8}{\rmdefault}{\mddefault}{\updefault}$a_i$}}}}
\put(1054,2551){\makebox(0,0)[lb]{\smash{{\SetFigFont{9}{10.8}{\rmdefault}{\mddefault}{\updefault}$a'_{i+1}$}}}}
\end{picture}
}
\end{center}
\caption{}
\label{fig:admissible}
\end{figure}

To prove it, we write
\[
\dist(x_{i-1},a_i) \ge \dist(x_{i-1},y) + \dist(y, z) \ge \dist(x_{i-1},y)+ \tfrac{11}{20} \lgt,
\]
and
\begin{equation}\label{eq:s1}
\dist(x_{i-1},a'_i) \le \dist(x_{i-1},y') \le \dist(x_{i-1},y) + \dist(y,y') \le \dist(x_{i-1},y) + 7\theta;
\end{equation}
claim (1) follows because $7\theta< \frac{\lgt}{20}$.\\

Our second claim is 
\begin{enumerate}
\item[(2)] $\dist(x'_{i},a'_{i+1}) \le \dist(x_{i-1},a_i) - \frac{1}{20}\lgt$.
\end{enumerate}

Let us postpone the proof of this second estimate, and deduce the lemma from these two inequalities. \\

{\sl{Conclusion.--}} 
Since the modification is local, and does not change the points $a_j$ 
for $j\neq i$, we can perform a similar replacement for all indices $i$ in $\mathcal{I}$. We obtain a new presentation for $h$ and a new list of bad indices $\mathcal{I}'$. Either this new list is empty, or by the two estimates (claims (1) and (2)) the maximum of the lengths $\dist(x_{j-1},a_j)$ over all the non neutral pairs $(x_{j-1},a_j)$, $j \in \mathcal{I}',$ drops at least by $\frac{1}{20}\lgt$. By induction, after a finite number of such replacements, we obtain an admissible presentation. \\

{\sl{Proof of the second estimate.--}} To prove estimate (2), 
denote by $b$ the projection of $a_i$ on $\axe(f)$. Similarly as (\ref{eq:s1})we have
\begin{equation}\label{eq:s2}
\dist(b,a_i) \le \dist(z,a_i) + 7\theta .
\end{equation}
There are two cases, according to the position of $f(a_i')$ with respect to the segment
$[a_i',b]\subset \axe(f)$. If $f(a_i')$ is in this segment, then 
\begin{eqnarray*}
d(x'_i, a_i)  & \leq & d(f(x_{i-1}),f(a'_i)) + d(f(a'_i),b) + d(b,a_i)\\
& \leq & d(x_{i-1},a'_i) + d(a'_i,b) - L + d(b, a_i).
\end{eqnarray*}
Applying Lemma \ref{lem:obtuse} for the triangles $x_{i-1}$, $y'$, $a_i'$ 
and  $a_i$, $z'$, $b$, and the inequalities $\dist(y,y')$, $\dist(z,z')\leq 7\theta$,
we get $\dist(x_{i-1},y)=\dist(x_{i-1},a_i') + \dist(y, a_i')$ up to an error of $11 \theta$,
and similarly $\dist(a_{i},z)=\dist(a_i,b) + \dist(z, b)$ up to $11 \theta$.
Hence, we get
\begin{eqnarray*}
d(x'_i, a_i) & \leq & d(x_{i-1},y) + d(y,z)  + d(z, a_i) - L + 22 \theta\\
& \leq & d(x_{i-1}, a_i) - L + 22 \theta.
\end{eqnarray*}
This concludes the proof of (2) in the first case, because $- L + 22 \theta < -L/20$.

The second case occurs when $b$ is in the segment $[a_i',f(a_i')]\subset \axe(f)$
(see Figure~\ref{fig:admissible}). 
In this case we have
\begin{eqnarray*}
\dist(x'_{i},a_{i})
&\le& \dist(f(x_{i-1}),f(a'_i)) + \dist(f(a'_i),b) + \dist(b,a_i)\\
&\le& \dist(x_{i-1},z ) - \tfrac{11}{20}\lgt + 7\theta
+ \lgt - \dist(a'_i,b) + d(z,a_i) + 7\theta,  
\end{eqnarray*}
and thus

\begin{eqnarray}\label{eq:s3}
\dist(x'_{i},a_{i}) &\le& \dist(x_{i-1},a_i ) - \dist(a'_i,b)  + \tfrac{9}{20}\lgt + 14 \theta .
\end{eqnarray}

On the other hand, the triangular inequality implies
\begin{eqnarray*}
\dist(x_{i-1},a'_i) + \dist(a'_i, b) + \dist(b,a_i)
& \ge & \dist(x_{i-1},a_i) \\
&=& \dist(x_{i-1},y)+\dist(y,z)+\dist(z,a_i).
\end{eqnarray*}
and so
\[
\dist(a'_i, b) \ge \tfrac{11}{20}\lgt
+ \dist(x_{i-1},y) -\dist(x_{i-1},a'_i)
+  \dist(z,a_i) -  \dist(b,a_i).
\]
Using the inequalities (\ref{eq:s1}) and  (\ref{eq:s2}) we obtain
\[
\dist(a'_i, b) \ge \tfrac{11}{20}\lgt  - 14\theta.
\]
Finally inequality (\ref{eq:s3}) gives
\[
\dist(x'_{i},a'_{i+1}) \le \dist(x'_{i},a_{i})\le \dist(x_{i-1},a_{i}) - \tfrac{2}{20}\lgt + 28\theta
\]
hence claim (2) because $28\theta< \frac{\lgt}{20}$.
\end{proof}

The following lemma provides a useful property of admissible presentations
with a minimum number of factors $h_i$.

\begin{lem}
\label{lem:mini}
Let $h = h_k \cdots h_1$ be an admissible presentation with base point $x_0$. If there exist two indices $j > i$ such that $h_j = h_i^{-1}$, then $h$ admits an admissible presentation with base point $x_0$ and only $k-2$ factors.
\end{lem}

\begin{proof}
We may assume that $j\geq i+2$, otherwise the simplification is obvious.  The decomposition for $h$
is then
\[
h = h_k \cdots h_{j+1} h_i^{-1} h_{j-1} \cdots h_{i+1} h_i h_{i-1} \cdots h_1.
\]
Note that $i$ can be equal to $1$, and $j$ can be equal to $k$.
We have a sequence of triplets $(a_i,b_i,x_i)$, $i = 1, \cdots, k$, associated with this presentation and with the base point $x_0$
Then we claim that
\[
h =  h_k \cdots  h_{j+1} (h_i^{-1} h_{j-1} h_i)(h_i^{-1} h_{j-2} h_i) \cdots (h_i^{-1} h_{i+1} h_i)  h_{i-1} \cdots h_1
\]
is another admissible presentation with base point $x_0$ and with $k-2$ factors. 
Indeed the sequence of $k-2$ triplets associated with this new presentation is
\begin{multline*}
(a_1,b_1,x_1), \cdots, (a_{i-1},b_{i-1},x_{i-1}),(h_i^{-1}(a_{i+1}),h_i^{-1}(b_{i+1}),h_i^{-1}(x_{i+1})), \cdots,\\
(h_i^{-1}(a_{j-1}),h_i^{-1}(b_{j-1}),h_i^{-1}(x_{j-1})),(a_{j+1}, b_{j+1}, x_{j+1}), \cdots, (a_k,b_k,x_k)
\end{multline*}
and one checks that all relevant segments are neutral because they are obtained from neutral segments of the previous presentation
by isometric translations (either by $\rm Id$ or by $h_i$). 
\end{proof}

\subsection{Proof of the normal subgroup theorem}
\label{par:proof}

Let $g$ be an element of $G$ which satisfies the small cancellation property. By definition, $g$ is a \good\ element of~$G$, its 
axis is $(14\theta, B)$-rigidity for some $B>0$, and 
\[
\frac{1}{20}\lgt(g) \ge 60 \theta + 2B.
\]
Denote by $\lgt$ the translation length $\lgt(g)$.
Let $h$ be a non trivial element of the normal subgroup $\lld g \rrd$. Our goal is to prove $\lgt(h) \ge \lgt$, with equality if and only if $h$ is conjugate to $g$.  \\

Pick a base point $x_0$ such that $\dist(x_0,h(x_0)) \le \lgt(h) + \theta$. Lemma \ref{lem:presentation} applied to $g$
implies the existence of an admissible presentation with respect to the base point $x_0$:
\[
h=h_m \circ \cdots \circ h_1, \quad h_i=s_i g^{\pm 1} s_i^{-1}.
\]
We assume that $m$ is minimal among all such choices of base points and admissible
presentations. Lemma \ref{lem:mini} implies that $h_j$ is different from $h_i^{-1}$ for all $1\leq i<j\leq m$.

Let $(a_i),$ $(b_i),$ $(x_i)$ be the three sequences of points defined in  \S \ref{par:presentation}; they  satisfy the 
properties (1) to (4)  listed in Lemma \ref{lem:properties} and Remark \ref{rem:property4}. Since the constructions below are more natural
with segments than pairs of points, and $\GH$ is not assumed to be uniquely geodesic, we choose geodesic segments between
the points $x_i$, as well as geodesic segments $[a_i,b_i]\subset \axe(h_i)$. 

We now introduce the following definition in order to state the key Lemma~\ref{lem:infernal}.  
A sequence of points $(c_{-1}$, $c_0$, $\cdots$, $c_k, c_{k+1})$ in $ \GH$, with some choice of segments $[c_i,c_{i+1}]$, $-1\leq i\leq k$, is a \textbf{configuration of order $k \ge 1$} for the segment $[x_0,x_j]$ if
\begin{enumerate}[(i)]
\item \label{point:i} $x_0 = c_{-1}$ and $x_j = c_{k+1}$;
\item \label{point:ii} For all $0 \leq i \leq k$, and all $a \in [c_i,c_{i+1}]$, we have $( c_{i-1}\vert a )_{c_i} \leq 3\theta$.  In particular $( c_{i-1}\vert  c_{i+1} )_{c_i} \leq 3\theta$;
\item \label{point:iii} For all $0 \leq i \leq k+1$ we have $\dist(c_{i-1}, c_{i}) \ge 10 \theta$;
\item \label{point:iv} For all $0 \leq i \leq k$ the segment $[c_i,c_{i+1}]$ is either neutral or a piece, with the following rules:
       \begin{enumerate}[({iv}-a)]
        \item There are never two consecutive neutral segments;
        \item The last segment $[c_k,c_{k+1}]= [c_k, x_j]$ is neutral;
        \item The second segment $[c_0,c_1]$ is a piece of size $18/20$ if  $[c_1, c_2]$ is neutral (this is always the case when $k=1$), and of size $17/20$ otherwise;
        \item For the other pieces $[c_{i-1}, c_i],$ with $i>1,$ the size is $5/20$ when $[c_{i}, c_{i+1}]$ is neutral and  $4/20$ otherwise;
       \end{enumerate}
\item  \label{point:v} For all $0 \leq i \leq k,$ if $[c_i,c_{i+1}]$ is a piece, then there is an index $l$ with $1\leq l\leq j,$ such that $h_l$ is the support of the piece $[c_i,c_{i+1}]$.
\end{enumerate}
Note that properties (\ref{point:iv}) and (\ref{point:v}) do not concern the initial segment $[x_0,c_0],$ and that the size
$p/q$ of a piece $[c_i, c_{i+1}],$ $0\leq i\leq k-1,$ is equal to  $18/20,$ $17/20,$ $5/20,$ or  $4/20$; moreover, this 
size is computed with respect to $\lgt=\lgt(g)$, and thus the minimum length of a piece $[c_i,c_{i+1}]$ is
bounded from below by $4(60 \theta + 2B)$.

The following lemma is inspired by Lemma 2.4 in \cite{Delzant:1996}.
As mentionned there, this can be seen as a version of the famous Greendlinger Lemma in classical small cancellation theory. 
Recall that this lemma claims the existence of a region with a large external segment in a van Kampen diagram (see \cite{Lyndon-Schupp:Book}, Chapter V).  The segment $[c_0,c_1]$ in the previous definition plays a similar role. 
One can consult \cite{Lamy:HDR} for a simpler proof of Lemma \ref{lem:infernal} and Theorem \ref{thm:criterion} in the case of a group acting on a tree, and for small cancellation theory in the context of $\Aut[{\mathbf{k}}^2]$ and its  amalgamated
product structure.

\begin{lem}
\label{lem:infernal}
For each $j = 1,\cdots,m$, there exists $k \ge 1$ such that the segment $[x_0,x_j]$ admits a configuration of order $k$.
Moreover if $j \ge 2$ and $k = 1$ then the first segment $[x_0,c_0]$ of the configuration has length at least $\frac{3}{20}\lgt$.
\end{lem}

\begin{proof}
The proof is by induction on $j$, and uses  the four properties that are listed in Lemma \ref{lem:properties} and Remark \ref{rem:property4}; we refer to them as properties (\ref{point:1}), (\ref{point:2}), (\ref{point:3}), and (\ref{point:4}). Note that
(\ref{point:2}) and (\ref{point:3}) enable us to apply Lemma \ref{lem:canoeing}; similarly, properties
(\ref{point:ii}) and (\ref{point:iii}) for a configuration of order $k$ show that Lemma \ref{lem:canoeing} can be applied
to the sequence of points in such a configuration.\\

{\sl{Initialization.-- }}  When $j = 1,$ we take $c_0 = a_1, c_1 = b_1$ and get a configuration of order 1.
Indeed, by property (2), we have
\[(x_0\vert b_1 )_{a_1} \le 2 \theta \mbox{ and }
( a_1 \vert x_1 )_{b_1} \le 2 \theta,
\]
and, by property (\ref{point:1}), $[c_0,c_1]$ is a piece of size 19/20 (it is a subsegment of $\axe(h_1)$). The segments $[x_0,c_0]$ and $[x_1,c_1]$ are neutral (property (\ref{point:4})) and of length at least $10 \theta$ (property (\ref{point:3})).
\\

Suppose now that $[x_0,x_{j}]$ admits a configuration of order $k$. We want to find a configuration of order $k'$ for $[x_0,x_{j+1}]$. As we shall see, the proof provides
a configuration of order $k'=1$ in one case, and of order $k'\leq k+ 2$ in the other case. \\

{\sl{Six preliminary facts.-- }} Consider the approximation tree $(\Phi, T)$ of  $(x_j$, $x_0$, $x_{j+1})$. We choose $p \in [x_j,x_0]$ such that $\Phi(p) \in T$ is the branch point of the tripod $T$ (with $p = x_j$ if $T$ is degenerate). 
Let $a$ (resp. $b$) be a projection of $a_{j+1}$ (resp. $b_{j+1}$) on the segment $[x_j,x_{j+1}]$. 
By assertion (3) in Lemma \ref{lem:canoeing} we have $\dist(a,a_{j+1}) \le 5\theta$ and $\dist(b,b_{j+1}) \le 5\theta$. 
Thus $\dist (a,b) \ge \frac{19L}{20} -10 \theta$, and by Lemma \ref{lem:weak}, $[a,b]$ is a piece with support $h_{j+1}$.

For all $i\leq k+1,$ note $c'_i$ a projection of  $c_i$ on the segment $[x_0,x_j]$. Note
that, by Lemma \ref{lem:canoeing}, we have $\dist(c_i,c'_i) \le 5\theta$, and the points $c'_i,$ $-1\leq i\leq k+1,$ form a monotonic sequence of points from $x_0$ to $x_j$  along the geodesic segment $[x_0,x_j]$.

Let $S_i$ be the interval of $T$ defined by $S_i= [\Phi(a),\Phi(b)] \cap [\Phi(c'_i),\Phi(c'_{i+1})]$. The preimages of $S_i$ by $\Phi$ are two intervals
$I\subset [a,b]$ and $I' \subset [c'_i,c'_{i+1}]$ such that
\[
\Phi(I) = \Phi(I')= S_i.
\]

\begin{fact}
\label{fact:1}
If $[c_i, c_{i+1}]$ is neutral then
$\diam(S_i) \le \frac{12}{20}\lgt$.
\end{fact}
Since $\Phi$ is an isometry along the geodesic segments $[x_j,x_0]$ and $[x_j,x_{j+1}],$
we only have to prove $\diam(I') \le \frac{12}{20}\lgt$. By Lemma \ref{lem:weak} we know that $[c_i,c_{i+1}] \subset \Tub_{7\theta}([c'_i, c'_{i+1}])$. 
Thanks to the triangular inequality, we can choose $J \subset [c_i,c_{i+1}]$ such that $J \subset \Tub_{7\theta}(I')$ and $\diam(J) \ge \diam(I') - 14 \theta$. 
The properties of the approximation tree imply  $I' \subset \Tub_{\theta}(I)$, and $I \subset \Tub_{7\theta}(\axe(h_{j+1}))$. Thus $J \subset \Tub_{15\theta}(\axe(h_{j+1}))$. Applying Lemma \ref{lem:closer} (with $\beta=16\theta$) and shortening $J$ by $32\theta$ ($16\theta$ at each end), we obtain  $\diam(J) \ge \diam(I') - 46 \theta$ and $J \subset \Tub_{2\theta}(\axe(h_{j+1}))$. Thus $J$ is a piece contained in the neutral segment $[c_i, c_{i+1}]$, hence
\[
\tfrac{11}{20}\lgt \ge \diam(J) \ge \diam(I') - 46 \theta, \quad {\text{and therefore}} \quad \tfrac{12}{20}\lgt \ge \diam(I').
\]

\begin{fact}
\label{fact:2}
If $[c_i, c_{i+1}]$ is a piece then
$\diam(S_i) \le B$.
\end{fact}

By property (\ref{point:v}), there exists an index $l,$ with $1\leq l \leq j,$ such that $[c_i, c_{i+1}]  \subset \Tub_{7\theta}(\axe(h_l))$. Since  $I' \subset \Tub_{7\theta}([c_i, c_{i+1}]),$
we get 
\[
I' \subset \Tub_{14\theta}(\axe(h_l)).
\]
On the other hand $I' \subset \Tub_{\theta}(I)$ and $I  \subset \Tub_{7\theta}(\axe(h_{j+1}))$, thus
\[
I' \subset \Tub_{8\theta}(\axe(h_{j+1})).
\]
If $\diam(I')> B,$ the rigidity assumption shows that  $h_{j+1}$ and $h_l$ share the same axis, with opposite orientations; since $g$ is \good, we obtain $h_{j+1} = h_l^{-1}$. By 
Lemma \ref{lem:mini} this contradicts the minimality of the presentation of $h$. This proves that $\diam(I') \le B$.

\begin{fact}
\label{fact:3}
Suppose $[c_i, c_{i+1}]$ is a piece. Then
\begin{eqnarray*}
\diam([\Phi(x_j),\Phi(a)] \cap [\Phi(c'_i),\Phi(c'_{i+1})]) & \leq & \tfrac{12}{20}\lgt.
\end{eqnarray*}
\end{fact}

By property (\ref{point:v}),
there exists an index $l,$ with $1\leq l\leq j,$ such that $h_l$ is the support
of $[c_i,c_{i+1}]$.
Let $K \subset [a,x_j]$ and $K' \subset [c'_i,c'_{i+1}]$ be two intervals such that
\[
\Phi(K) = \Phi(K') = [\Phi(x_j),\Phi(a)] \cap [\Phi(c'_i),\Phi(c'_{i+1})].
\]
We want to prove $\diam(K) \le \frac{12}{20}\lgt$.
Applying the triangular inequality, we can choose $J \subset [a_{j+1},x_j]$ such that $J \subset \Tub_{7\theta}(K)$ and $\diam(J) \ge \diam(K) - 14 \theta$. Now we have $K \subset \Tub_{\theta}(K')$, $K' \subset \Tub_{7\theta}([c_i, c_{i+1}])$ and $[c_i, c_{i+1}] \subset \Tub_{7\theta}(\axe(h_l))$. Thus
$J \subset \Tub_{22\theta}(\axe(h_l))$. Applying Lemma \ref{lem:closer} with $\beta=23\theta,$ we shorten $J$ by $46\theta$ ($23\theta$ on each end) and obtain
\[
J \subset \Tub_{2\theta}(\axe(h_l)) \quad {\text{and }} \quad \diam(J) \ge \diam(K) - 60 \theta.
\]
The admissibility condition implies that  $J \subset [a_{j+1},x_j]$ is neutral, and therefore
$$
\tfrac{11}{20}\lgt \ge \diam(J) \ge \diam(K) - 60 \theta,
\quad {\text{so that}} \quad \tfrac{12}{20}\lgt \ge \diam(K).
$$

\begin{fact}
\label{fact:4}
The segment $[\Phi(b),\Phi(a)]$ is not contained in the segment $[\Phi(c'_0)$, $\Phi(x_j)]$.
\end{fact}

Let us prove this by contradiction, assuming $[\Phi(b),\Phi(a)] \subset [\Phi(c'_0),\Phi(x_j)].$
By Property (iv-a) and Fact \ref{fact:1}, $[\Phi(b),\Phi(a)]$ intersects at least one
segment $[\Phi(c'_i),\Phi(c'_{i+1})]$ for which $[c_i,c_{i+1}]$ is a piece; Fact \ref{fact:2} implies
that  $[\Phi(c'_i),\Phi(c'_{i+1})]$ is not contained in $[\Phi(a),\Phi(b)]$ (it must intersect the
boundary points of $[\Phi(a),\Phi(b)]$). From this follows that  $[\Phi(a),\Phi(b)]$ intersects at most two pieces and one neutral segment. Facts \ref{fact:1} and \ref{fact:2} now give the contradictory inequality $2B + \frac{12}{20}\lgt  \geq \frac{19}{20}\lgt -10 \theta$.\\

\begin{fact}
\label{fact:5}
The segment $[\Phi(c'_{0}), \Phi(c'_1)]$ is not contained in the segment $[\Phi(x_{j+1})$, $\Phi(x_j)]$.
\end{fact}

Since $[a_{j+1},x_j]$ is neutral and $[c_0,c_1]$ is a piece of size $\geq 17/20$, the segment $[\Phi(c'_0),\Phi(c'_1)]$ is not contained
in $[\Phi(a_{j+1}),\Phi(x_j)]$. Assume that $[\Phi(c'_{0}), \Phi(c'_1)] \subset [\Phi(x_{j+1}), \Phi(x_j)]$, apply Fact \ref{fact:4} and then 
Fact \ref{fact:2} and Fact \ref{fact:3}; this gives the contradictory inequality $B + \frac{12}{20}\lgt  \geq \frac{17}{20}\lgt -10\theta$.\\

These last two facts imply $\Phi(b) \in [\Phi(x_{j+1}), \Phi(p)]$ and $\Phi(c'_0) \in [\Phi(x_0), \Phi(p)]$.  
Moreover, with this new property in mind, the proofs of Facts \ref{fact:4} and \ref{fact:5} give:

\begin{fact}
\label{fact:6}
The segment $[\Phi(b), \Phi(p)]$ has length at least $\tfrac{6}{20}\lgt$, and the segment $[\Phi(c'_0), \Phi(p)]$ has length at least $\tfrac{4}{20}\lgt$.
\end{fact}

In particular the tripod $T$ is not degenerate at $\Phi(x_0)$ nor at $\Phi(x_{j+1})$.\\

{\sl{The induction.-- }} We now come back to the proof by induction.
We distinguish two cases, with respect to the position of $\Phi(a)$ relatively to the branch point $\Phi(p)$.\\

\noindent$\bullet\, $\textbf{First case.} $(\Phi(b) \vert \Phi(x_{0}))_{\Phi(a)} \le \frac{1}{20} \lgt - 12\theta$.\\

In other words we assume that either $\Phi(a) \in [\Phi(b),\Phi(p)]$, or $\Phi(a)$ is close to $\Phi(p)$.
Note that this includes the situation where  $\Phi(x_j)$ is a degenerate vertex of $T$. By
Fact \ref{fact:6}, the situation is similar to  Figure \ref{fig:firstcase}.
Since the distance between $a_{j+1}$ (resp. $b_{j+1}$) and $a$ (resp. $b$) is at most $5\theta$, 
the triangular inequality and Corollary \ref{cor:TGromov} imply
\begin{align*}
(b_{j+1} \vert x_0)_{a_{j+1}}
&\le (b \vert x_0)_a + (10\theta + 5 \theta + 5 \theta)/2 \\
&\le (\Phi(b) \vert \Phi(x_0))_{\Phi(a)} + \theta + 10\theta,
\end{align*}
and the assumption made in this first case gives
\begin{align}\label{ineq:1}
(b_{j+1} \vert x_0)_{a_{j+1}}&\le \frac{1}{20} L - \theta.
\end{align}

\begin{figure}[ht]
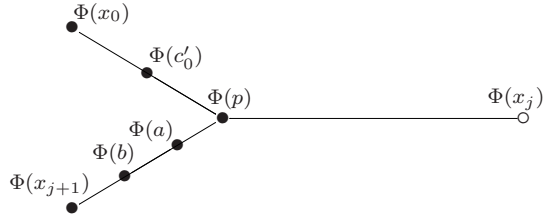

$$\mygraph{
!{<0cm,0cm>;<\sizex,0cm>:<0cm,\sizey>::}
!{(0,-2)}*{\bullet}="xj+1" !{(0,2)}*{\bullet}="x0" !{(1,1)}*{\bullet}="c0"
!{(.7,-1.3)}*{\bullet}="b" !{(1.4,-.6)}*{\bullet}="a"
!{(6,0)}*{\circ}="xj"
!{(2,0)}*{\bullet}="fork"
"xj+1"-^<{\Phi(x_{j+1})}"fork" "x0"-^<{\Phi(x_{0})}"fork"
"b"-^<{\Phi(b)}^>{\Phi(a)}"a" "c0"-^<{\Phi(c'_{0})}"fork"
"fork"-^<{\Phi(p)}^>{\Phi(x_{j})}"xj"
}$$
\caption{The point $\Phi(a)$ could also be on $[\Phi(p),\Phi(x_j)]$, but not far from $\Phi(p)$.}
\label{fig:firstcase}
\end{figure}

At this point we would like to define a configuration of order 1 for $[x_0,x_{j+1}]$ by using $a$ and $b$. The problem is that since the property to be a neutral segment is not stable under perturbation, there is no guarantee that $[b, x_{j+1}]$ is neutral even if $[b_{j+1}, x_{j+1}]$ is.
Another candidates would be $a_{j+1}$ and $b_{j+1}$, but then we would not have the estimate (\ref{point:ii}) in the definition of a configuration of order $k$.

To solve this dilemma we consider another approximation tree  $\Psi \colon X\to T'$, for the list $(a_{j+1}$, $b_{j+1}$, $x_0)$, and
choose a point $q\in [a_{j+1}, b_{j+1}]$ such that $\Psi(q)$ is the branch point of $T'$ (see Figure \ref{fig:firstcasebis}).

\begin{figure}[ht]
$$\mygraph{
!{<0cm,0cm>;<\sizex,0cm>:<0cm,\sizey>::}
!{(-2,0)}*{\bullet}="b" !{(0,1)}*{\circ}="a"
!{(2,0)}*{\bullet}="x0" !{(0,0)}*{\bullet}="fork"
"a"-^<{\Psi(a_{j+1})}"fork" "x0"-^<{\Psi(x_{0})}"fork"
"b"-_<{\Psi(b_{j+1})}_>{\Psi(q)}"fork"
}$$
\caption{}
\label{fig:firstcasebis}
\end{figure}

We then define a new configuration of points ${\hat{c}}_{-1}=x_0$, ${\hat{c}}_0=q$, ${\hat{c}}_1 = b_{j+1}$, ${\hat{c}}_2=x_{j+1}$, and we show that this defines
a configuration of order 1 for $[x_0,x_{j+1}]$.
Using inequality (\ref{ineq:1}), the fact that $\Psi$ is an isometry along $[a_{j+1},b_{j+1}]$, and Corollary \ref{cor:TGromov}, we obtain
\[
\dist(q, b_{j+1}) \ge \tfrac{19}{20}\lgt - (b_{j+1} \vert x_0)_{a_{j+1}} - \theta \ge  \tfrac{18}{20}\lgt.
\]
Thus $[q,b_{j+1}]$ is a piece of size $18/20$ with support $h_{j+1}$:
This gives properties (\ref{point:iv}-c) and (\ref{point:v}). 

By Lemma \ref{lem:obtuse} we have $(q \vert x_{j+1} )_{b_{j+1}} \le 2 \theta$, and by Corollary \ref{cor:TGromov} we have  $(x_0 \vert b_{j+1} )_{p_1} \le \theta$, so we obtain property (\ref{point:ii}); property (\ref{point:iii}) follows from the definition of $b_{j+1}$ (property (\ref{point:3}) in Lemma \ref{lem:properties}).
Thus, $({\hat{c}}_j)_{-1\leq j\leq 2}$ is a configuration of order $1$ for the segment $[x_0,x_{j+1}]$.

Moreover, we have
\begin{align*}
d(x_0,q)
&= (\Psi(a_{j+1}) \vert \Psi(b_{j+1}) )_{\Psi(x_0)} \\
&\ge (a_{j+1} \vert b_{j+1} )_{x_0} -\theta \\
&\ge (a \vert b)_{x_0} - 11\theta \\
&\ge (\Phi(a) \vert \Phi(b))_{\Phi(x_0)} - 12\theta \\
&\ge \dist(\Phi(x_0), \Phi(p))  - 12\theta.
\end{align*}
Fact \ref{fact:6} then gives
$$\dist(x_0,q) \ge \tfrac{4}{20} \lgt - 12 \theta \ge \tfrac{3}{20} \lgt$$
so we obtain the second assertion in Lemma \ref{lem:infernal}.\\

\noindent$\bullet\,$\textbf{Second case.} $(\Phi(b) \vert \Phi(x_{0}))_{\Phi(a)} > \frac{1}{20} \lgt - 12\theta$.\\

Let $i$ be the index  such that $\Phi(p) \in [\Phi(c'_i),\Phi(c'_{i+1})]$. This index is uniquely defined if we impose $\Phi(c'_i) \neq \Phi(p)$. 
The assumption implies   $\Phi(a)\in [\Phi(p),\Phi(x_j)]$ and  $\dist(\Phi(p),\Phi(a)) > \frac{1}{20}\lgt -12\theta$. 
By Fact \ref{fact:6} again,  the situation is similar to Figure \ref{fig:secondcase}. 
We distinguish two subcases according to the nature of $[c_i,c_{i+1}]$.\\

\begin{figure}[ht]
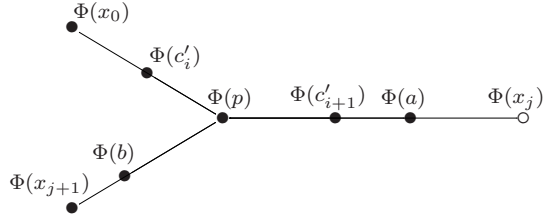

$$\mygraph{
!{<0cm,0cm>;<\sizex,0cm>:<0cm,\sizey>::}
!{(0,-2)}*{\bullet}="xj+1" !{(0,2)}*{\bullet}="x0"
!{(.7,-1.3)}*{\bullet}="b" !{(4.5,0)}*{\bullet}="a"
!{(6,0)}*{\circ}="xj" !{(1,1)}*{\bullet}="ci"
!{(2,0)}*{\bullet}="fork" !{(3.5,0)}*{\bullet}="ci+1"
"xj+1"-^<{\Phi(x_{j+1})}"fork" "x0"-^<{\Phi(x_{0})}"fork" "fork"-^>{\Phi(x_{j})}"xj"
"fork"-^>{\Phi(a)}"a" "ci"-^<{\Phi(c'_i)}"fork"
"b"-^<{\Phi(b)}"fork" "fork"-^<{\Phi(p)}^>{\Phi(c'_{i+1})}"ci+1"
}$$
\caption{On this picture, $\Phi(c'_{i+1})$ is in between $\Phi(p)$ and $\Phi(a)$; another
possible case would be $\Phi(a)\in [\Phi(p),\Phi(c'_{i+1})]$.}
\label{fig:secondcase}
\end{figure}

\noindent\textbf{Second case - first subcase.} Assume $[c_i,c_{i+1}]$ is neutral. Then  if $i < k $ the segment $[c_{i+1},c_{i+2}]$ is a piece (if $i = k$ the following discussion is even easier).
If  $\dist(\Phi(p),\Phi(c'_{i+1})) \le 34\theta$ then Fact \ref{fact:2} implies $\dist(\Phi(p),\Phi(a)) \le B + 34\theta \le \frac{1}{20}\lgt -12\theta$.
This contradicts the assumption of the second case; as a consequence,
\begin{equation} \label{eq:34theta}
\min\{ \dist(\Phi(p), \Phi(c'_{i+1})), 
\dist(\Phi(p), \Phi(a) )\}  > 34\theta.
\end{equation}
Consider now the approximation tree $\Psi \colon X\to T'$
of $(c_{i+1}$, $c_{i}$, $a_{j+1}$, $b_{j+1})$. We have
\begin{align*}
(\Psi(c_{i+1}) \vert \Psi(a_{j+1}))_{\Psi(b_{j+1})}
&\ge (c'_{i+1} \vert a)_{b} - 15 \theta - \theta \\
&\ge (\Phi(c'_{i+1}) \vert \Phi(a) )_{\Phi(b)} -17 \theta \\
&\ge  \dist(\Phi(b), \Phi(p))- 17\theta \\
&\text{ } + \min\{ \dist(\Phi(p), \Phi(c'_{i+1})), \dist(\Phi(p), \Phi(a) )\} 
\end{align*}
and
\begin{align*}
(\Psi(c_{i}) \vert \Psi(a_{j+1}))_{\Psi(b_{j+1})}
&\le (c'_{i} \vert a)_{b} + 15 \theta +\theta \\
&\le (\Phi(c'_{i}) \vert \Phi(a) )_{\Phi(b)} +17 \theta \\
&\le \dist(\Phi(b), \Phi(p)) +17\theta.
\end{align*}
By (\ref{eq:34theta}) we get
$$(\Psi(c_{i+1}) \vert \Psi(a_{j+1}))_{\Psi(b_{j+1})} >  (\Psi(c_{i}) \vert \Psi(a_{j+1}))_{\Psi(b_{j+1})}.$$
Thus we obtain the  pattern depicted on Figure \ref{fig:neutralcase}, where $q$ is a point of $[c_i,c_{i+1}]$
which is mapped to the branch point of $T'$.
Note that $[c_i, q] \subset [c_i,c_{i+1}]$ is again a neutral segment. \\

\begin{figure}[ht]
$$\mygraph{
!{<0cm,0cm>;<\sizex,0cm>:<0cm,\sizey>::}
!{(0,0)}*{\bullet}="bj+1" !{(1,2)}*{\bullet}="ci" !{(4,2)}*{\circ}="ci+1" !{(3,0)}="fork"
!{(6,0)}*{\bullet}="aj+1"
!{(2,0)}*{\bullet}="p1"
"bj+1"-^<{\Psi(b_{j+1})}"p1" "ci"-^<{\Psi(c_i)}"p1" "p1"-_<{\Psi(q)}^>{\Psi(a_{j+1})}"aj+1"
"ci+1"-^<{\Psi(c_{i+1})}"fork"
}$$
\caption{}
\label{fig:neutralcase}
\end{figure}

Since the point $q$ is $4\theta$-close to $\axe(h_{j+1})$, 
the segment $[q,b_{j+1}]$ is a piece with support $h_{j+1}$.
We have
\begin{align*}
\dist(b_{j+1},q)
&\ge \dist(\Psi(b_{j+1}),\Psi(q))   \\
&\ge (c_i \vert a_{j+1})_{b_{j+1}} - \theta \\
&\ge (c'_i\vert a)_{b} - 15 \theta - \theta \\
&\ge \dist(\Phi(b), \Phi(p)) -17 \theta.
\end{align*}
Thus by Fact \ref{fact:6} we see that $[q,b_{j+1}]$ is a piece of size $5/20$, and the same conclusion holds even if one moves $q$ along the segment $[q,b_{j+1}]$ 
by at most $12\theta$. The inequalities $(x_{j+1} \vert a_{j+1})_{b_{j+1}} \le 2\theta$ and $(a_{j+1} \vert b_{j+1})_{q} \le \theta$   imply that $(x_{j+1}\vert q)_{b_{j+1}} \le 3\theta$. 

In this first subcase, we define the new configuration $(\hat{c}_{l})$ by  $\hat{c_l}=c_l$  for $l$ between $-1$ and $i$, and by 
$\hat{c}_{i+1} = q$, $\hat{c}_{i+2} = b_{j+1}$ and $\hat{c}_{i+3}=x_{j+1}$; this defines a configuration of order $i+2$ for the segment $[x_0,x_{j+1}]$.

By construction, properties (\ref{point:i}), (\ref{point:ii}), and (\ref{point:v}) are satisfied, and property (\ref{point:iii})  is obtained after translating $q$ along $[q,b_{j+1}]$  
on a distance less than $12\theta$. This does not change the fact that $[q,b_{j+1}]$ is a piece of size $5/20$. Since the new configuration is obtained from the previous one by cutting it after $c_i$ and then adding a piece $[q,b_{j+1}]$ and a neutral segment $[b_{j+1},x_{j+1}]$, property (\ref{point:iv}) is also satisfied.  \\

\noindent\textbf{Second case - second subcase.} 
Assume $[c_i,c_{i+1}]$ is a piece. By Fact \ref{fact:2} the segment $[\Phi(p),\Phi(c_{i+1})]$ has length at most $B$, and $[c_{i+1},c_{i+2}]$ must be neutral: otherwise  $d(\Phi(p),\Phi(a)) \le 2B$, in contradiction with the case assumption $\dist(\Phi(p),\Phi(a)) > \frac{1}{20}\lgt -12\theta$. Thus the piece $[c_i,c_{i+1}]$ has size $5/20$ (or $18/20$ if $i = 0$). We consider the approximation tree $\Psi \colon X\to T$
of $(c_{i+1}$, $c_{i}$, $c_{i+2}$, $a_{j+1}$, $b_{j+1})$. We obtain one of the situations depicted on Fig. \ref{fig:piececase}, where now $q \in [c_{i+1},c_i]$ or $[c_{i+1},c_{i+2}]$. \\

In case (a), $[c_{i+1},q]$ is small, and therefore neutral. We are in a case similar to the first subcase and to Figure \ref{fig:neutralcase}. We still have 
\begin{align*}
\dist(b_{j+1},q)&\geq \dist(\Phi(b), \Phi(p)) - 17 \theta,
\end{align*}
and one can define a new sequence $(\hat{c}_{l})$ by  $\hat{c_l}=c_l$  for $l$ between $-1$ and $i+1$, and by 
$\hat{c}_{i+2} = q$, $\hat{c}_{i+3} = b_{j+1}$ and $\hat{c}_{i+4}=x_{j+1}$.  This new sequence cuts $(c_l)$ after $l=i+1$,
adds a neutral segment after the last piece $[c_i,c_{i+1}]$, and then a piece $[q,b_{j+1}]$ and a final neutral segment 
$[b_{j+1}, x_{j+1}]$. Hence, $(\hat{c}_{l})$ is a configuration of order $i+3$ for the segment $[x_0,x_{j+1}]$.

\begin{figure}[ht]
$$
\begin{array}{c}
\mygraph{
!{<0cm,0cm>;<\sizex,0cm>:<0cm,\sizey>::}
!{(0,0)}*{\bullet}="bj+1" !{(1,2)}*{\bullet}="ci" !{(4,1)}*{\bullet}="ci+2"
!{(3.5,0)}="fork"
!{(4,0)}*{\bullet}="aj+1"
!{(2,0)}*{\bullet}="p1"
!{(1.75,0.5)}="fork'" !{(2.5,1)}*{\circ}="ci+1"
"bj+1"-^<{\Psi(b_{j+1})}"p1" "ci"-^<{\Psi(c_i)}"p1" "p1"-_<{\Psi(q)}_>{\Psi(a_{j+1})}"aj+1"
"ci+2"-^<{\Psi(c_{i+2})}"fork"
"ci+1"-_<{\Psi(c_{i+1})}"fork'"
}\\
\\ (a) \end{array}
\quad
\begin{array}{c}
\mygraph{
!{<0cm,0cm>;<\sizex,0cm>:<0cm,\sizey>::}
!{(0,0)}*{\bullet}="bj+1" !{(1,2)}*{\bullet}="ci" !{(4,1)}*{\bullet}="ci+2" !{(3.5,0)}="fork"
!{(4,0)}*{\bullet}="aj+1"
!{(2,0)}*{\bullet}="p1"
!{(2.5,0)}="fork'" !{(2.5,1)}*{\circ}="ci+1"
"bj+1"-^<{\Psi(b_{j+1})}"p1" "ci"-^<{\Psi(c_i)}"p1" "p1"-_<{\Psi(q)}_>{\Psi(a_{j+1})}"aj+1"
"ci+2"-^<{\Psi(c_{i+2})}"fork"
"ci+1"-^<{\Psi(c_{i+1})}"fork'"
}\\
\\ (b) \end{array}
$$
\caption{}
\label{fig:piececase}
\end{figure}

In case (b), let us check that $\dist(c_{i+1},q) \le B + 15\theta$, so that $[c_i,q]$ is a piece of size $4/20$ (or $17/20$ if $i = 0$). We have
\begin{align*}
\dist(c_{i+1},q)
&\le (c_i \vert b_{j+1})_{c_{i+1}} +2\theta\\
&\le (\Phi(c'_i) \vert \Phi(b))_{\Phi(c'_{i+1})} + 15\theta\\
&\le B + 15\theta.
\end{align*}

As in the first subcase, the estimate for $\dist(b_{j+1},q)$ gives
\begin{align*}
\dist(b_{j+1},q)&\geq \dist(\Phi(b), \Phi(p)) - 17 \theta,
\end{align*}
and Fact \ref{fact:6} implies that $[q,b_{j+1}]$ is a piece of size $5/20$; moreover $(x_{j+1}\vert p_1)_{b_{j+1}}$ $\le 3\theta$. 

We define the new configuration $(\hat{c}_{l})$ by taking $\hat{c_l}=c_l$  for $l$ between $-1$ and $i$, and 
$\hat{c}_{i+1} = q$, $\hat{c}_{i+2} = b_{j+1}$ and $\hat{c}_{i+3}=x_{j+1}$; this defines a configuration of order $i+2$ for the segment $[x_0,x_{j+1}]$.

By construction, properties (\ref{point:i}), (\ref{point:ii}), (\ref{point:iii}), (\ref{point:iv}) and (\ref{point:v}) are satisfied. 
\end{proof}

\begin{proof}[Proof of Theorem~\ref{thm:criterion}]
By Lemma \ref{lem:infernal} there exists $(c_i)$ a configuration of order $k$ for $[x_0,x_{m}]$, where $x_m = h(x_0)$. Recall that by our choice of $x_0$ we have $\lgt(h) \ge \dist(x_0,x_m) - \theta$.

If $k \ge 2$, then we have at least two distinct pieces: The first one, $[c_0,c_1]$, has size at least 17/20, and the last one, $[c_{k-1},c_k]$, has size 5/20. By Lemma \ref{lem:canoeing} we obtain $\lgt(h) \ge \frac{22}{20} \lgt - 11\theta > \lgt$.

If $k = 1$, either $m = 1$ and $h$ is conjugate to $g$, or we have $\dist(x_0,c_0) \ge \frac{3}{20}\lgt$. On the other hand $\dist(c_0,c_1)\ge \frac{18}{20}\lgt$, so we obtain $\lgt(h) \ge \frac{21}{20}\lgt - 11\theta > \lgt$.
\end{proof}

\section{Hyperbolic spaces with constant negative curvature}
\label{sec:NegCurvature}

This section is devoted to the classical hyperbolic space $\Hyp^n$, where the  dimension $n$ is allowed to be infinite. 
As we shall see, constant negative curvature, which is stronger than $\delta$-hyperbolicity, is a useful  property to decide whether the axis $\axe(g)$ of a hyperbolic isometry is rigid. 

\subsection{Hyperbolic spaces}\label{par:hypspaces}

\subsubsection{Definition}
Let $\Hilb$ be a real Hilbert space, with scalar product $(x \cdot y)_\Hilb$, and let  $\lVert x \rVert_{\Hilb}$ denotes  the norm of any element $x\in \Hilb$. 
Let $u$ be a unit vector in $\Hilb$, and $u^\perp$ be its orthogonal complement; each element $x\in \Hilb$ decomposes uniquely into $x=\alpha(x) u + v_x$ with
$\alpha(x)$ in $\R$ and $v_x$ in $u^\perp$.
Let $\langle \, . \, \vert \, .\, \rangle \colon \Hilb\times \Hilb\to \R$ be the symmetric bilinear mapping  defined by
\[
\langle x \vert y \rangle = \alpha(x)\alpha(y) - (v_x \cdot v_y)_\Hilb.
\]
This bilinear mapping is continuous and has signature equal to $(1,\dim(\Hilb)-1)$. 
The set of points $x$ with $\langle x \vert x \rangle=0$ is the light cone of $\langle \, . \, \vert \, .\, \rangle$.
Let $\Hyp$ be the subset of $\Hilb$ defined by
\[
\Hyp = \{ x \in \Hilb \,; \, \langle x \vert x \rangle = 1 \, {\text{ and }} \, \langle u \vert x \rangle >0\}.
\]
The space $\Hyp$ is the sheet of the hyperboloid $\langle x \vert x\rangle =1$ which contains $u$.

The function $\hd \colon \Hyp \times \Hyp \to \R_+$ defined by
\[
\cosh(\hd(x,y))=\langle x \vert y\rangle
\]
gives a distance on $\Hyp,$ and $(\Hyp,\hd)$ is a complete and simply connected riemannian manifold of dimension $\dim(H)-1$
with constant scalar curvature $-1$ (this characterizes $\Hyp$ if the dimension is finite). 
As such, $\Hyp$ is  a {\sc{cat}}(-1) space, and therefore a 
$\delta$-hyperbolic space. 
More precisely, $\Hyp$ is $\delta$-hyperbolic, in any dimension, even infinite, with
$\delta = \log(3)$ (see \cite{CDP:Book}, \S~I.4 page 11).  In particular, all properties
listed in \S \ref{par:basicdef} are satisfied in $\Hyp$.

\subsubsection{Geodesics and boundary}
The hyperbolic space $\Hyp$ is in one to one correspondence with its projection into the projective space $\P(\Hilb)$.
The boundary of this subset of   $\P(\Hilb)$ is the projection of the light cone of $\langle \, . \, \vert \, .\, \rangle$; we shall
denote it by $\partial \Hyp$ and call it the {\bf{boundary}} of $\Hyp$ (or boundary at infinity).

Let $\Gamma$ be a geodesic line in $\Hyp$. Then there is a unique plane $V_\Gamma\subset\Hilb$ such that $\Gamma= V_\Gamma\cap \Hyp$. 
The plane $V_\Gamma$ intersects the light cone on two lines, and these lines determine two points of $\partial \Hyp,$
called the endpoints of $\Gamma$. If $x$ and $y$ are two distinct points of $\Hyp,$ there is a unique geodesic segment $[x,y]$ from $x$ to $y$; this segment
is contained in a unique geodesic line, namely $\Vect(x,y)\cap \Hyp$.

Let $x$ be a point of $\Hyp$ and $\Gamma$ be a geodesic. The projection $\pi_\Gamma(x)\in \Gamma$, i.e. the point $y\in \Gamma$ which minimizes $\hd(x,y)$, is unique. 

\subsubsection{Distances between geodesics}
The following lemma shows that close geodesic segments are indeed exponentially close in a neighborhood of
their middle points; this property is not satisfied in general $\delta$-hyperbolic spaces because $\delta$-hyperbolicity
does not tell anything at scale $\leq \delta$. 
\begin{lem}
\label{lem:serge}
Let $x,$ $x',$ $y,$ and $y'$ be four points of the hyperbolic space $\Hyp$.
Assume that 
\begin{enumerate}
\item[(i)] $\hd(x,x')\leq \eps,$ and $\hd(y,y')\leq \eps$,
\item[(ii)] $\hd(x,y)\geq B$, 
\item[(iii)] $B\geq 10 \eps$,
\end{enumerate}
where $\eps$ and $B$ are positive real numbers. Let $[u,v]\subset [x,y]$ 
be the geodesic segment of length $\frac{3}{4} \hd(x,y)$ centered at the middle 
of $[x,y]$. Then $[u,v]$ is contained in the tubular neighborhood $\Tub_\nu([x',y'])$ with 
\[
\cosh(\nu)-1\leq  5 \frac{\cosh(2\eps)-1}{\exp(B/2-4\eps)}.
\]
\end{lem}

Thus, when $B$ is large, $\nu$ becomes much smaller than $\eps$.
The constants in this inequality are not optimal and depend on the choice
of the ratio 
$$\hd(u,v)/\hd(x,y)=3/4.$$ 

\begin{proof}
We use coordinates $(x_0, x_1, ..., x_n ,...)$ in $H$ with 
\[
\langle x \vert x\rangle = x_0^2-\sum_{i\geq 1} x_i^2.
\]
We denote by $\vert x\vert $ the square root of the absolute 
value of $\langle x \vert x\rangle$, and denote by 
$\Vert x \Vert_{euc}$ the square root of $\sum_{i\geq 0} x_i^2$.
We have $\vert \langle u \vert v \rangle \vert \leq \Vert u \Vert_{euc} \Vert v \Vert_{euc}$.

If $x$ and $y$ are two elements of $\Hyp$ the segment $[x,y]$
is parametrized by $t\mapsto u_t/\vert u_t\vert$ where $u_t=tx + (1-t)y$.

Let $x$, $x'$, $y$, $y'$ satisfy $\hd(x,x')\leq \eps$, $\hd(y,y')\leq \eps$, and $\hd(x,y)=B$.
We first apply an isometry of $\Hyp$ to assume that $x=(1,0, ...0, ..)$ and 
$y=(\cosh(B), \sinh(B), 0, ...)$. 

The sphere of $\Hyp$ centered at $x$ with radius $B$ is the set 
$\{ y+z \in \Hyp \, \vert \, \langle x \vert z\rangle =0 \}$.

\vspace{0.1cm}

{\sl{First Step}.---} Let us first assume that $\hd(x,y')=B=\hd(y,x')$, and write $x'=x+r$ and $y'=y+s$.
We have 
\begin{itemize}
\item $\langle x'\vert x'\rangle = 1$, i.e. $2\langle r \vert x \rangle + \langle r \vert r\rangle =0$;
\item $\langle x \vert x+r\rangle \leq \cosh(\eps)$;
\item $\langle x+r\vert y\rangle = \cosh(B)$.
\end{itemize}
Thus, writing $r=(r_0, r_1, ...)$ we get 
\[
0\leq r_0\leq \cosh(\eps)-1, \text{ and } \, 2r_0+r_0^2= \sum_{i\geq 1} r_1^2.
\]
In particular, 
\[
\Vert r \Vert^2_{euc}= 2r_0+2r_0^2 \leq 2 \cosh(\eps)(\cosh(\eps)-1)
\]
A similar analysis for $s$ leads to the following estimates
\begin{itemize}
\item $s_0=0$;
\item $-2\sinh(B)s_1=\sum_{i\geq 1} s_i^2$;
\item $\Vert s\Vert^2_{euc}=\sum_{i\geq 1} s_i^2\leq 2(\cosh(\eps)-1)$.
\end{itemize}

Let us now parametrize the geodesic segments between $x$ and $y$ and
between $x'$ and $y'$. The first one is parametrized by 
\[
m_t = u_t/\vert u_t\vert, {\text{ where }} u_t = tx + (1-t) y.
\]
The second one by $m'_t=v_t/\vert v_t\vert $ where 
\[
v_t = tx'+(1-t)y'= u_t + tr+(1-t)s = u_t + w_t
\]
where $w_t=tr+(1-t)s$.
Thus, 
\[
\langle m_t\vert m'_t\rangle = \frac{\vert u_t\vert}{ \vert v_t\vert } + \frac{1}{ \vert v_t\vert} \langle m_t\vert t r + (1-t)s\rangle.
\]
We now restrict this computation to  $t\in [t_0, 1-t_0]$ where $t_0$ is a fixed positive number
(we shall need $t_0= \exp(-B/2)$). We have 
\begin{eqnarray*}
\vert u_t\vert^2 = t^2  + (1-t)^2 + 2 t (1-t)  \langle x \vert y\rangle  & \geq & 1/2 + 2 t_0 ( 1-t_0) \cosh(B), \\
\vert v_t\vert^2 & \geq & 1/2 + 2 t_0(1-t_0) \cosh(B-2\eps) \\
\langle u_t \vert w_t\rangle \leq \frac{1}{2} (\cosh(\eps)-1).
\end{eqnarray*}
Expanding $\vert v_t\vert^2$ and using that $2 \langle u_t \vert w_t\rangle + \langle w_t\vert w_t\rangle$ is equal 
to $2 t(1-t) \langle r\vert s\rangle$
\begin{eqnarray*}
\left\vert  \frac{\vert u_t\vert}{\vert v_t\vert}-1\right\vert & \leq &  2 t_0 (1-t_0) \frac{\langle r \vert s\rangle}{\vert v_t\vert^2}
\end{eqnarray*}
All together, one easily gets the estimate
\begin{eqnarray*}
\left\vert\langle m_t \vert m'_t\rangle -1\right\vert& \leq & \left( 2 \sqrt{2} t_0 (1-t_0) \sqrt{\cosh(\eps)+1} + \frac{1}{2} \right)\frac{\cosh(\eps)-1}{\frac{1}{2} + 2 t_0(1-t_0) \cosh(B-2\eps)}.
\end{eqnarray*}
Let us now choose $t_0=\exp(-B/2)$ ; with such a choice, the geodesic segments described by $m_t$ (resp. $m'_t$) 
contains a segment of length $3B/4$ centered at the middle of $[x,y]$ (resp. $[x',y']$). Since $B\geq 10 \eps$
we have 
\begin{eqnarray*}
\left\vert\langle m_t \vert m'_t\rangle -1\right\vert& \leq &  \left( 4 \exp(-B/2+\eps/2) + 1/2) \right)\frac{\cosh(\eps)-1}{   \cosh(B/2-2\eps)}\\
& \leq & 5 \frac{\cosh(\eps)-1}{   \cosh(B/2-2\eps)}
\end{eqnarray*}
This upper bound show that the point $m'_t$ is at distance at most $\nu$ from $[x,y]$ with 
\[
\cosh(\nu)-1\leq 5 \frac{\cosh(\eps)-1}{   \cosh(B/2-2\eps)}
\]

\vspace{0.1cm}

{\sl{Second Step}.---} When $x'$ is not on the sphere of radius $B$ centered at $y$, we replace
$x'$ by the intersection of the geodesic containing $[x',y']$ with the sphere; since $B\geq 10\eps$,
the distance between $x$ and the new point $x'$ is at most $2\eps$. We do the same for $y'$
and the sphere of radius $B$ centered at $x$, and then apply the first step to conclude. 
\end{proof}

\subsubsection{Isometries}\label{par:isohypspace}
Let $f$ be an isometry of $\Hyp$; then $f$ is the restriction of a unique continuous linear transformation of the Hilbert space $\Hilb$. 
In particular, $f$ extends to a homeomorphism of the boundary $\partial \Hyp$. 
The three types of isometries (see \S~\ref{par:isometriesGromov}) have the following properties.
\begin{enumerate}
\item If $f$ is elliptic there is a point $x$ in $\Hyp$ with $f(x)=x$, and $f$ acts as a rotation
centered at $x$. Fixed points of $f$ are eigenvectors of the linear extension $f\colon H\to H$ corresponding to the eigenvalue $1$.
\item If $f$ is parabolic there is a unique fixed point of $f$ on the boundary $\partial \Hyp$; this point corresponds to a line of eigenvectors 
with eigenvalue $1$ for the linear extension $f\colon H\to H$. 
\item If $f$ is hyperbolic, then $f$ has exactly two fixed points $\alpha(f)$ and $\omega(f)$ on the boundary  $\partial \Hyp$ and the orbit $f^k(x)$
of every point $x\in \Hyp$ goes to $\alpha(f)$ as $k$ goes to $-\infty$ and $\omega(f)$ as $k$ goes to $+\infty$. The set $\Min(f)$
coincides with the geodesic line from $\alpha(f)$ to $\omega(f)$. In particular, $\Min(f)$ coincides with the unique $f$-invariant
axis $\axe(f)=\Vect(\alpha(f),\omega(f))\cap \Hyp$.
The points $\alpha(f)$ and $\omega(f)$ correspond to eigenlines of the linear extension $f\colon H\to H$ with eigenvalues $\lambda^{-1}$
and $\lambda$, where $\lambda >1$; the translation length $\lgt(f)$ is equal to the logarithm of $\lambda$ (see Remark \ref{rem:trans_length}
below).
\end{enumerate}

\subsection{Rigidity of axis for hyperbolic spaces}


Let $G$ be a group of isometries of the hyperbolic space $\Hyp,$
and $g$ be a hyperbolic element of $G$.
The following strong version of Lemma \ref{lem:rigidGH}  
is a direct consequence of Lemma \ref{lem:serge}.

\begin{lem} \label{lem:rigidPM}
Let $\eps' < \eps$ be positive real numbers. If the segment $A \subset \Hyp$ is $(\eps', B')$-rigid, then $A$ is also $(\eps, B)$-rigid
with 
\[
B= \max \left\lbrace 
22\eps, 18\eps+2\log  \left(5 \tfrac{\cosh(4\eps)-1}{\cosh(\eps')-1} \right) , 
 \tfrac{4}{3}B'+2\eps \right\rbrace.
\]
\end{lem}

In the context of  hyperbolic spaces $\Hyp$, Lemma \ref{lem:rigidPM} enables us to drop the $\eps$ in the notation 
for $\eps$-rigidity; we simply say that $A \subset \Hyp$ is rigid.

\begin{proof}
By assumption, $B$ satisfies the following three properties
\begin{itemize}
\item $B-2\eps \geq 10 (2\eps)= 20 \eps$;
\item $\cosh(\eps')-1\geq  {(\cosh(4\eps)-1)}/{\exp(B/2-9\eps)}$;
\item ${(3/4)(B-2\eps)}\geq B'$.
\end{itemize}
Let $f$ be an element of the group $G$ such that
$A\cap_\eps f(A)$ contains  two points at distance $B$. Then $A$ and $f(A)$ contain two
segments $I$ and $J$ of length $B-2\eps$ which are $2\eps$-close. Lemma \ref{lem:serge} and the
inequalities satisfied by $(B-2\eps)$
show that $I$ and $J$ contain subsegments of length $3(B-2\eps)/4$ which are $\eps '$-close. Since $A$ is
$(\eps',B')$-rigid, $f(A)$ coincides with $A$, and therefore $A$ is $(\eps, B)$-rigid.
\end{proof}

\begin{pro} \label{pro:almostfixed}
Let $G$ be a group of isometries of $\Hyp$, and $g\in G$ be a hyperbolic isometry.
Let $n$ be a positive integer, $p \in \axe(g)$, and $\eta >0$. 
If $\axe(g)$ is not rigid, there exists
an element $f$ of $G$  such that $f(\axe(g)) \neq \axe(g)$ and $\hd(f(x),x) \le \eta$ for all $x \in [g^{-n}(p),g^n(p)]$.
\end{pro}

\begin{proof}
Since $\axe(g)$ is not rigid, in particular $\axe(g)$ is not $(\eps,B)$-rigid for $B = (3n + 4)\lgt(g)$ and $\eps = \eta/2$.
Then there exists an isometry $h\in G$ such that 
$h(\axe(g))$ is different from $\axe(g),$ but $\axe(g)$ contains a segment $J$ of length
$B$ which is mapped by $h$ into the tubular neighborhood $\Tub_\eps(\axe(g))$.

Changing $h$ into $g^j\circ h$, we can assume that the point $p$ is near the middle of
the segment $h(J)$; and changing $h$ into $h\circ g^k$ moves $J$ to $g^{-k}(J)$. 
We can thus change $h$ into $h_1= h^j \circ h\circ g^k$ for some $j,k\in \Z$, $B$ into $B_1\geq B- 4 \lgt(g) \ge 3n \lgt(g)$, and find two points $x$ and $y$ 
on $\axe(g)$ that satisfy 
\begin{enumerate}[(1)]
\item $\hd(x,y)\geq B_1$, $p\in [x,y]$, and $\hd(p,x)\geq B_1/3$, $\hd(p,y)\geq B_1/3$;
\item either (a) $h_1(x)\in \Tub_\eps([g^{-1}(x),x])$ and $h_1(y)\in \Tub_\eps([g^{-1}(y),y])$, 

 \hspace{0.09cm} or (b) $h_1(y)\in \Tub_\eps([g^{-1}(x),x])$ and $h_1(x)\in \Tub_\eps([g^{-1}(y),y])$;
\item $g^i(x)$ and  $g^i(y)$, with $-2\leq i\leq 2$,  are at distance at most $\eps$ from $h(\axe(g))$.
\end{enumerate}
This does not change the axis $h(\axe(g))$ and the value of $\eps$. 

We now change $h_1$ into the commutator $h_2=h_1^{-1}g^{-1}h_1 g$. We still 
have $h_2(\axe(g))\neq \axe(g)$, because otherwise $h_1^{-1}g^{-1}h_1$ fixes
$\axe(g)$ and by  uniqueness of the axis of a hyperbolic
isometry of $\Hyp$ we would have $h_1^{-1}(\axe(g))= \axe(g)$. Moreover, property (2) above is replaced by $\hd(x,h_2(x))\leq 2 \eps < \eta$ and $\hd(y,h_2(y))\leq \eta$
in case (a), and by 
$\hd(g^2(x), h_2(x))\leq \eta$  and $\hd(g^2(y),h_2(y))\leq \eta$ in case (b); similar
properties are then satisfied by the points $g^i(x)$ and $g^i(y)$, $-2\leq i\leq 2$, in place  of $x$ and $y$.

Changing, once again, $h_2$ into $h_3=h_2\circ g^{-2}$ if necessary, we can assume that 
$\hd(x,h_3(x))\leq \eta$ and $\hd(y,h_3(y))\leq \eta$.

Consider the arc length  parametrization  
\[
m \colon t\in [-\infty,+\infty] \mapsto \axe(g)
\]
such that $m(0) = p$ and $g(m(t)) = m(t+\lgt(g))$. 
Since $p$ is in the interval $[x+B_1/3,y-B_1/3],$ we obtain 
$
\hd(m(t), h_3(m(t)))\leq \eta
$ 
for all $t$ in $[-B_1/3,B_1/3]$. Defining $f=h_3$ we get
\[
\hd(z,f(z))\leq \eta, \quad \forall z \in [g^{-n}(p), g^n(p)]
\]
because $n\lgt(g)\leq B_1/3$. 
\end{proof}

In particular, the proof for $n = 1$ gives the following corollary.

\begin{cor}
\label{cor:constants}
Let $G$ be a group of isometries of $\Hyp$. Let $g$ be a hyperbolic element 
of $G$ and $p$ be a point of $\axe(g)$. Let  $\eta$ be a positive real number.
If there is no $f$ in $G\setminus \{ {\rm Id}\}$ such that $\dist(f(x),x) \le \eta$ for all $x \in [g^{-1}(p),g(p)]$, then $\axe(g)$ is $(\eta/2, 7\lgt(g))$-rigid. 
\end{cor}

%
%


\part{Algebraic Geometry and the Cremona Group}

%
%

\section{The Picard-Manin space}
\label{sec:cremona-PM}

%
%

\subsection{N\'eron-Severi groups and rational morphisms}

Let $X$ be a smooth 
projective surface defined over an algebraically
closed field ${\mathbf{k}}$.
The N\'eron-Severi group 
$\NS(X)$ is the group 
of Cartier divisors modulo numerical equivalence.
When the field  
of definition is the field of complex numbers, $\NS(X)$ coincides
with the space of Chern classes of holomorphic line bundles of $X$
(see \cite{Lazarsfeld:Book}), and thus
\[
\NS(X)=H^2(X(\C),\Z)_{t.f.}\cap H^{1,1}(X,\R)
\]
where $H^2(X(\C),\Z)_{t.f.}$ is the torsion free part of $H^2(X(\C),\Z)$
(the torsion part being killed when one takes the image of $H^2(X(\C),\Z)$
into $H^2(X(\C),\R)$).
The rank $\rho(X)$ of this abelian group is called the Picard number of $X$.
If $D$ is a divisor, we denote by $[D]$ its class in $\NS(X)$. 
The intersection form defines an integral quadratic form 
\[
([D_1], [D_2])\mapsto [D_1] \cdot [D_2]
\]
on $\NS(X),$ the
signature of which is equal to $(1,\rho(X)-1)$ by Hodge index theorem
(see \cite{Hartshorne:book}, \S V.1).
We also denote   $\NS(X) \otimes \R$ by $\NS(X)\ten$.

If $\pi\colon X\to Y$ is a birational morphism, then the pullback morphism 
\[
\pi^*\colon \NS(Y)\to\NS(X)
\]
is injective and preserves the intersection form. For example, 
if $\pi$ is just the blow-up of a point with exceptional divisor $E\subset X$,
then $\NS(X)$ is isomorphic to $\pi^*(\NS(Y))\oplus \Z [E],$ this sum is orthogonal
with respect to the intersection form, and $[E]\cdot [E]=-1$.

\begin{eg}
The N\'eron-Severi group of the plane is isomorphic to $\Z[H]$ where $[H]$ 
is the class of a line and $[H]\cdot [H]=1$. After $n$ blow-ups of points,
the N\' eron-Severi group is isomorphic to $\Z^{n+1}$ with a basis of orthogonal 
vectors $[H]$, $[E_1], \cdots$, $[E_n]$ such that $[H]^2=1$ and $[E_i]^2=-1$
for all $1\leq i\leq n$.\end{eg}

\subsection{Indeterminacies}

Let $f\colon X\dashrightarrow Y$ be a rational map between smooth projective surfaces.
The indeterminacy set $\Ind(f)$ is  finite, and the curves which are blown down by $f$ form a codimension 
$1$, analytic subset of $X,$ called the exceptional set $\Exc(f)$. 

Let $H$ be an ample line bundle on $Y$. Consider the pullback by $f$ of the linear system of divisors which are linearly equivalent to $H$; 
when $X=Y=\P^2$ and $H$ is a line, this linear system is called the homaloidal net of $f$. The base locus of this linear system 
is supported on $\Ind(f)$, but it can of course include infinitely near points. We shall call it the {\textbf{base locus}} of $f$. 
To resolve the indeterminacies of $f$,  one blows up the base locus (see \cite{Lazarsfeld:Book} for base loci, base ideal, and their blow ups);
in other words, one blows up $\Ind(f)$, obtaining $\pi \colon X'\to X$, then one blows up $\Ind(f\circ \pi)$, and so on; the process stops
in a finite number of steps (see \cite{Hartshorne:book}, \S V.5).

\begin{rem} \label{rem:ind_in_exc}
If $f$ is a birational transformation of a projective surface with Picard number
one, then $\Ind(f)$ is contained in $\Exc(f)$. 
One can prove this by considering the factorization of $f$ as a sequence of blow-ups followed by a sequence of blow-downs. Any curve in $\Exc(f)$ corresponds to a $(-1)$-curve at some point in the sequence of blow-downs, but is also the strict transform of a curve of positive self-intersection from the source (this is where we use $\rho(X) = 1$).
\end{rem}

\subsection{Dynamical degrees}

The rational map $f\colon X\dashrightarrow Y$  determines a linear map $f^*\colon \NS(Y)\to \NS(X)$. For complex surfaces,
$f$  determines a linear map
$f^*\colon H^2(Y,\Z)\to H^2(X,\Z)$ which preserves the Hodge decomposition: The action of $f^*$ on $\NS(X)$
coincides with the action by pull-back on  $H^2(X(\C),\Z)_{t.f.}\cap H^{1,1}(X,\R)$ (see \cite{Diller-Favre:2001} for example). 

Assume now that $f$ is a birational  selfmap of $X$. The {\textbf{dynamical degree}} $\dd(f)$ of $f$ is 
the spectral radius of the sequence of linear maps $((f^n)^*)_{n\geq 0}$:
\[
\dd(f)=\lim_{n\to + \infty} \left( \lVert  (f^n)^* \rVert^{1/n} \right)
\]
where $\lVert \, .\, \rVert$ is an operator norm on $\End(\NS(X)_\R)$; the limit exists because the sequence
$\lVert (f^n)^* \rVert$ is submultiplicative (see \cite{Diller-Favre:2001}).
The number $\dd(f)$ is invariant under conjugacy: $\dd(f)=\dd(gfg^{-1})$ if
$g\colon X \dashrightarrow Y$ is a birational map. 

\begin{eg}
Let $[x:y:z]$ be homogeneous coordinates for the projective plane $\P^2_{\mathbf{k}}$. Let
$f$ be an element of $\Bir(\P^2_{\mathbf{k}})$; there are three homogeneous polynomials
$P,$ $Q,$ and $R\in {\mathbf{k}}[x,y,z]$ of the same degree $d,$ and without 
common factor of degree $\geq 1,$ such that 
\[
f[x:y:z]=[P:Q:R].
\] 
The degree $d$ is called the
degree of $f,$ and is denoted by $\deg(f).$ The action $f^*$ of $f$ on $\NS(\P^2_{\mathbf{k}})$ is the
multiplication by $\deg(f)$. 
The dynamical degree of $f$ is thus equal to the limit of $\deg(f^n)^{1/n}$.
\end{eg}

\subsection{Picard-Manin classes}\label{par:picman}


We follow the presentation given in \cite{Boucksom-Favre-Jonsson:2008,Cantat:Preprint, Favre:Bourbaki} which,  in turn, is inspired by the fifth chapter of Manin's book~\cite{Manin:Book}. 

\subsubsection{Models and Picard-Manin space}

A model of $X$ is a smooth projective surface $X'$ with a birational morphism $X' \to X$. If $\pi_1\colon X_1\to X$, and $\pi_2\colon X_2\to X$ are  two models, we say that $X_2$ dominates $X_1$ if the induced birational map  $\pi_1^{-1}\circ \pi_2 \colon X_2 \dashrightarrow X_1$ is a morphism. In this case, $\pi_1^{-1}\circ \pi_2$ contracts
a finite number of exceptional divisors and induces an injective map $(\pi_1^{-1}\circ \pi_2)^*\colon \NS(X_1)\ten \hookrightarrow \NS(X_2)\ten$. 

Let $\mod_X$ be the set of models that dominate $X$. 
If $X_1, X_2 \in \mod_X$ then,by resolving the indeterminacies of the induced birational map $X_1 \dashrightarrow X_2$ we obtain $X_3 \in \mod_X$ which dominates both $X_1$ and $X_2$.  
The space $\PM(X)$ of \textbf{(finite) Picard-Manin classes}
is the direct limit
\[
\PM(X) = \lim_{\rightarrow \mod_X}  \NS(X')\ten.
\]
The N\'eron-Severi group  $\NS(X')\ten$ of any model $X'\to X$ embeds in $\PM(X)$ and can be identified to its image into $\PM(X)$.
Thus, a Picard-Manin class is just a (real) N\'eron-Severi class of some model dominating $X$.
The Picard-Manin class of a divisor $D$ is still denoted by $[D],$ as for N\'eron-Severi classes.
Note that $\PM(X)$ contains the direct limit of the lattices $\NS(X')$ (with integer coefficients).
This provides an integral structure for $\PM(X)$. 
In the following paragraph we construct a basis of $\PM(X)$ made of integral points.

For all birational morphisms $\pi$, the pull-back operator $\pi^*$ preserves the intersection form and maps nef classes to nef classes; 
as a consequence, the limit space $\PM(X)$ is endowed with an intersection form (of signature $(1,\infty)$) and a nef cone. 

\subsubsection{Basis of $\PM(X)$}\label{par:basisofpm}

A basis of the real vector space $\PM(X)$ is constructed as follows. 
On the set of models $\pi\colon Y\to X$ together with marked points $p\in Y,$ 
we introduce the equivalence relation: $(p,Y) \sim (p',Y')$ if the induced birational map $\pi'^{-1}\circ \pi \colon Y \dashrightarrow Y'$ is an isomorphism in a neighborhood of $p$ that maps $p$ onto $p'$. Let $\pts_X$ be the quotient space; to denote points of $\pts_X$ we just write $p\in \pts_X,$ without any further reference to a model $\pi:Y\to X$ with $p\in Y$.

Let $(p,Y)$ be an element of $\pts_X$. Consider $\bar{Y} \to Y$ the blow-up of $p$ and  $E_p \subset \bar{Y}$ the exceptional divisor; the N\'eron-Severi class $[E_p]$ determines a Picard-Manin class and one easily verifies that this class depends only on the class $p\in \pts_X$ (not on the model $(Y,\pi)$, see \cite{Manin:Book}). 
The classes $[E_p],$ $p\in \pts_X$, have self-intersection $-1,$  are mutually orthogonal,
and are orthogonal to $\NS(X)\ten$. Moreover, 
\[
\PM(X)=\NS(X)\ten\oplus \Vect([E_p]; \, p\in \pts_X)
\]
and this sum is orthogonal with respect to the intersection form on $\PM(X)$. To get a basis of $\PM(X),$
we then fix a basis $([H_i])_{1\leq i\leq \rho(X)}$ of $\NS(X)\ten$, where the $H_i$ are
Cartier divisors, and complete it with the family $([E_p])_{p\in \pts_X}$.

\subsubsection{Completion}
The \textbf{(completed) Picard-Manin space} $\man(X)$ of $X$ is the $L^2$-completion of $\PM(X)$ (see \cite{Boucksom-Favre-Jonsson:2008, Cantat:Preprint} for details); in other words 
$$
\man(X) = \left\lbrace  [D] + \sum_{p \in \pts_X} a_p [E_p]; \, [D] \in \NS(X)\ten,\, a_p \in \R \mbox{ and } \sum a_p^2 < +\infty
\right\rbrace 
$$ 
whereas $\PM(X)$ corresponds to the case where the $a_p$ vanish for all but a finite
number of $p\in \pts_X$.
For the projective plane $\P^2_{\mathbf{k}}$, the N\'eron-Severi group $\NS(\P^2_{\mathbf{k}})$ is isomorphic to $\Z[H],$
where $H$ is a line; elements of $\man(X)$ are then given by sums 
\[
a_0 [H] + \sum_{p \in \pts_\cpd} a_p [E_p]
\]
with $\sum a_p^2 < +\infty$. 
We shall call this space the Picard-Manin space without further reference to $\P^2_{\mathbf{k}}$ or to the completion.

\subsection{Action of $\Bir(X)$ on $\man(X)$}\label{par:actionofbir}


If $\pi\colon X' \to X$ is a morphism, then $\pi$ induces an isomorphism $\pi^*\colon \PM(X) \to \PM(X')$. Let us describe this fact when $\pi$ is the (inverse of the) blow-up of a point 
$q \in X$. In this case we have 
\[
\NS(X') = \pi^*(\NS(X)) \oplus \Z [E_q] \quad \mbox{ and } \quad \pts_X = \pts_{X'} \cup \left\lbrace (q,X) \right\rbrace
\]
where $(q,X)\in \pts_X$ denotes the point of $\pts_X$ given by $q\in X$.
Thus the bases of $\PM(X)$ and $\PM(X')$ are in bijection, the only difference being that $[E_q]$ is first viewed 
as the class of an exceptional divisor in $\PM(X)$, and  then as an element of $\NS(X')\subset \PM(X')$; 
the isomorphism $\pi^*$ corresponds to this bijection of basis.
Note that  $\pi^*$ extends uniquely as a continuous  isomorphism
$\pi^*\colon \man(X') \to \man(X)$ that preserves the intersection form. 
Since any birational morphism $\pi$ is a composition of blow-downs of exceptional curves of the first kind, this
proves that $\pi^*$ is an isometry from $\PM(X)$ (resp. $\man(X)$) to $\PM(X')$ (resp. $\man(X')$).

Now consider $f\colon X \dashrightarrow X $ a birational map. There exist a surface $Y$ and two morphisms $\pi\colon Y \to X$, $\sigma\colon Y \to X$, such that $f = \sigma \circ \pi^{-1}$. Defining $f^*$ by $f^*=(\pi^*)^{-1}\circ\sigma^*$, and $f_*$ by $f_*=(f^{-1})^*$, we get a representation 
\[
f\mapsto f_*
\]
of $\Bir(X)$ into the orthogonal group of $\PM(X)$ (resp. $\man(X)$) with respect to the intersection form. This representation is faithful, because $f_*[E_p]=[E_{f(p)}]$ for all points $p$ in the domain of definition of $f$; it preserves the integral structure of $\PM(X)$ and the nef cone.

In what follows, we restrict the study to the projective plane $\P^2_{\mathbf{k}}$ and the Cremona group  $\Bir(\P^2_{\mathbf{k}})$.

\subsection{Action on an infinite dimensional hyperbolic space}
\subsubsection{The hyperbolic space $\hypman$}
We define
\[
\hypman = \left\lbrace [D] \in \man(\P^2_{\mathbf{k}}) ; \, \, [D]^2 = 1, [H] \cdot [D] > 0 \right\rbrace
\] 
and a distance $\hd$ on $\hypman$ by the following formula
\[
\cosh (\hd([D_1],[D_2])) = [D_1] \cdot [D_2].
\]
Since the intersection form is of Minkowski type, this endows  $\hypman$ with the structure of an infinite dimensional hyperbolic space, as in \S
\ref{par:hypspaces}. 

\subsubsection{Cremona isometries}\label{par:cremonaisometries}
By paragraph \ref{par:actionofbir} the action of the Cremona group on $\man(\P^2_{\mathbf{k}})$ preserves the two-sheeted hyperboloid 
$\left\lbrace [D] \in \man(\P^2_{\mathbf{k}}); [D]^2 = 1 \right\rbrace$ and since the action also preserves the nef cone, we obtain a faithful representation 
of the Cremona group into the group of isometries of $\hypman$:
\[
\Bir(\P^2_{\mathbf{k}}) \hookrightarrow \Isom(\hypman).
\]
In the context of the Cremona group, the classification of isometries into three types (see \S~\ref{par:isohypspace}) has an algebraic-geometric meaning. 

\begin{thm}[\cite{Cantat:Preprint}]
Let $f$ be an element of $\Bir(\P^2_{\mathbf{k}})$. 
The isometry $f_*$ of $\hypman$ is  hyperbolic if and only if the dynamical degree $\dd(f)$ is strictly larger than $1$. 
\end{thm}

As a consequence, when $\dd(f) > 1$ then $f_*$ preserves a unique 
geodesic line $\axe(f)\subset \hypman$; this line is the intersection of $\hypman$ with a 
plane $V_f\subset \man(\P^2_{\mathbf{k}})$ which intersects the isotropic cone of $\man(\P^2_{\mathbf{k}})$
on two lines $R^+_f$ and $R^-_f$ such that 
\[
f_* (a) =\dd(f)^{\pm 1} a
\] 
for all $a\in R^\pm_f$ (the lines $R^+_f$ and $R^-_f$ correspond to $\omega(f)$ and $\alpha(f)$ 
in the notation of \S~\ref{par:isohypspace}).

\begin{rem} \label{rem:trans_length}
The translation length $\lgt(f_*)$  is therefore equal to $ \log(\dd(f))$. Indeed, take $[\alpha] \in R^-_f$, $[\omega] \in R^+_f$ normalized such that $[\alpha]\cdot[\omega] = 1$. The point  $[P] = \frac{1}{\sqrt{2}}(  [\alpha] + [\omega])$ is on the axis $\axe(f_*) $. Since $f_*[P]= \frac{1}{\sqrt{2}}(\dd(f)^{-1}[\alpha] + \dd(f) [\omega])$, we get
\[
e^{\lgt(f_*)} + \frac{1}{e^{\lgt(f_*)}} = 2 \cosh\bigl(\hd([P], f_*[P])\bigr) = 2([P]\cdot f_*[P]) = \dd(f) + \frac{1}{\dd(f)}.
\]
\end{rem}

\begin{rem} Over the field of complex numbers $\C$, \cite{Cantat:Preprint} proves that:
$f_*$ is elliptic if, and only if, there exists a positive iterate $f^k$ of $f$ and a birational 
map $\eps\colon \cpd \dashrightarrow X$ such that $\eps \circ f^k \circ \eps^{-1}$ is 
an element of $\Aut(X)^0$ (the connected component of the identity in $\Aut(X)$); 
$f_*$ is parabolic if, and only if, $f$ preserves a pencil of elliptic curves and $\deg(f^n)$ grows quadratically with $n,$ or 
$f$ preserves a pencil of rational curves and $\deg(f^n)$ grows linearly with $n$.
\end{rem}

\subsubsection{Automorphisms} \label{par:autom-action}

Assume that $f\in \Bir(\P^2_{\mathbf{k}})$ is conjugate, via a birational transformation $\varphi$, to an automorphism $g$ of a smooth rational surface $X$:
$$\xymatrix{
X \ar[r]_g \ar@{-->}[d]_\varphi & X \ar@{-->}[d]^\varphi \\
\P^2_{\mathbf{k}} \ar@{-->}[r]_f & \P^2_{\mathbf{k}}
}$$
Then we have an isomorphism $\varphi_*\colon \man(X) \to \man(\P^2_{\mathbf{k}})$ and an orthogonal 
decomposition
\[
\man(X) = \NS(X)\ten \oplus {\NS(X)\ten} ^\perp
\]
where ${\NS(X)\ten}^\perp$ is spanned by the classes $[E_p],$ $p\in \pts_X$. 
This orthogonal decomposition is $g_*$-invariant. In particular, $f_*$ preserves
the finite dimensional subspace $\varphi_*\NS(X)\ten\subset \man(\P^2_{\mathbf{k}})$. 

By Hodge index theorem, the intersection form has signature $(1,\rho(X)-1)$ on 
$\NS(X),$ so that $\varphi_*\NS(X)\ten$ intersects $\hypman$ on an $f_*$-invariant hyperbolic
subspace of dimension $\rho(X)-1$. This proves the following lemma.

\begin{lem} \label{lem:f_*}
If $f$ is conjugate to an automorphism $g\in \Aut(X)$ by a birational map $\varphi\colon X\dashrightarrow \P^2_{\mathbf{k}},$
then:
\begin{enumerate}
\item The isometry $f_*\colon \hypman \to \hypman$ is hyperbolic (resp. parabolic, resp. elliptic) if and only if the isometry $g_*\colon \NS(X)\ten \to \NS(X)\ten$ is hyperbolic (resp. parabolic, resp. elliptic) for the intersection form on $\NS(X)\ten$; 
\item the translation length of $f_*$ is equal to the translation length of $g_*$; 
\item if $f_*$ is hyperbolic then, modulo $\varphi_*$-conjugacy,  the plane $V_f$ corresponds to $V_g$, which is contained
in $\NS(X)\ten$, and $ \axe(f_*) $ corresponds to $\axe(g_*)$. 
\end{enumerate}
\end{lem}

\subsubsection{Example: quadratic mappings (see \cite{Cerveau-Deserti:2009})}
The set of birational quadratic maps $\Bir_2(\P^2_\C)$ is an irreducible algebraic variety of
dimension $14$. Let  $f\colon \cpd\dashrightarrow \cpd$ be a quadratic birational map. 
The base locus of $f$ (resp.  $f^{-1}$) is made of 
three points $p_1, p_2, p_3$ (resp. $q_1, q_2,q_3$), where infinitely near points
are allowed.  
We have
\[ f_*([H]) = 2[H] - [E_{q_1}]  - [E_{q_2}]  - [E_{q_3}]
\]
with $[H]$ the class of a line in $\cpd$.
If $f$ is an isomorphism on a neighborhood of $p$, and $f(p) = q$, then
$
f_*([E_p]) = [E_q].
$
These formulas are correct even when there are infinitely near base points. For instance if $f$ is the H\'enon map
$$ [x:y:z] \dashrightarrow [yz:y^2-xz:z^2]$$
then $q_2$ is infinitely near $q_1$, and $q_3$ is infinitely near  $q_2$, but  the formula for the image of $[H]$ is still the same. 

 A Zariski open subset of $\Bir_2(\P^2_\C)$ is made of birational transformations
$f=h_1\circ \sigma\circ h_2$, with $h_i\in \Aut(\cpd)$, $i=1,2$, and $\sigma$
the standard quadratic involution 
\[
\sigma[x:y:z]=[yz:zx:xy].
\] 
For such maps, 
$\{p_1,p_2,p_3\}$ is the image of $\Ind(\sigma)=\{[1:0:0]; [0:1:0];[0:0:1]\}$
by $h_2^{-1}$, and $\{q_1,q_2,q_3\}=h_1(\Ind(\sigma))$. Moreover,
 the exceptional set $\Exc(f^{-1})$ is the union of the three lines through the pairs of points $(q_i,q_j)$, $i\neq j$. 
Assume, for example, that the line through $q_1$ and $q_2$ is contracted on $p_1$ by $f^{-1}$, then we have
\[
f_*([E_{p_1}]) = [H] - [E_{q_2}]- [E_{q_3}].
\]
If $f$ is a general quadratic map, then $\dd(f)$ is equal to $2$ and $f_*$ induces a hyperbolic isometry on $\hypman$; we shall see that it is possible to compute explicitly the points on the axis of $f$ (see \S \ref{sec:rigidity_generic} for precise statements) and that $[H]$ is \textit{not} on the axis. In fact, in this situation the axis of $f$ does not contain any finite class $[D]\in \PM(\P^2_\C)$.

%
%

\section{\Good\ birational maps}

%
%

\subsection{General Cremona transformations}
\label{par:generic}

%
%

In this section we prove Theorem~A concerning normal subgroups generated by iterates of general Cremona transformations.

\subsubsection{Jonqui\`eres transformations}

Let $d$ be a positive integer. 
As mentioned in the introduction, the set $\Bir_d(\cpd)$ 
of plane birational transformations  of degree $d$ is quasi-projective: It is a 
Zariski open subset in a subvariety of the projective space made of triples of homogeneous polynomials of degree $d$ modulo scalar multiplication.

Recall that $\J_d$ denotes the set of Jonqui\`eres transformations of degree $d$, defined as the birational transformations of degree $d$ of $\cpd$ that preserve the pencil of lines through $q_0=[1:0:0]$. Then we define $\V_d$ as the image of the composition map
\[
(h_1, f, h_2) \mapsto h_1 \circ f \circ h_2
\]
where $(h_1, f, h_2)$ describes $\PGL_3(\C)\times \J_d\times \PGL_3(\C)$.
As the image of an irreducible algebraic set by a regular map, $\V_d$ is an irreducible subvariety of $\Bir_d(\cpd)$. 
The dimension of $\Bir_d(\cpd)$ is equal to $4d+6$ and $\V_d$ is the unique
irreducible component of $\Bir_d(\cpd)$ of maximal dimension (in that sense, generic
elements of $\Bir_d(\cpd)$ are contained in $\V_d$). In degree $2,$ i.e. for quadratic
Cremona transformations, $\V_2$ coincides with a Zariski open subset of $\Bir_2(\cpd)$. 

Let $f$ be an element of  $\J_d$. In affine coordinates, 
\[
f(x,y)=(B_y(x),A(y))
\]
where $A$ is in $\PGL_2(\C)$ and $B_y$ in $\PGL_2(\C(y))$. 
Clearing denominators we can assume that
$B_y$ is given by a function $B\colon y\mapsto B(y)$ with 
\[
B(y)= \left(
\begin{array}{cc} a(y) & b(y) \\ c(y) & d(y) 
\end{array}
\right) \in \GL_2(\C(y)),
\] 
where the coefficients $a, b, c$ and $d$ are polynomials of respective degrees $d-1, d, d-2$ and $d-1$.
The degree of  the function $\det(B(y))$ is equal to $2d -2$; if $B$ is generic, $\det(B(y))$ has 
$2d-2$ roots $y_i,$ $1\leq i\leq 2d-2,$ and $B({y_i})$ is a rank $1$ complex matrix
for each of these roots. The image of $B({y_i})$ is a line, and this line corresponds
to a point $x_i$ in  $\P(\C^2)$. The birational transformation $f$ contracts each horizontal 
line corresponding to a root $y_i$ onto a point $q_i=(x_i,A(y_i))$. 
This provides $2d-2$ points of indeterminacy for $f^{-1}$; again, if $B$ is generic,
the $2d-2$ points $q_i$ are distinct, generic points of the plane. The same conclusion
holds if we change $f$ into its inverse, and gives rise to $2d-2$ indeterminacy points
$p_1, \cdots, p_{2d-2}$ for $f$. One more indeterminacy point (for $f$ and $f^{-1}$) 
coincides with $p_0 =q_0 =[1:0:0]$.

An easy computation shows that the base locus of $f$  is made of 
\begin{enumerate}
\item the point $p_0$ itself, with multiplicity $d-1$;
\item the $2d-2$ single points $p_1, \cdots, p_{2d-2}$.
\end{enumerate}
Any set of three distinct points $\{p_0,p_1, p_{2}\}$ such that  $p_0=[1:0:0]$ 
and $p_0$, $p_1$, and $p_2$ are not collinear is the indeterminacy set of a Jonqui\`eres 
transformation of degree $2$. All sets of distinct points $\{p_0, p_1, \cdots, p_{4}\}$ such that
$p_0=[1:0:0]$, no three of them are on a line through $p_0$, and there is no line containing $p_1$, $p_2$, 
$p_3$, and $p_4$  can be obtained as the indeterminacy
set of a Jonqui\`eres transformation of degree $3$. More generally, the indeterminacy sets of 
Jonqui\`eres transformations of degree $d$ form a non empty, Zariski open subset in the product
$\{p_0\}\times S^{2d-2}(\cpd)$, where $S^{2d-2}(\cpd)$ is the symmetric product of $2d-2$ copies
of the projective plane. 

In particular, on the complement of a strict Zariski closed subset of $\J_d,$ 
the points $p_i$ form a set of $2d-1$ distinct points in the plane: There are no infinitely
near points in the list.  Thus, for a generic element of $\V_d,$ we have 
\[
f_*[H]= d[H]-(d-1)[E_{p_0}]-\sum_{i=1}^{2d-2}[E_{p_i}]
\]
where the $p_i$ are generic distinct points of the plane.

\begin{rem} \label{rem:fourpoints}
If $\Sigma\subset \cpd$ is a generic set of $k$ points, and $h$ is an automorphism
of $\cpd,$ then $h$ is the identity map as soon as $h(\Sigma)\cap \Sigma$ contains five points. Applied
to $\Ind(f),$ we obtain the following:
Let $g$ be a generic element of $\V_d,$ and $h$ be an automorphism of $\cpd$;
if $h$ is not the identity map, then $h(\Ind(g))\cap \Ind(g)$ contains at most four points.
\end{rem}

\subsubsection{\Good ness is a general property in $\Bir_d(\cpd)$}

Our goal is to prove the following statement. 

\begin{thm}\label{thm:goodinvd}
There exists a positive integer $k$  such that for all integers $d\geq 2$ 
the following properties are satisfied by a general element $g\in \V_d$: 
\begin{enumerate}
\item $g$ is a \good\ Cremona transformation;
\item If $n\geq k,$ then $g^{n}$ generates a normal subgroup of 
$\Bir(\cpd)$ whose non trivial elements $f$  satisfy $\dd(f)\geq d^n$.
\end{enumerate}
\end{thm}

Since $\V_d$ is irreducible, \good\ Cremona transformations 
are dense in $\V_d$; since $\V_d$ is the unique component of $\Bir_d(\cpd)$ of
maximal dimension, properties (1) and (2) are also generally satisfied  in $\Bir_d(\cpd)$.
Thus Theorem~\ref{thm:goodinvd} implies Theorem~A.

\subsubsection{Strong algebraic stability is a general property}

Given a surface $X$, a birational transformation $f \in \Bir(X)$ is said to be \textbf{algebraically stable} if 
one of the following equivalent properties hold: 
\begin{enumerate}
\item if $x$ is a point of $\Ind(f),$ then $f^k(x)\notin \Ind(f^{-1})$ for all $k\leq 0$;
\item if $y$ is a point of $\Ind(f^{-1})$ then $f^k(y)\notin \Ind(f)$ for all $k\geq 0$;
\item the action $f^*$  of $f$ on the N\'eron-Severi group $\NS(X)$ satisfies $(f^k)^*=(f^*)^k$
for all $k\in \Z$.
\end{enumerate}
If $f$ is an element of $\Bir_d(\cpd),$ algebraic stability is equivalent to $\dd(f)=\deg(f)$
(if $f$ is not algebraically stable, property (3) implies $\dd(f)<\deg(f)$).
Condition (2)  can be rephrased by saying that $f$ is algebraically stable if
\[
f^k(\Ind(f^{-1}))\cap \Ind(f)=\emptyset
\]
for all $k\geq 0$. We now prove that general elements of $V_d$ satisfy a property which is stronger
than algebraic stability (recall from Remark \ref{rem:ind_in_exc} that $\Ind(f) \subset \Exc(f)$). 

\begin{lem}
\label{lem:generic-exc}
If $g$ is a general element of $\V_d$ then 
\begin{equation}\label{eq:sas}
g^k(\Ind(g^{-1}))\cap \Exc(g)= \emptyset
\end{equation}
for all $k\geq 0$. In particular $g$ is algebraically stable.
\end{lem}

\begin{proof}
For a fixed  integer $k\geq 0,$
condition (\ref{eq:sas}) is equivalent to the existence of a point $m\in \Ind(g^{-1})$ such that
$g^k(m)\in \Exc(g),$ and hence to ${\rm Jac}_g(g^k(m))=0,$ where ${\rm Jac}_g$
is the equation of $\Exc(g)$. This is an algebraic
condition, which defines an algebraic subvariety $I_k$ in $\Bir_d(\P^2_\C)$.

\begin{fact}
\label{fact:generic}
There exists $g \in \V_d$ such that $g \not\in I_k$ for any $k\geq 0$.
\end{fact}

Since $\V_d$ is irreducible,  this fact implies that the codimension of $I_k$ is positive and shows that equation (\ref{eq:sas})  is satisfied on the intersection of countably many
Zariski dense open subsets of $V_d$. 
We are therefore reduced to prove Fact \ref{fact:generic}. For this, let $f$ be the H\'enon mapping 
defined by
\[
f  \left( [x:y:z] \right)= \left[  \frac{yz^{d-1}}{10}: y^d + xz^{d-1}: z^d\right].
\]
The affine plane $\{z\neq 0\}$ is $f$-invariant, and $f$ restricts to a polynomial automorphism
of this plane.
The exceptional set $\Exc(f)$ is the line at infinity $\{ z=0\},$ and its image $\Ind(f^{-1})$ 
is the point $q = [0:1:0]$. This point is fixed by $f$:  In the affine chart $\{y\neq 0\},$ with
affine coordinates $(x,z)=[x:1:z]$ around $q,$ the map $f$  is given by
\[
(x,z) \mapsto \left( \frac{z^{d-1}}{10(1 + xz^{d-1})} , \frac{z^d}{1 + xz^{d-1}} \right).
\]
In particular, $q$ is an attracting fixed point.
Let $U$ be the neighborhood of $q$ defined by
\[
U = \left\lbrace [x:1:z] \,; \,  |x| < 1/10, |z| < 1/10 \right\rbrace.
\]
Then, $f(U)$ is contained in $\left\lbrace [x:1:z] \,; \,  |x| < 1/90, |z| < 1/90 \right\rbrace$.
Let $h$ be the linear transformation of the plane which, in the affine coordinates
$(x,z),$ is the translation $(x,z) \mapsto (x + 1/20, z+ 1/20)$.
We have $h(f(U)) \subset U \setminus \{z=0\}$, and $f(U) \cap h(f(U)) = \emptyset$. 
We now take $g=h\circ f$. Its exceptional set is the line at infinity $\{z=0\}$; the unique indeterminacy point of $g^{-1}$ is $h(q)=[1/20:1:1/20]$. By construction, for all $k\geq 0,$ 
$g^k(\Ind(g^{-1}))$ is a point of $U\setminus \Exc(g)$.  This proves that $g$ is not in $I_k,$ 
for any $k\geq 0$.
\end{proof}

We say that a Cremona transformation $g$ is \textbf{strongly algebraically stable} if $g$ and $g^{-1}$ satisfy  property (\ref{eq:sas}) stated in Lemma \ref{lem:generic-exc}. 
This lemma shows that general elements of $\V_d$ are strongly algebraically stable.

Recall that algebraic stability implies that $g$ is well defined along the forward orbit $g^k(\Ind(g^{-1})),$
$k\geq 0$; similarly, $g^{-1}$ is well defined along the backward orbit of $\Ind(g)$. 

\begin{lem}
\label{lem:pointsdisjoints}
Let $g$ be strongly algebraically stable. Then we have:
\begin{enumerate}
\item For all $k > j \ge 0$, $g^k(\Ind(g^{-1})) \cap g^j(\Ind(g^{-1})) = \emptyset$;
\item For all $k > j \ge 0$, $g^{-k}(\Ind(g)) \cap g^{-j}(\Ind(g)) = \emptyset$;
\item For all $k \ge 0$ and $j \ge 0$, $g^k(\Ind(g^{-1})) \cap g^{-j}(\Ind(g)) = \emptyset$.
\end{enumerate}
\end{lem}

\begin{proof}
Suppose that there is a point $q \in g^k(\Ind(g^{-1})) \cap g^j(\Ind(g^{-1}))$. 
Suppose first that $j = 0$. 
This means that there exist $p,q \in \Ind(g^{-1})$ such that $g^k(p) = q$. 
But then $g^{k-1}(p)\in \Exc(g),$  contradicting the assumption.
Now assume $j > 0$.
By the first step, we know that $g^{-1}$ is well defined along the positive orbit of $\Ind(g^{-1})$.
Thus we can apply $g^{-j}$, which brings us back to the case $j = 0$. 

Property (2) is equivalent to Property (1), replacing $g$ by $g^{-1}$.

Suppose (3) is false. Then there exists $p \in \Ind(g^{-1})$, $q \in \Ind(g)$, $k,j \ge 0$ such that $g^k(p) = g^{-j}(q)$. 
By (2) we can apply $g^j$ to the right hand side of this equality, thus $q \in g^{k+j}( \Ind(g^{-1})) \cap \Ind(g)$.
This contradicts the algebraic stability of $g$.
\end{proof}

\subsubsection{Rigidity is a general property}
\label{sec:rigidity_generic}

Let  $g$ be an element of $\V_d$. Consider the isometry $g_*$ of $\man(\P^2_{\mathbf{k}})$. 
If $g$ is algebraically stable, its dynamical degree is equal to $d$ and the translation 
length of $g_*$ is equal to $\log(d)$.
The Picard-Manin 
classes $[\alpha]$ and $[\omega]$ corresponding to the end points of the axis of $g_*$ 
satisfy $g_*[\alpha] =[\alpha]/d$ and $g_*[\omega]= d[\omega]$. To compute explicitly 
such classes, we start with the class $[H]$ of a line $H\subset \P^2_\C,$ and use the fact
that 
\[
\frac{1}{d^n}g_*^n [H]\to c^{ste}[\omega]
\]
when $n$ goes to $+\infty$ (see \S~\ref{par:cremonaisometries}). 

Assuming that $g$ is a general element of $\V_d$, its base locus is made of one point $p_0$ of multiplicity $d-1$ and $2d-2$ 
points $p_i$, $1\leq i\leq 2d-2$, of multiplicity $1$, and similarly for the base locus of $g^{-1}$.
Note $[E^+]$ (resp. $[E^-]$) the sum of the classes of the exceptional divisors, with multiplicity $d-1$ for the first one, 
obtained by blowing-up the $2d-1$ distinct points in $\Ind(g)$ (resp. $\Ind(g^{-1})$).
We have
\[
g_*[H] = d[H] - [E^{-}], \quad g^2_*[H] = d^2[H] -d[E^{-}] - g_*[E^{-}], \quad \mbox{etc.}
\]
Thus, the lines $R^-_g$ and $R^+_g$ of the Picard-Manin space generated by
\[
[\alpha] = [H] - \sum_{i = 1}^\infty \frac{g^{-i+1}_*[E^+]}{d^i} \quad {\text{and}}\quad 
[\omega] = [H] - \sum_{i = 1}^\infty \frac{g^{i-1}_*[E^{-}]}{d^i}
\]
correspond to the end points of the axis of $g_*$. 
By Lemma \ref{lem:pointsdisjoints}, both infinite sums appearing in these formulas are
sums of classes of the exceptional divisors obtained by blowing up the backward (resp. forward) orbit of $\Ind(g)$ (resp. $\Ind(g^{-1})$). 
By construction, $[\alpha]$ and $[\omega]$ satisfy $[\alpha] \cdot [\omega] = 1$ and $[\alpha]^2 = [\omega]^2 = 0$, because both of them 
are on the boundary of $\hypman$.
All points on $\axe(g_*)$ are linear combinations $u[\alpha] + v[\omega]$ with the condition
\[
1 = (u[\alpha] + v[\omega])^2 = 2uv.
\]
The intersection of $[H]$ with a point on $\axe(g_*)$ is minimal for $u = v = \frac{1}{\sqrt{2}}$ and is then 
equal to $\sqrt{2}$ (independently on $d$);
denote by $[P] = \frac{1}{\sqrt{2}}(\alpha + \omega)$  the point which realizes the minimum. We have 
\[
[P] = \sqrt{2}[H] - \frac{1}{\sqrt{2}} [R] \quad {\text{with}} \quad [R]=  \frac{[E^+] + [E^-]}{d} + \frac{g_*^{-1}[E^+] + g_*[E^-]}{d^2} +\cdots 
\]
Once again, Lemma \ref{lem:pointsdisjoints} implies the following fact.

\begin{fact}
\label{fact:pointsdisjoints}
The class $[R]$ is a sum of classes of exceptional divisors obtained by blowing up
distinct points of $\P^2_\C$ (there is no blow-up of infinitely near points).
\end{fact}

\begin{pro}
\label{prop:f=id}
Let $\eps_0 =0.289$. Let $d \ge 2$ be an integer and $g$ be a general element of $V_d$. 
Let $[P]$ be the Picard-Manin class defined above. If $f$ is a birational transformation of the plane such that 
$\hd(f_*[Q],[Q]) \le \eps_0$ for all $[Q]$ in $\{ g^{-1}_*[P], [P], g_*[P] \}$, 
then $f$ is the identity map.
\end{pro}

The proof uses explicit values for distances and hyperbolic cosines that
are recorded on Table \ref{table}.
Using this table, we see that
\[
\cosh(a_1+\eps)< 4 \, {\text{ and }}\, a_2+\eps <  a_4 < a_3
\]
as soon as  $\eps \leq \eps_0 = 0.289$.

\begin{table}[t]
\begin{tabular}{|c|c|}
\hline 
$\cosh(a_i)$ & $a_i$ \\
\hline
$3$ & $a_1 \simeq 1.76274 $ \\
$\sqrt{2}$&  $a_2 \simeq 0.88137 = a_1/2$\\
$ 3/\sqrt{2}$ &$a_3\simeq 1.38432 $\\
$ 5/(2\sqrt{2})$&$a_4 \simeq 1.17108$\\  
$4$ &$a_5 \simeq 2.06343$\\
\hline
\end{tabular}
\vspace{0.2cm}
\caption{Distances and hyperbolic cosines}
\label{table}
\end{table}

\begin{proof} We proceed in two steps.\\

\noindent{\textbf{First step.--}} We show that if $\hd(f_*[P],[P])\leq \eps_0$, then $f$ is linear.\\

By the triangular inequality
\begin{eqnarray*}
\hd(f_*[H],[H]) & \le & \hd(f_*[H],f_*[P]) + \hd(f_*[P],[P]) + \hd([P],[H]) \\
                        & \le & 2\hd([P],[H]) + \eps_0.
\end{eqnarray*}
Recall that $\cosh(\hd([D_1],[D_2])) = [D_1]\cdot [D_2]$ for all pairs of points $[D_1],$ $[D_2]$ in $\hypman$ and that the degree of $f$ is given by  
$\deg(f)= f_*[H] \cdot [H]$. Using Table \ref{table} we see that $ [P]\cdot [H] = \sqrt{2}$ implies $2\hd([P],[H]) = 2 a_2 = a_1$. Taking hyperbolic cosines we get
\[
\deg(f) \le \cosh(a_1+\eps_0) < 4.
\]
We conclude that $\deg(f) \le 3$. \\

Now we want to exclude the cases $\deg(f) = 3$ or $2$. The following remark will be used twice: Since $[H]\cdot [P]=\sqrt{2}$, we have $\hd([H],[P])=a_2$; thus, if $f_*$ is an isometry and $\hd(f_*[P],[P])\leq \eps_0$, applying hyperbolic cosines to the triangular inequality $\hd(f_*[H],[P]) \le \hd(f_*[H],f_*[P]) + \hd(f_*[P],[P])$, we get $f_*[H]\cdot [P]\leq \cosh(a_2+\eps_0)$.\\

Suppose   that $\deg(f) = 3$.  Then 
\[
f_*[H] = 3[H] - 2[E_1] - [E_2] -[E_3]-[E_4]-[E_5]
\]
for some exceptional divisors $E_i$ above $\P^2_\C$ (they may come from blow-ups of infinitely near points).
By Fact \ref{fact:pointsdisjoints}, all exceptional classes in the infinite sum defining $[R]$
come from blow-ups of distinct points; hence
\[
[2E_1+E_2+E_3+E_4+E_5]\cdot [R] \ge \frac{1}{d} (-2(d-1)-1-1-1-1) \ge -3.
\]
Consider  the point $[D]$ of the Picard-Manin space such that $[P]$ is the middle point of the geodesic segment from $[H]$ to $[D]$; explicitly 
\[
[D] =  2\sqrt{2}[P] - [H] = 3[H] - 2[R].\\
\]
We obtain 
\begin{eqnarray*}
f_*[H]\cdot [D] &=& f_*[H]\cdot (3[H] - 2[R]) \\
 &=& 9 + 2 [2E_1+E_2+E_3+E_4+E_5]\cdot [R]  \\
&\ge& 3.
\end{eqnarray*}
On the other hand, $f_*[H]\cdot [H]=3$ because $f$ has degree $3$.
Since $2\sqrt{2}[P]=[H]+[D]$ we obtain
\[
\cosh(a_2+\eps_0)\geq f_*[H]\cdot [P] \geq \frac{3+3}{2\sqrt{2}}=\frac{3}{\sqrt{2}}=\cosh(a_3).
\]
This contradicts the choice of $\eps_0$. \\

Suppose that $\deg(f) = 2$. We have $ f_*[H] = 2[H] - [E_1] - [E_2] -[E_3]$
where the $[E_i]$  are classes of exceptional divisors above $\P^2_\C.$ The product $f_*[H]\cdot [P]$  is given by
\begin{eqnarray*}
f_*[H]\cdot [P] & = & [2H - E_1 - E_2 -E_3]\cdot \left( \sqrt{2}[H] - \frac{1}{\sqrt{2}} [R]\right) \\
& = & 2\sqrt{2} + \frac{1}{\sqrt{2}} [E_1+E_2+E_3] \cdot [R] \\
& \geq & 2\sqrt{2} + \frac{1}{\sqrt{2}} (-1-\frac{1}{d}) \\
& \geq  & 2\sqrt{2} - \frac{3}{2\sqrt{2}} = \frac{5}{2\sqrt{2}}=\cosh(a_4),
\end{eqnarray*}
where the inequality follows from Fact \ref{fact:pointsdisjoints}: The curves $E_1,$ $E_2,$ and $E_3$ appear at most once in $[R]$, so that  $ [E_1+E_2+E_3]\cdot [R] \ge -1 -\frac{1}{d} \ge - \frac{3}{2}$. As a consequence, 
\[
\cosh(a_2+\eps_0)\geq \cosh(a_4),
\]
in contradiction with the choice of $\eps_0$. \\

Thus, we are reduced to the case where $f$ is an element of $\Aut(\P^2_\C)$. 
\\

\noindent{\textbf{Second step.--}} Now we show that if $f$ is in $\Aut(\P^2_\C)$ and  $\hd(f_*[Q],[Q]) \le \eps_0$ for $[Q]$ in $\{ g^{-1}_*[P], g_*[P] \}$, then $f$ is the identity map.\\

Applying the assumption to $[Q] = g^{-1}_*[P]$, we get
\[
\hd((gfg^{-1})_*[P],[P]) \le \eps_0
\] 
and the first step shows that $gfg^{-1}$ must be linear. This implies that $f(\Ind(g))$ coincides with $\Ind(g)$. The same argument
with $[Q]=g_*[P]$ gives  $f(\Ind(g^{-1})) = \Ind(g^{-1})$. If $\deg(g) \ge 3$, the set $\Ind(g)$ is a generic set of points in $\cpd$ with cardinal at least 5, so $f(\Ind(g)) = \Ind(g)$ implies that $f$ is the identity map (see remark  \ref{rem:fourpoints}). Finally, if $\deg(g) = 2$, the set $\Ind(g) \cup \Ind(g^{-1})$ is a generic set of 6 points in the plane, so if   $f(\Ind(g)) = \Ind(g)$ and $f(\Ind(g^{-1})) = \Ind(g^{-1})$ then again $f$ is the identity map.
\end{proof}

\begin{cor}
\label{cor:rigidity}
If  $g$ is a general element of $V_d$, then $\axe(g_*)$ is rigid.
\end{cor}

\begin{proof}
Suppose $\axe(g_*)$ is not rigid. Choose $\eta > 0$. By Proposition \ref{pro:almostfixed} there exists $f$ which does not preserve $\axe(g_*)$ such that $d(f_*[Q],[Q]) < \eta$ for all $[Q]$ in the segment $\bigl[ g^{-1}_*[P],g_*[P] \bigr]$. For $\eta < \eps_0$, this contradicts Proposition~\ref{prop:f=id}.
\end{proof}

\begin{proof}[Proof of Theorem \ref{thm:goodinvd}]
Let $g$ be a general element of $\V_d,$ with $d\geq 2$.
Suppose that $f_*(\axe(g_*)) = \axe(g_*)$. If $f_*$ preserves the orientation on $\axe(g_*)$, then $(fgf^{-1}g^{-1})_*$ fixes each point in $\axe(g_*)$, and Proposition \ref{prop:f=id} gives $fgf^{-1} = g$. Similarly, if $f_*$ reverses the orientation, considering $(fgf^{-1}g)_*$ we obtain $fgf^{-1} = g^{-1}$.

Since we know by Corollary \ref{cor:rigidity} that $\axe(g_*)$ is rigid, we obtain that $g_*$ is \good, hence by Theorem~\ref{thm:criterion} $g_*^k$ generates a proper subgroup of $\Bir(\cpd)$ for large enough $k$. 

We can be more precise on $k$ by reconsidering the proof of Corollary \ref{cor:rigidity}. The rigidity of the axis of $g_*$ follows from Proposition \ref{prop:f=id}, where the
condition on $\eps_0$ is independent of $d$, so by Corollary \ref{cor:constants} we obtain that $\axe(g_*)$ is $(\eps_0/2, 7\lgt(g_*))$-rigid. 
Recall that $\theta$ is defined in section \ref{par:approxbytrees} by $\theta= 8\delta$ and that $\delta=\log(3)$ works for the hyperbolic space $\hypman$; thus 
we can choose $\theta=8\log(3)$, so that $14\theta=112\log(3)=123.4...\leq 124$.
By Lemma \ref{lem:rigidPM} we obtain that $\axe(g_*)$ is also $(14\theta, B)$-rigid, for 
\begin{eqnarray*}
B & \geq & \max\left\{2728, 2232 + 2\log\left(5\frac{\cosh(56\theta)-1}{\cosh(\eps_0/2)-1}\right), (28/3)\lgt(g_*)+ 248\right\}\\
& = & \max\left\{ 3220, (28/3)\lgt(g_*)+ 248\right\}.
\end{eqnarray*}
Then, Theorem~\ref{thm:criterion} requires $k \lgt(g_*)\geq 40 B + 1200 \theta$. 
Thus, we see that 
\[
k=\max\left( \frac{139347}{\lgt(g_*)}, 374+\frac{10795}{\lgt(g_*)}\right)\leq 201021
\]
is sufficient to conclude that $g^k$
generates a proper normal subgroup of $\Bir(\P^2_\C)$. Asymptotically, for degrees
$d$ with $d\geq \exp(10796)$, we can take $k=375$. 
\end{proof}

\subsection{Automorphisms of rational surfaces}
\label{par:goodauto}

\subsubsection{Preliminaries} 
Let ${\mathbf{k}}$ be an algebraically closed field, and $X$ 
be a rational surface defined over $\mathbf{k}$. 
Let $g$ be an automorphism of  $X$. 
If $\varphi \colon X\da \P^2_{\mathbf{k}}$ is a birational map, conjugation by $\varphi$
provides an isomorphism between $\Bir(X)$ 
and $\Bir(\P^2_{\mathbf{k}})$ and $\varphi_*$ conjugates the actions of $\Bir(X)$
on $\man(X)$ and $\Bir(\P^2_{\mathbf{k}})$ on $\man(\P^2_{\mathbf{k}})$. We thus identify
$\Bir(X)$ to $\Bir(\P^2_{\mathbf{k}})$ and $\man(X)$ to $\man(\P^2_{\mathbf{k}})$ without 
further reference to the choice of $\varphi$.
This paragraph provides a simple criterion 
to check whether  $g$  is a 
\good\ element of the Cremona group $\Bir(X)$.

As explained in \S \ref{par:autom-action}, the N\'eron-Severi group
$\NS(X)$ embeds into the Picard-Manin space $\man(\P^2_{\mathbf{k}})$ and  $g_*$ preserves
the orthogonal decomposition
\[
\man(\P^2_{\mathbf{k}})=\NS(X)\ten \oplus {\NS(X)\ten}^\perp,
\]
where ${\NS(X)\ten}^\perp$ is the orthogonal complement with respect to the intersection 
form. 
Assume that $g_*$ is hyperbolic, with axis $\axe(g_*)$ and translation length $\log(\dd(g))$. 
By Lemma \ref{lem:f_*}, the plane
$V_g$ such that $\axe(g_*)=\hypman \cap V_g$ is contained in $\NS(X)\ten$.
\begin{rem}
\label{rem:Vgh}
Let $h$ be a birational transformation of $X$ such that $h_*$ preserves $V_g$. The restriction 
of $h_*$ to $V_g$ satisfies  one  of the following two properties.
\begin{enumerate}
\item $h_*$ and $g_*$ commute on $V_g$;
\item $h_*$ is an involution of $V_g$ and $h_*$ conjugates $g_*$ to its inverse on $V_g$. 
\end{enumerate}
Indeed, the group of isometries of a hyperbolic quadratic form in two variables
is isomorphic to the semi-direct product $\R \rtimes \Z/2\Z$. 
\end{rem}

\begin{lem} 
\label{lem:auto-hyp-faithful}
Let $g$ be a hyperbolic automorphism of a rational surface $X$. 
Assume that 
\begin{enumerate}[(i)]
\item $g_*$ is the identity on the orthogonal complement of $V_g$ in $\NS(X)\ten$;
\item the action of $\Aut(X)$ on $\NS(X)$ is faithful. 
\end{enumerate}
Let $h$ be an automorphism of $X$ such that $h_*$ preserves $V_g$.
Then $h g h^{-1} $ is equal to $g$ or~$g^{-1}$. 
\end{lem}

\begin{proof}
Let us study the action of $h_*$ and $g_*$ on $\NS(X)$.
Since $h_*$ preserves $V_g,$ it preserves
its orthogonal complement $V_g ^\perp$ and, by assumption $(i),$ commutes to 
$g_*$ on $V_g^\perp$.
If $h_*$ commutes to $g_*$ on $V_g,$ then $h_*$ and $g_*$ commute on $\NS(X)$,
and the conclusion follows from the second assumption.
If $h_*$ does not commute to $g_*,$ Remark \ref{rem:Vgh} implies that $h_*$ 
is an involution on $V_g,$ and that $h_*  g_*  (h_*) ^{-1} = (g_*)^{-1}$
on $V_g$ and therefore also on $\NS(X)$. Once again, assumption 
$(ii)$ implies that $h$ conjugates $g$ to its inverse.
\end{proof}

\begin{rem}
\label{rem:qi}
The plane $V_g$ is a subspace of $\NS(X)\ten$ and it may very well 
happen that this plane does not intersect the lattice $\NS(X)$ (except
at the origin). But, if $g_*$ is the identity on $V_g^\perp,$ then 
$V_g^\perp$ and $V_g$ are defined over $\Z$. In that case, 
$V_g\cap \NS(X) $ is a rank 2 lattice in the plane $V_g$. This lattice 
is $g_*$-invariant, and the spectral values of the linear transformation 
$g_*\in \GL(V_g)$ are quadratic integers. From this follows that $\dd(g)$
is a quadratic integer.
On the other hand, if $\dd(g)$ is a quadratic integer, then $V_g$ 
is defined over the integers, and so is $V_g^\perp$. The restriction of 
$g_*$ to $V_g^\perp$ preserves the lattice $V_g^\perp \cap \NS(X)$ 
and a negative definite quadratic form. This implies that a positive 
iterate of $g_* $ is the identity on $V_g^\perp$. 
\end{rem}

Let us now assume that there exists an ample class $[D']\in \NS(X)\cap V_g$.
In that case, $[D']$ and $g_*[D']$ generate a rank $2$ subgroup of $\NS(X)\cap V_g$ 
and Remark \ref{rem:qi} implies that $\dd(g)$ is a quadratic integer.
In what follows, we shall denote by $[D]$ the class 
\begin{equation}\label{eq:defD'}
[D]=\frac{1}{\sqrt{[D']\cdot [D']}}[D'].
\end{equation}
This is an ample class  with real coefficients that determines a point $[D]$ in~$\hypman$.

\begin{lem}
\label{lem:rig-auto1}
Let $h$ be a birational transformation of a projective surface $X$. 
Let $[D']\in \NS(X)$ be an ample class, and $[D]= [D']/{\sqrt{[D']\cdot [D']}} \in \hypman$. If
\[
\cosh(\hd(h_*[D], [D])) < 1+  ([D']\cdot [D']) ^{-1}
\]
then $h_*$ fixes $[D']$ and is an automorphism of $X$. 
\end{lem}

\begin{proof} 
Write $h_*[D]=  [D]+ [F] +[R] $ where $[F]$ is in $\NS(X)\ten$  
and $[R]$ is in the subspace ${\NS(X)\ten}^\perp$ of $\man(X)$. 
More precisely, $F$ is an element of $\NS(X)$ divided by
the square root of  the self intersection $[D']^2$, and $[R]$ is a sum $\sum m_i [E_i]$ 
of exceptional divisors obtained by blowing up points of $X$, coming
from indeterminacy points of $h^{-1}$; the $m_i$ are integers divided by the square root 
of $[D']^2$. We get 
\[
1 \le [D] \cdot h_*[D] = [D]\cdot([D] + [F]+[R]) = 1 + [D]\cdot [F]
\]
because $[D]$ does not intersect the $[E_i]$. The number $[D]\cdot [F] $ is a non negative
integer divided by $[D']^2$. By assumption, this number is less than $([D']\cdot [D']) ^{-1}$ and so it must be zero. In other words, the distance between 
$[D]$ and $h_*[D]$ vanishes, and $[D]$ is fixed.
Since $[D]$ is ample,  $h$ is an automorphism of $X$. 
\end{proof}

\begin{pro}
\label{pro:auto-good}
Let $g$ be a hyperbolic automorphism of a rational surface $X$. 
Assume that 
\begin{enumerate}[(i)]
\item  $V_g$ contains an ample class $[D']$ and 
\item $g_*$ is the identity on $V_g^\perp\cap \NS(X)$.
\end{enumerate}
Then $\axe(g_*)$ is rigid. Assume furthermore that
\begin{enumerate}[(i)]
\setcounter{enumi}{2}
\item if $h\in \Aut(X)$ satifies $h_*(\axe(g_*)) = \axe(g_*)$ then  $hgh^{-1}=g$ or $g^{-1}$. 
\end{enumerate}
Then any $h\in \Bir(X)$ which 
preserves $\axe(g_*)$ is an automorphism of $X$, and $g$ is a \good\ element of $\Bir(X)$.
Thus, for sufficiently large $k$ the iterate $g^k$ generates
a non trivial normal subgroup in the Cremona 
group $\Bir(\cpd) = \Bir(X)$.
\end{pro}

Note that if the action of $\Aut(X)$ on $\NS(X)$ is faithful then by Lemma \ref{lem:auto-hyp-faithful} condition $(iii)$ is automatically satisfied. 

\begin{proof}
Let $[D] \in \axe(g_*)$ be the ample class defined by equation (\ref{eq:defD'}).
If the axis of $g_*$ is not rigid, Proposition \ref{pro:almostfixed} provides a birational transformation 
$f$ of $X$ such that the distances between $f_*[D]$ and $[D]$ and between 
$f_*(g_*[D])$ and $g_*[D]$ are bounded by $([D']^2)^{-1},$ 
and, moreover, $f_*(\axe(g_*))\neq \axe(g_*)$. Lemma \ref{lem:rig-auto1}
implies that $f$ is an automorphism of $X$ fixing both $[D]$ and $g_*[D]$.
This contradicts $f_*(\axe(g_*))\neq \axe(g_*)$ and shows that $\axe(g_*)$
is rigid.

Assume now that $h\in\Bir(X)$ preserves the axis of $g_*$. Then $h_*[D']$ is an ample class, hence $h$ is an automorphism of $X$.
Property $(iii)$ implies that $g$
is a \good\ element of $\Bir(X)$. The conclusion follows from Theorem~\ref{thm:criterion}.
\end{proof}

In the following paragraphs we construct two families of examples which satisfy the assumption of Proposition \ref{pro:auto-good}. Note that the surfaces $X$ that we shall consider 
have quotient singularities; if we blow-up $X$ to get a smooth surface, then the class
$[D']$ is big and nef but is no longer ample. The first example  is defined over the field of
complex numbers $\C$, while the second  works for any algebraically closed
field $\mathbf{k}$.

\subsubsection{Generalized Kummer surfaces}

Consider $\Z[i] \subset \C$ the lattice of Gaussian integers, and let $Y$ be the
abelian surface $\C/\Z[i]\times\C/\Z[i]$. The group $\GL_2(\Z[i])$ acts by linear
transformations on $\C^2,$ preserving the lattice $\Z[i] \times \Z[i]$; this provides an
embedding $\GL_2(\Z[i])\to \Aut(Y)$. Let $X$ be the quotient of $Y$ by the action 
of the group of order 4 generated by $\eta(x,y)=(ix,iy)$. This surface is rational, with 
ten singularities, and all of them are quotient singularities that can be resolved
by a single blow-up. Such a surface is a so called (generalized) {\textbf{Kummer surface}};
classical Kummer surfaces are quotient of tori by $(x,y)\mapsto (-x,-y),$ 
and are not rational (these surfaces are examples of K3-surfaces). 

The linear map $\eta$ generates the center of the group $\GL_2(\Z[i])$.
As a consequence, $\GL_2(\Z[i])$, or more precisely $\PGL_2(\Z[i]),$ acts
by automorphisms on $X$. 
Let $M$ be an element of the subgroup $\SL_2(\Z)$ of $\GL_2(\Z[i])$ such that 
\begin{enumerate}
\item[(i)] the trace ${\text{tr}}(M)$ of $M$ is at least $3$;
\item[(ii)]  $M$ is in the level $2$ congruence subgroup of $\SL_2(\Z),$ i.e.
$M$ is equal to the identity modulo $2$.
\end{enumerate}
Let $\hat{g}$ be the automorphism of $Y$ defined by $M$ and $g$ be the automorphism of $X$ induced by $\hat{g}$.

\begin{thm}
\label{thm:kummergood}
The automorphism $g\colon X\to X$ satisfies properties $(i)$, $(ii),$ and $(iii)$ of Proposition
\ref{pro:auto-good}. In particular, $g$ determines a \good\ element of $\Bir(\cpd)$; hence,
if $k$ is large enough, $g^k$  generates a non trivial normal subgroup of $\Bir(\cpd)$.
\end{thm}

\begin{proof}
The N\'eron-Severi group of $Y$ has rank $4,$ and is generated by the following classes:
The class $[A]$ of horizontal curves $\C/\Z[i]\times \{*\},$ 
the class $[B]$ of vertical curves $ \{*\}\times\C/\Z[i]$,
the class $[\Delta]$ of the diagonal $\{(z,z)\in Y\}$,
and the class $[\Delta_i]$ of the graph $\{(z,iz)\in Y\, \vert\, \, z \in \C/\Z[i]\}$.

The vector space $H^2(Y,\R)$ is isomorphic to the space of bilinear alternating 
two forms on the $4$-dimensional real vector space $\C^2$. The action of
$g_*$ is given by the action of $M^{-1}$ on this space.
The dynamical degree of $\hat{g}$ is equal to the square of the spectral radius of $M,$ 
i.e. to the quadratic integer
\[
\dd(\hat{g})=\frac{1}{2} \left( a +\sqrt{a^2 - 4}\right)
\]
where $a={\text{tr}}(M)^2 -2 > 2$. Thus, the plane $V_{\hat{g}}$ intersects
$\NS(Y)$ on a lattice. Let $[F]$ be an element of $V_{\hat{g}}\cap \NS(X)$
with $[F]^2>0$ (see Remark \ref{rem:qi}). Since $Y$ is an abelian variety, then $[F]$ (or its opposite) is ample.
Since $M$ has integer coefficients, the linear map $\hat{g}_*$ preserves the three dimensional subspace $W$  of $\NS(Y)$ 
generated by $[A],$ $[B]$ and $[\Delta]$. The orthogonal complement of $V_{\hat{g}}$ intersects $W$ on a line, on which $\hat{g}_*$ must be the identity, because $\det(M)=1$. The orthogonal complement of $W$ is also a line, so that $\hat{g}_*$ is the identity on $V_{\hat{g}}^\perp \subset \NS(Y)$. 

Transporting this picture in $\NS(X),$ we obtain: 
The dynamical degree of $g$ is equal to 
the dynamical degree of $\hat{g}$ (see \cite{Guedj:Panorama}, for more general results), 
the plane $V_{\hat{g}}$ surjects onto $V_g$, the image of $[F]$ is an ample
class $[D']$ contained in $V_g\cap \NS(X),$ and $g_*$ is the identity on $V_g^{\perp}$. 

Automorphisms of $X$ permute the ten singularities of $X$. The fundamental 
group $\Gamma$ of $X\setminus {\text{Sing}}(X)$ is the affine group $\Z/4\Z\ltimes (\Z[i]\times\Z[i])$  
where  $\Z/4\Z$ is generated by $\eta$. The abelian group $\Z[i]\times\Z[i]$
is the unique maximal free abelian subgroup of rank $4$ in $\Gamma$ and, as such, 
is invariant under all automorphisms of $\Gamma$. This implies that all automorphisms
of $X$ lift to (affine) automorphisms of~$Y$. 

Let $h$ be an automorphism of $X$ which preserves the axis $\axe(g_*)$. Then $h_*$ 
conjugates $g_*$ to $g_*$ or $(g_*)^{-1}$ (Remark \ref{rem:Vgh}), and we must show that $h$ conjugates 
$g$ to $g$ or $g^{-1}$. Let $\hat{h}$ be a lift of $h$ to $Y$. There exists a linear 
transformation $N\in \GL_2(\C)$ and a point $(a,b)\in Y$ such that
\[
\hat{h}(x,y)=N(x,y)+(a,b).
\]
The lattice $\Z[i]\times\Z[i]$ is $N$-invariant, and $N$ conjugates $M$ to $M$ or
its inverse $M^{-1},$ because $h_*$ conjugates $g_*$ to $g_*$ or its inverse. 
Then 
\[
\hat{h}\circ\hat{g}\circ\hat{h}^{-1}= M^{\pm 1}(x,y) + (\text{Id} - M^{\pm 1})(a,b) .
\]
On the other hand, since $\hat{h}$ is a lift of an automorphism of $X,$
the translation $t\colon (x,y)\mapsto (x,y)+(a,b)$ is normalized by the cyclic 
group generated by $\eta$. Thus $a$ and $b$ are in $(1/2)\Z[i]$. Since
$M$ is the identity modulo $2,$ we have $M(a,b)=(a,b)$ modulo $\Z[i]\times \Z[i]$.
Hence  $\hat{h}\circ\hat{g}\circ\hat{h}^{-1}=\hat{g}^{\pm 1}$ and, coming back to $X,$  $h$ conjugates
$g$ to $g$ or its inverse.
\end{proof}

\begin{rem}
The lattice of Gaussian integers can be replaced by the lattice of Eisenstein integers $\Z[j]\subset\C$,
with $j^3=1$, $j\neq 1$, and the homothety $\eta$ by $\eta(x,y)=(jx,jy)$. This leads to a second rational
Kummer surface with an action of the group $\PSL_2(\Z)$; a statement similar to Theorem~B can 
be proved for this example. 
\end{rem}

\subsubsection{Coble surfaces}\label{par:coble}

Let ${\mathbf{k}}$ be an algebraically closed field.
Let $S\subset \P^2_{\mathbf{k}}$ be a rational sextic curve, with ten double points $m_i,$ $1\leq i \leq 10$ ;
such sextic curves exist and, modulo the action of $\Aut(\P^2_{\mathbf{k}})$, they depend on $9$ parameters
(see \cite{Halphen:1882}, the appendix of \cite{Gizatullin:1980},
or \cite{Cossec-Dolgachev:book}). 
Let $X$ be the surface obtained by blowing up the ten double points of $S$: By definition 
$X$ is the {\textbf{Coble surface}}  defined by $S$.

Let $\pi\colon X\to \P^2_k$ be the natural projection and $E_i,$ $1\leq i\leq 10,$ be the exceptional divisors of $\pi$.  The canonical class of $X$ is 
\[
[K_X]=-3[H]+\sum_{i=1}^{10} [E_i]
\]
where $[H]$ is the pullback by $\pi$ of the class of a line. The strict transform
$S'$ of $S$ is an irreducible divisor of $X,$ and its class $[S']$ coincides with $-2[K_X]$; more precisely, 
there is a rational section $\Omega$ of $2 K_X $ that does not vanish and has simple poles 
along $S'$. 

\begin{rem}
Another definition of Coble surfaces requires $X$ to be a smooth rational surface 
with a non zero regular section of $-2K_X$: Such a definition includes our Coble surfaces (the section vanishes
along $S'$) but it includes also the Kummer surfaces from the previous
paragraph (see \cite{Dolgachev-Zhang:2001}). Our definition is more restrictive.
\end{rem}

The self-intersection of $S'$ is $-4,$ and $S'$ can be blown down: This provides a birational morphism $q\colon X\to X_0$;
the surface $X_0$ has a unique singularity, at $m=q(S')$. The section $\Omega$ defines a holomorphic section of $K_{X_0}^{\otimes 2}$
that trivializes $2 K_{X_0}$ in the complement of $m$; in particular, $H^0(X,-2 K_X)$ has dimension $1,$ and the base locus of $-2K_X$
coincides with $S'$. The automorphism group $\Aut(X)$ acts linearly on the space of sections of $-2K_X,$ and preserves its
base locus $S'$. It follows that $q$ conjugates $\Aut(X)$ and $\Aut(X_0)$.

The rank of the N\'eron-Severi group $\NS(X)$ is equal to $11$. Let $W$ be the orthogonal complement 
of $[K_X]$ with respect to the intersection form. The linear map $q^*\colon \NS(X_0)\to \NS(X)$ provides an isomorphism
between $\NS(X_0)$ and its image $W=[K_X]^\perp\subset\NS(X)$. Let $O'(\NS(X))$ be the group of isometries of the lattice $\NS(X)$ 
which preserve the canonical class $[K_X]$. 

We shall say that $S$ (resp. $X$) is {\bf{special}} when at least one of the following properties occurs 
(see \cite{Dolgachev:1986}, \cite{Cossec-Dolgachev:book} page 147):
\begin{enumerate}
\item three of the points $m_i$ are colinear;
\item six of the points $m_i$ lie on a conic;
\item eight of the $m_i$ lie on a cubic curve with a double point at one of them;
\item the points $m_i$ lie on a quartic curve with a triple point at one of them.
\end{enumerate}

\begin{thm}[Coble's theorem]\label{thm:coble}
Let ${\mathbf{k}}$ be an algebraically closed field.
The set of special sextics is a  proper Zariski closed subset of the space of rational sextic curves $S\subset \P^2_k$.
If $S$ is not special and $X$ is the associated Coble surface then
the morphism 
\begin{align*}
\Aut(X) &\to O'(\NS(X))\\
 f\quad &\mapsto \quad f_*
\end{align*}
is injective and its image is the level $2$ congruence subgroup of $O'(\NS(X))$.
\end{thm}

This theorem is due to Coble, and can be found in \cite{Coble:Book}. A proof is sketched in 
\cite{Dolgachev-Ortland:1989, Dolgachev:2008}, and a complete, characteristic free, proof is available in 
\cite{Cantat-Dolgachev};  analogous result and proof  for generic Enriques surfaces can be found in
\cite{Dolgachev:1984}. The main steps are the following. The automorphism group of $X$ 
can be identified with a normal subgroup of $O'(\NS(X))$; this group
is generated, as a normal subgroup, by an explicit involution. To realize this
involution by an automorphism, one constructs a $2$ to $1$ morphism from $X$ 
to a Del Pezzo surface and takes the involution of this cover. 

\begin{rem}
\label{rem:amplecone}
Let ${\text{Amp}}(W)$ be the set of ample classes in $\NS(X_0)\otimes \R\simeq W\otimes \R$. This convex cone is invariant
under the action of $\Aut(X),$ hence under the action of a finite index  subgroup of the orthogonal group $O(W)$.
As such, ${\text{Amp}}(W)$ is equal to $\{ [D]\in W\, \vert \, [D]^2>0, \, [D]  \cdot[H]>0\}$.
\end{rem}

Let $S$ be a generic rational sextic, and $X$ be its associated Coble surface.
Theorem \ref{thm:coble} gives a recipe to construct automorphisms of~$X$: Let $\psi$ be an isometry 
of the lattice $W$; if $\psi$ is equal to the identity modulo $2,$ and $ \psi[H]\cdot [H] >0,$ 
then $\psi=g_*$ for a unique automorphism of $X$. Let us apply this idea to cook up a \good\ automorphism of $X$. 

\begin{lem}[Pell -Fermat equation]
\label{lem:PellFermat}
Let $Q(u,v)=au^2+ buv+cv^2$ be a quadratic binary form with integer coefficients. Assume that $Q$ is non 
degenerate, indefinite, and does not represent $0$. Then there exists an isometry $\phi$ in the orthogonal group ${\rm O}_Q(\Z)$  
with eigenvalues $\lambda(Q)>1>1/\lambda(Q)>0$.
\end{lem}

\begin{proof}[Sketch of the proof]
Isometries  in ${\rm O}_Q(\Z)$ correspond to units in the quadratic field defined by the polynomial $Q(t,1) \in \Z[t]$;
finding units, or isometries, is a special case of Dirichlet's units theorem and amounts to solve a Pell-Fermat equation (see \cite{Frohlich-Taylor:Book},
chapter V.1).
\end{proof}

Let $[D_1]$ and $[D_2]$ be the elements of $\NS(X)$ defined by
\begin{eqnarray*} 
[D_1]  & = & 6 [H] - [E_9] - [E_{10}] - \sum_{i=1}^{8} 2[E_i], \\
\! [D_2] &  =  & 6 [H] -  [E_7] - [E_8] - \sum_{i\neq 7,8}  2[E_i].
\end{eqnarray*}
Both of them have self intersection $2,$ are contained in $W,$ and intersect $[H]$ positively; as explained
in Remark \ref{rem:amplecone}, $[D_1]$ and $[D_2]$ are ample classes of $X_0$. 
Let $V$ be the plane containing $[D_1]$ and $[D_2],$ and $Q$ be the restriction  of the intersection form 
to $V$. If $u$ and $v$ are integers, then 
\[
Q(u[D_1]+v[D_2])= (u[D_1]+v[D_2]) \cdot (u[D_1]+v[D_2])= 2 u^2 + 8 uv + 2 v^2
\]
because $[D_1]\cdot [D_2]  = 4$. This quadratic form does not represent $0,$
because its discriminant is not a perfect square. From Lemma \ref{lem:PellFermat},
there is an isometry $\phi$ of $V$ with an eigenvalue $\lambda(\phi)>1$. In fact, an explicit 
computation shows that the group of isometries of $Q$ is the semi-direct product of
 the group $\Z/2\Z$ generated by the involution which permutes $[D_1]$ and $[D_2]$ and the
  cyclic group $\Z$ generated by the isometry $\phi$ defined by 
\[
\phi([D_1]) = 4 [D_1] + [D_2], \quad \phi([D_2])=-[D_1].
\]
The second iterate of $\phi$ is equal to the identity modulo $2$. Coble's theorem (Theorem~\ref{thm:coble}) now
implies that there exists a unique automorphism $g$ of $X$ such that 
\begin{enumerate}
\item $g_*$ coincides with $\phi^2$ on $V$;
\item $g_*$ is the identity on the orthogonal complement $V^\perp$.
\end{enumerate}
The dynamical degree of $g$ is the square of $\lambda(\phi),$ and is
equal to $7+4\sqrt{3}$. Properties $(i)$, $(ii)$, and $(iii)$ of Proposition \ref{pro:auto-good} are satisfied. Thus, large powers
of $g$ generate non trivial normal subgroups of the Cremona group: 

\begin{thm}
\label{thm:coblegood}
Let ${\mathbf{k}}$ be an algebraically closed field.
Let $X$ be a generic Coble surface defined over ${\mathbf{k}}$. There are hyperbolic automorphisms of $X$ that
generate non trivial normal subgroups of the Cremona group $\Bir(X_{\mathbf{k}})=\Bir(\P^2_{\mathbf{k}})$. 
\end{thm}

As a corollary, the Cremona group $\Bir(\P^2_{\mathbf{k}})$ is not simple 
if ${\mathbf{k}}$ is algebraically closed,
as announced in the Introduction.


%
%

%
%

\section{Complements}

\subsection{Polynomial automorphisms and monomial transformations}

The group of polynomial automorphisms of the affine plane, and the group of monomial transformations of $\cpd$ were both sources of inspiration for the results in this paper.
We now use these groups to construct hyperbolic elements $g$ of $\Bir(\cpd)$ for which $\lld g\rrd$ coincides with $\Bir(\cpd)$.

\subsubsection{Monomial transformations} 

Consider the group of monomial transformations of $\P^2_{\mathbf{k}}$. By definition, this  group is isomorphic to $\GL_2(\Z)$, acting  by 
\[
\left( \begin{matrix}
 a&b\\ c&d
\end{matrix} \right) \colon (x,y) \mapsto (x^ay^b,x^cy^d)
\]
in affine coordinates $(x,y)$. The matrix $-{\text{Id}}$ corresponds to the standard quadratic involution $\sigma(x,y)=(1/x,1/y)$.

If one considers $\PSL_2(\Z)$ as a subgroup of $\PSL_2(\R)\simeq \Isom(\Hyp^2)$, it is an interesting exercise to check that all hyperbolic matrices of $\PSL_2(\Z)$ are \good\ elements of $\PSL_2(\Z)$ (see \cite{Lamy:HDR}). 
However, when we see $\GL_2(\Z)$ as a subgroup of the Cremona group, we obtain the following striking remark.

\begin{pro} \label{pro:monomial}
Let $g\colon \P^2_{\mathbf{k}} \to \P^2_{\mathbf{k}}$ be a non-trivial monomial transformation.
The normal subgroup of $\Bir(\P^2_{\mathbf{k}})$ generated by $g$ is not proper:  $\lld g \rrd = \Bir(\P^2_{\mathbf{k}})$.
\end{pro}

\begin{rem}[Gizatullin and Noether]\label{rem:GN}
If $N$ is a normal subgroup of $\Bir(\P^2_{\mathbf{k}})$ containing a non trivial automorphism of $\P^2_{\mathbf{k}}$, then $N$
coincides with $\Bir(\P^2_{\mathbf{k}})$. The proof is as follows (see \cite{Gizatullin:1994, Cerveau-Deserti:2009}). Since $\Aut(\P^2_{\mathbf{k}})$ is the simple group $\PGL_3({\mathbf{k}})$
and $N$ is normal, $N$ contains $\Aut(\P^2_{\mathbf{k}})$. In particular, $N$ contains the automorphism $h$ defined
by 
\[
h(x,y)=(1-x,1-y)
\]
in affine coordinates. An easy calculation shows that the standard quadratic involution $\sigma$
satisfies $\sigma = (h\sigma)h(h\sigma)^{-1}$; hence, $\sigma$ is conjugate to $h$, and $\sigma$ 
is contained in $N$. The conclusion follows from Noether's theorem, which states that
$\sigma$ and $\Aut(\P^2_{\mathbf{k}})$ generate $\Bir(\P^2_{\mathbf{k}})$ (see \cite{KSC:book}, \S 2.5).
\end{rem}

\begin{proof}[Proof of Proposition \ref{pro:monomial}]
Let $g = \bigl( \begin{smallmatrix}
 a&b\\ c&d
\end{smallmatrix} \bigr)$ be any non trivial monomial map in $\Bir(\cpd)$. 
The commutator of $g$ with the diagonal map $f(x,y) = (\alpha x,\beta y)$ is the diagonal map 
\begin{equation}\label{eq:com}
g^{-1}f^{-1}gf \colon (x,y)\mapsto (\alpha^{1-d}\beta^b x, \alpha^c\beta^{1-a} y).
\end{equation}
Thus, the normal subgroup $\lld g \rrd$ contains an element of $\Aut(\cpd)\setminus\{\rm Id\}$ and 
Remark \ref{rem:GN} concludes the proof. \end{proof}

\subsubsection{Polynomial automorphisms}\label{par:Danilovbad}
As mentioned in the Introduction, Danilov proved that the group $\Aut[{\mathbb{A}}^2_{\mathbf{k}}]_1$ 
of polynomial automorphisms of the affine plane ${\mathbb{A}}^2_{\mathbf{k}}$ 
with Jacobian determinant one is not simple. 
Danilov's proof uses an action on a tree. Since $\Aut[{\mathbb{A}}^2_{\mathbf{k}}]_1$ is a subgroup of $\Bir(\P^2_{\mathbf{k}})$, we also have the action of $\Aut[{\mathbb{A}}^2_{\mathbf{k}}]_1$ on the hyperbolic space $\hypman$. It is a nice observation that $g\in \Aut[{\mathbb{A}}^2_{\mathbf{k}}]_1$ 
determines a hyperbolic isometry of the tree if and only if it determines a hyperbolic isometry of $\hypman$: In both cases, hyperbolicity corresponds to an exponential growth of the sequence of degrees $\deg(g^n)$.

Fix a number $a\in {\mathbf{k}}^*$ and a polynomial   $p\in {\mathbf{k}}[y]$ of degree $d\geq 2$, and consider the  automorphism of ${\mathbb{A}}^2_{\mathbf{k}}$ defined by
$
h(x,y) = (y,p(y) - a x).
$
This automorphism determines an algebraically stable birational transformation of $\P^2_{\mathbf{k}}$, namely
\begin{equation}\label{eq:henon}
h[x:y:z]=[y z^{d-1}: P(y,z)-axz^{d-1}:z^d],
\end{equation}
where $P(y,z)=p(y/z) z^d$. There is a unique indeterminacy point $\Ind(h)=\{ [1:0:0] \}$, and a unique 
indeterminacy point for the inverse, $\Ind(h^{-1})=\{ [0:1:0] \}$. This Cremona transformation is hyperbolic, with translation 
length $\lgt(h_*)= \log(d)$; in particular, the translation length goes to infinity with $d$.

\begin{pro}[See also  \cite{Cerveau-Deserti:2009}]
For all integers $d\geq 2$, equation (\ref{eq:henon}) defines a subset $H_d\subset \V_d$ 
which depends on $d+2$ parameters and satisfies: For all $h$ in $H_d$, $h$ is a hyperbolic,
algebraically stable Cremona transformation, but the normal subgroup generated by $h$
coincides with $\Bir(\P^2_{\mathbf{k}})$.
\end{pro}

\begin{proof}
The automorphism $h$ is the composition of the Jonqui\`eres transformation 
$
(x,y)\mapsto (P(y)-ax,y)
$
and the linear map $(x,y)\mapsto (y,x)$. As such, $h$ is an element of $\V_d$. 
If $f$ denotes the automorphism $f(x,y)=(x,y+1)$, then 
\[
(h^{-1} \circ f \circ h) (x,y) = (x-a^{-1},y)
\]
is linear (thus, the second step in the proof of proposition \ref{prop:f=id} does not work for $h$).
As a consequence, the commutator $f^{-1}h^{-1}fh$ is linear  and $\lld h \rrd$ 
intersects $\Aut(\P^2_{\mathbf{k}})$ non trivially. The conclusion follows from Remark~\ref{rem:GN}.
\end{proof}

Note that, for $h$ in $H_d$ and large integers $n$,  we expect $\lld h^n \rrd$ to be a proper normal subgroup 
of the Cremona group:

\begin{que}
Let $\mathbf{k}$ be any field. Consider the polynomial automorphism 
\[
g\colon (x,y) \mapsto (y, y^2 + x).
\] 
Does there exist an integer $n>0$ (independent of $\mathbf{k}$) such that 
$\lld g^n \rrd$ is a proper normal subgroup of  $\Bir(\P^2_{\mathbf{k}})$?
\end{que} 

The main point would be to adapt Step 2 in the proof of Proposition \ref{prop:f=id}. 

\subsection{Projective surfaces}

The reason why we focused on the group $\Bir(\cpd)$ comes from the fact that 
$\Bir(X)$ is small compared to $\Bir(\cpd)$ when $X$ is an irrational complex 
projective surface. The proof of the following proposition illustrates this property.

\begin{pro}
Let $X$ be a complex projective surface. If the group $\Bir(X)$ is infinite
and simple, then $X$ is birationally equivalent to $C\times \P^1(\C)$ where
$C$ is a curve with trivial automorphism group. 
\end{pro}

\begin{proof}[Sketch of the proof]
Assume, first, that the Kodaira dimension of $X$ is non negative. Replace $X$ by its unique 
minimal model, and identify $\Bir(X)$ with $\Aut(X)$. The group $\Aut(X)$
acts on the homology of $X$, and the kernel is equal to its connected component $\Aut(X)^0$ up
to  finite index. The action on the homology group provides
a morphism to $\GL_n(\Z)$ for some $n\geq 1$. Reducing modulo $p$
for large primes $p$, one sees that $\GL_n(\Z)$ is residually finite. Since
$\Aut(X)$ is assumed to be simple, this implies that $\Aut(X)$ coincides
with $\Aut(X)^0$. But $\Aut(X)^0$ is abelian for
surfaces with non negative Kodaira dimension (see \cite{Akhiezer:Book}). Thus $\Bir(X)=\Aut(X)$
is not both infinite and simple when the Kodaira dimension of $X$ is $\geq 0$. 
Assume now that $X$ is ruled and not rational. Up to a birational
change of coordinates, $X$ is a product $\P^1_\C\times C$
where $C$ is a smooth curve of genus $g(C)\geq 1$. The group 
$\Bir(X)$ projects surjectively onto $\Aut(C)$. By simplicity, 
$\Aut(C)$ must be trivial. In that case, $\Bir(X)$ coincides
with the infinite simple group $\PGL_2({\mathcal{M}}(C))$ where
${\mathcal{M}}(C)$ is the field of meromorphic functions of $C$. 
The remaining case is when $X$ is rational, and Theorem~A
concludes the proof. \end{proof}

\subsection{SQ-universality and the number of quotients}
\label{par:SQ}
As a direct consequence of Theorem \ref{thm:criterion} and the existence of \good\ elements in $\Bir(\cpd)$, 
{\sl{the Cremona group $\Bir(\cpd)$ has an uncountable number of distinct normal subgroups}}.
Recently, Dahmani, Guirardel, and Osin obtained a better, much more powerful version of Theorem \ref{thm:criterion},
which applies to the Cremona group in the same way as Theorem~C implies Theorem~A, because  the existence of {\good} 
elements in $\Bir(\cpd)$ is sufficient to apply Dahmani, Guirardel, and Osin's theorems. We only describe one consequence
of their results that strengthen the above mentioned fact that $\Bir(\cpd)$ has uncountably many normal subgroups, and refer to \cite{DahGui} for other statements. 

A group is said to be {\textbf{SQ-universal}} (or SubQuotient-universal) if every countable group can be embedded into one of its quotients. For example, the pioneering work \cite{HNN} proves that the free group over two generators is SQ-universal.
If $G$ is a non elementary hyperbolic group, then $G$ is SQ-universal. This result has been obtained by Delzant and Olshanskii in \cite{Delzant:1996} and \cite{Ol}.

\begin{thm}[see \cite{DahGui}] Let ${\mathbf{k}}$ be an algebraically closed field. The Cremona group 
$\Bir(\P^2_{\mathbf{k}})$ is SQ-universal. 
\end{thm}

 Note that SQ-universality implies the existence of an uncountable number of non isomorphic quotients.

%
%

\appendix
\section{The Cremona group is not an amalgam} \label{deCornulier}
\begin{center}{\sc Yves de Cornulier \footnote{Laboratoire de Math\'ematiques,
B\^atiment 425, Universit\'e Paris-Sud 11,
91405 Orsay\\FRANCE; yves.cornulier@math.u-psud.fr}}\end{center}
\medskip

Let $\mathbf{k}$ be a field. The {\em Cremona group} $\mathsf{Bir}(\mathbb{P}^d_\mathbf{k})$ of $\mathbf{k}$ in dimension $d$ is defined as the group of birational transformations of the $d$-dimensional $\mathbf{k}$-affine space. It can also be described as the group of $\mathbf{k}$-automorphisms of the field of rational functions $\mathbf{k}(t_1,\dots,t_d)$. We endow it with the discrete topology.

Let us say that a group has {\em Property} $(\textnormal{F}\mathbf{R})_\infty$ if it satisfies the following
\begin{enumerate}
\item[(A.1)]\label{every1} For every isometric action on a complete real tree, every element has a fixed point.
\end{enumerate}

Here we prove the following result.

\begin{thm}\label{main}
If $\mathbf{k}$ is an algebraically closed field, then $\mathsf{Bir}(\mathbb{P}^2_\mathbf{k})$ has Property
$(\textnormal{F}\mathbf{R})_\infty$.
\end{thm}

\begin{cor}\label{cmain}
The Cremona group does not decompose as a nontrivial amalgam.
\end{cor}

Recall that a {\em real tree} can be defined in the following equivalent ways (see \cite{Ch})
\begin{itemize}
\item A geodesic metric space which is 0-hyperbolic in the sense of Gromov;
\item A uniquely geodesic metric space for which $[ac]\subset [ab]\cup [bc]$ for all $a,b,c$;
\item A geodesic metric space with no subspace homeomorphic to the circle.
\end{itemize}
In a real tree, a {\em ray} is a geodesic embedding of the half-line. An {\em end} is an equivalence class of rays modulo being at bounded distance. For a group of isometries of a real tree, to {\em stably fix} an end means to pointwise stabilize a ray modulo eventual coincidence (it means it fixes the end as well as the corresponding Busemann function).

For a group $\Gamma$, Property $(\textnormal{F}\mathbf{R})_\infty$ has the following equivalent characterizations:
\begin{enumerate}
\item[(A.2)]\label{every2} For every isometric action of $\Gamma$ on a complete real tree, every finitely generated subgroup has a fixed point.
\item[(A.3)]\label{end} Every isometric action of $\Gamma$ on a complete real tree either has a fixed point, or stably fixes a point at infinity (in the sense above).
\end{enumerate}

The equivalence between these three properties is justified in Lemma \ref{eqofdef}.
Similarly, we can define the weaker {\em Property} $(\textnormal{FA})_\infty$, replacing complete real trees by ordinary trees (and allowing fixed points to be middle of edges), and the three corresponding equivalent properties are equivalent \cite{Se} to the following fourth: the group is not a nontrivial amalgam and has no homomorphism onto the group of integers. In particular, Corollary \ref{cmain} follows from Theorem \ref{main}.

\begin{rem}\label{patapouf}{\bf a.--} Note that the statement for actions on real trees (rather than trees) is strictly stronger. Indeed, unless $\mathbf{k}$ is algebraic over a finite field, the group $\mathsf{PGL}_2(\mathbf{k})=\mathsf{Bir}(\mathbb{P}^1_\mathbf{k})$ does act isometrically on a real tree with a hyperbolic element (this uses the existence of a nontrivial real-valued valuation on $\mathbf{k}$), but does not have such an action on a discrete tree (see Proposition \ref{auxiliary}).


\noindent{\bf \ref{patapouf}. b.--} Note that $\mathsf{Bir}(\mathbb{P}^2_\mathbf{k})$ always has an action on a discrete tree with no fixed point (i.e.~no fixed point on the 1-skeleton) when $\mathbf{k}$ is algebraically closed, and more generally whenever $\mathbf{k}$ is an infinitely generated field: write, with the help of a transcendence basis, $\mathbf{k}$ as the union of an increasing sequence of proper subfields $\mathbf{k}=\bigcup \mathbf{k}_n$, then $\mathsf{Bir}(\mathbb{P}^2_\mathbf{k})$ is the increasing union of its proper subgroups $\mathsf{Bir}\left(\mathbb{P}^2_{\mathbf{k}_n}\right)$, and thus acts on the disjoint union of the coset spaces $\mathsf{Bir}(\mathbb{P}^2_\mathbf{k})/\mathsf{Bir}\left(\mathbb{P}^2_{\mathbf{k}_n}\right)$, which is in a natural way the vertex set of a tree on which $\mathsf{Bir}(\mathbb{P}^2_\mathbf{k})$ acts with no fixed point (this is a classical construction of Serre \cite[Chap I, \S 6.1]{Se}).

\noindent{\bf \ref{patapouf}. c.--} Theorem \ref{main} could be stated, with a similar proof, for actions on $\Lambda$-trees when $\Lambda$ is an arbitrary ordered abelian group (see \cite{Ch} for an introduction to $\Lambda$-metric spaces and $\Lambda$-trees).
\end{rem}

In the following, $\mathcal{T}$ is a complete real tree; all actions on $\mathcal{T}$ are assumed to be isometric. We begin by a few lemmas.

\begin{lem}\label{trealign}
Let $x_0,\dots,x_k$ be points in a real tree $\mathcal{T}$ and $s\ge 0$. Assume that $d(x_i,x_j)=s|i-j|$ holds for all $i,j$ such that $|i-j|\le 2$. Then it holds for all $i,j$.
\end{lem}
\begin{proof}
This is an induction; for $k\le 2$ there is nothing to prove. Suppose $k\ge 3$ and the result known up to $k-1$, so that the formula holds except maybe for $\{i,j\}=\{0,k\}$. Join $x_i$ to $x_{i+1}$ by segments. By the induction, the $k-1$ first segments, and the $k-1$ last segments, concatenate to geodesic segments. But the first and the last of these $k$ segments are also disjoint, otherwise picking the ``smallest" point in the last segment that also belongs to the first one, we find an injective loop, contradicting that $\mathcal{T}$ is a real tree. Therefore the $k$ segments concatenate to a geodesic segment and $d(x_0,x_k)=sk$.
\end{proof}

\begin{lem}\label{hfa}
If $\mathbf{k}$ is any field and $d\ge 3$, then $\Gamma=\mathsf{SL}_d(\mathbf{k})$ has
Property $(\textnormal{F}\mathbf{R})_\infty$. In particular, if $\mathbf{k}$ is algebraically closed, then $\mathsf{PGL}_d(\mathbf{k})$ has Property $(\textnormal{F}\mathbf{R})_\infty$.
\end{lem}
\begin{proof}
Let $\Gamma$ act on $\mathcal{T}$. Let $F$ be a finite subset of $\Gamma$. Every element of $F$ can be written as a product of elementary matrices. Let $A$ be the (finitely generated) subring of $\mathbf{k}$ generated by all entries of those matrices. Then $F\subset \mathsf{EL}_d(A)$, the subgroup of $\mathsf{SL}_d(A)$ generated by elementary matrices. By the Shalom-Vaserstein Theorem (see \cite{ej}), $\mathsf{EL}_d(A)$ has Kazhdan's Property~(T) and in particular has a fixed point in $\mathcal{T}$, so $F$ has a fixed point in $\mathcal{T}$. (There certainly exists a more elementary proof, but this one also shows that for every isometric action of $\mathsf{SL}_d(\mathbf{k})$ on a Hilbert space, every finitely generated subgroup fixes a point.)
\end{proof}

Fix the following notation: $G=\mathsf{Bir}(\mathbb{P}^2_\mathbf{k})$; $H=\mathsf{PGL}_3(\mathbf{k})=\mathsf{Aut}(\mathbb{P}^2_\mathbf{k})\subset G$; $\sigma$ is the Cremona involution, acting in affine coordinates by $\sigma(x,y)=(x^{-1},y^{-1})$. The Noether-Castelnuovo Theorem is that $G=\langle H,\sigma\rangle$. Let $C$ be the standard Cartan subgroup of $H$, that is, the semidirect product of the diagonal matrices by the Weyl group (of order 6).
Let $\mu\in H$ be the involution given in affine coordinates by $\mu(x,y)=(1-x,1-y)$.

\begin{lem}\label{cmaximal}We have $\langle C,\mu\rangle=H$.
\end{lem}
\begin{proof}
We only give a sketch, the details being left to the reader. In $\mathsf{GL}_3$, $\mu$ can be written as the matrix
$\begin{pmatrix} -1 & 0 & 1\\ 0 & -1 & 1\\ 0 & 0 & 1\end{pmatrix}$. Multiply $\mu$ by its conjugate by a suitable diagonal matrix to obtain an elementary matrix; conjugating by elements of $C$ provide all elementary matrices and thus we obtain all matrices with determinant one; since $C$ also contains diagonal matrices, we are done. 
\end{proof}


\begin{lem}\label{hpara}
Let $G$ act on $\mathcal{T}$ so that $H$ has no fixed point and has a (unique) stably fixed end. Then $G$ stably fixes this unique end.
\end{lem}
\begin{proof}
Let $\omega$ be the unique end stably fixed by $H$ (recall that if it is represented by a ray $(x_t)$, this means that for every $h\in H$ there exists $t_0=t_0(h)$ such that $h$ fixes $x_t$ for all $t\ge t_0$). Then $\sigma H\sigma^{-1}$ stably fixes $\sigma\omega$. In particular, since $\sigma C\sigma^{-1}=C$, the end $\sigma\omega$ is also stably fixed by $C$. If $\sigma\omega=\omega$, then $\omega$ is stably fixed by $\sigma$ and then by the Noether-Castelnuovo Theorem, $\omega$ is stably fixed by $G$.
Otherwise, let $D$ be the line joining $\omega$ and $\sigma\omega\neq\omega$. Since both ends of $D$ are stably fixed by $C$, the line $D$ is pointwise fixed by $C$. Also, $\mu$ stably fixes the end $\omega$ and therefore for some $t$, $x_t$ is fixed by $\mu$ and therefore, by Lemma \ref{cmaximal}, is fixed by all of $H$, contradicting the assumption.
\end{proof}


\begin{proof}[Proof of Theorem \ref{main}]
Note that $\mu\in H$ and $\mu\sigma$ has order three. It follows that $\sigma=(\mu\sigma)\mu(\mu\sigma)^{-1}$. Using the Noether-Castelnuovo Theorem, it follows that $H_1=H$ and $H_2=\sigma H\sigma^{-1}$ generate $G$. 

Consider an action of $G$ on $\mathcal{T}$. By Lemmas \ref{hfa} and \ref{hpara}, we only have to consider the case when $H$ has a fixed point; in this case, let us show that $G$ has a fixed point. Assume the contrary.
Let $\mathcal{T}_i$ be the set of fixed points of $H_i$ $(i=1,2)$; they are exchanged by $\sigma$ and since $\langle H_1,H_2\rangle=G$, we see that the two trees $\mathcal{T}_1$ and $\mathcal{T}_2$ are disjoint. Let $\mathcal{S}=[x_1,x_2]$ be the minimal segment joining
the two trees $(x_i\in\mathcal{T}_i)$ and $s>0$ its length. Then $\mathcal{S}$ is pointwise fixed by $C \subset H_1\cap H_2$ and reversed by $\sigma$.

\noindent{\em Claim.} For all $k\ge 1$, the distance of $x_1$ with $(\sigma\mu)^kx_1$
is exactly $sk$.

The claim is clearly a contradiction since $(\sigma\mu)^3=1$.
To check the claim, let us apply Lemma \ref{trealign} to the sequence $((\sigma\mu)^kx_1)$: namely to check that $$d((\sigma\mu)^kx_1,(\sigma\mu)^\ell x_1)=|k-\ell|s$$ for all $k,\ell$ it is enough to 
check it for $|k-\ell|\le 2$; by translation it is enough to check it for $k=1,2$ and $\ell=0$.
For $k=1$, $d(\sigma \mu x_1,x_1)=d(\sigma x_1,x_1)=d(x_2,x_1)=s$. 
Since $\langle C,\mu\rangle=H$ by Lemma \ref{cmaximal}, the image of $[x_1,x_2]$ by $\mu$ is a segment $[x_1,\mu x_2]$ intersecting the segment $[x_1,x_2]$ only 
at $x_1$; in particular, $d(x_2,\mu x_2)=2s$. Hence, under the assumptions of the claim
$$d(\sigma \mu \sigma \mu x_1,x_1)=d(\mu \sigma x_1,\sigma x_1)= d(\mu x_2,x_2)=2s.$$
This proves the claim for $k=2$ and the proof is complete.
\end{proof}


For reference we include

\begin{pro}\label{auxiliary}
If $\mathbf{k}$ is algebraically closed, the group $\mathsf{PGL}_2(\mathbf{k})$ has Property $(\textnormal{FA})_\infty$ but, unless $\mathbf{k}$ is an algebraic closure of a finite field, does not satisfy $(\textnormal{F}\mathbf{R})_\infty$.
\end{pro}
\begin{proof}
If $\mathbf{k}$ has characteristic zero, the group $\mathsf{PGL}_2(\mathbf{k})$ has the property that the square of every element is divisible (i.e.\ has $n$th roots for all $n>0$). This implies that no element can act hyperbolically on a discrete tree: indeed, in the automorphism group of a tree, the translation length of any element is an integer and the translation length of $x^n$ is $n$ times the translation length of $x$. If $\mathbf{k}$ has characteristic $p$ the same argument holds: for every $x$, $x^{2p}$ is divisible.

On the other hand, let $I$ be a transcendence basis of $\mathbf{k}$ and assume it nonempty, and $x_0\in I$. Set $\mathbf{L}=\mathbf{k}(I-\{i_0\})$, so that $\mathbf{k}$ is an algebraic closure of $\mathbf{L}(x_0)$. The nontrivial discrete valuation of $\mathbf{L}(\!(x_0)\!)$ uniquely extends to a nontrivial, $\mathbf{Q}$-valued valuation on an algebraic closure. It restricts to a non-trivial $\mathbf{Q}$-valued valuation on $\mathbf{k}$.

The remaining case is the case of an algebraic closure of the rational field $\mathbf{Q}$; pick any prime $p$ and restrict the $p$-valuation from an algebraic closure of $\mathbf{Q}_p$.

Now if $F$ is any field valued in $\mathbf{R}$, then $\mathsf{PGL}_2(F)$ has a natural action on a real tree, on which an element $\textnormal{diag}(a,a^{-1})$, for $|a|>1$, acts hyperbolically.

(If $\mathbf{k}$ is algebraic over a finite field, then $\mathsf{PGL}_2(\mathbf{k})$ is locally finite and thus satisfies $(\textnormal{FA})_\infty$.)
\end{proof}

\begin{lem}The three definitions of $(\textnormal{F}\mathbf{R})_\infty$ in the introduction are equivalent.\label{eqofdef}
\end{lem}
\begin{proof}[Sketch of proof]
The implications (A.3)$\Rightarrow$(A.2)$\Rightarrow$(A.1) are clear. (A.1)$\Rightarrow$(A.2) is proved for trees in \cite[Chap.\ I \S 6.5]{Se}, the argument working for real trees. Now assume (A.2) and let us prove (A.3). Fix a point $x_0$. For every finite subset $F$ of the group, let $\mathcal{S}_F$ be the segment joining $x_0$ to the set of $F$-fixed points. Then the union of $\mathcal{S}_F$, when $F$ ranges over finite subsets of the group, is a geodesic emanating from 0. If it is bounded, its other extremity (which exists by completeness) is a fixed point. Otherwise, it defines a stably fixed end.
\end{proof}

\bibliographystyle{plain}
\bibliography{references-nsgc}

\def\cprime{$'$}
\begin{thebibliography}{10}

\bibitem{Akhiezer:Book}
Dmitri~N. Akhiezer.
\newblock {\em Lie group actions in complex analysis}.
\newblock Friedr. Vieweg \& Sohn, 1995.

\bibitem{Blanc:2009}
J{\'e}r{\'e}my Blanc.
\newblock Sous-groupes alg\'ebriques du groupe de {C}remona.
\newblock {\em Transform. Groups}, 14(2):249--285, 2009.

\bibitem{Blanc:2010}
J{\'e}r{\'e}my Blanc.
\newblock Groupes de {C}remona, connexit\'e et simplicit\'e.
\newblock {\em Ann. Sci. \'Ec. Norm. Sup\'er. (4)}, 43(2):357--364, 2010.

\bibitem{Boucksom-Favre-Jonsson:2008}
S{\'e}bastien Boucksom, Charles Favre, and Mattias Jonsson.
\newblock Degree growth of meromorphic surface maps.
\newblock {\em Duke Math. J.}, 141(3):519--538, 2008.

\bibitem{Bridson-Haefliger:Book}
Martin~R. Bridson and Andr{\'e} Haefliger.
\newblock {\em Metric spaces of non-positive curvature}, volume 319.
\newblock Springer-Verlag, Berlin, 1999.

\bibitem{Cantat:Preprint}
Serge Cantat.
\newblock Sur les groupes de transformations birationnelles des surfaces.
\newblock {\em Ann. of Math. (2)}, 174(1):299--340, 2011.

\bibitem{Cantat-Dolgachev}
Serge Cantat and Igor Dolgachev.
\newblock Rational surfaces with large groups of automorphisms.
\newblock {\em J. Amer. Math. Soc.}, 25(3):863--905, 2012.

\bibitem{Cerveau-Deserti:2009}
Dominique Cerveau and Julie D{\'e}serti.
\newblock Transformations birationnelles de petit degr\'e.
\newblock {\em Cours Sp\'ecialis\'es, Soci\'et\'e Math\'ematique de France, to
  appear.}

\bibitem{Chaynikov}
Vladimir Chaynikov.
\newblock On the generators of the kernels of hyperbolic group presentations.
\newblock {\em Algebra Discrete Math.}, 11(2):18--50, 2011.

\bibitem{Ch}
Ian Chiswell.
\newblock {\em Introduction to {$\Lambda$}-trees}.
\newblock World Scientific. World Scientific, Singapore, 2001.

\bibitem{Coble:Book}
Arthur~B. Coble.
\newblock {\em Algebraic geometry and theta functions}, volume~10 of {\em
  American Mathematical Society Colloquium Publications}.
\newblock American Mathematical Society, Providence, R.I., 1982.
\newblock Reprint of the 1929 edition.

\bibitem{CDP:Book}
Michel Coornaert, Thomas Delzant, and Athanase Papadopoulos.
\newblock {\em G\'eom\'etrie et th\'eorie des groupes}, volume 1441 of {\em
  Lecture Notes in Mathematics}.
\newblock Springer-Verlag, 1990.

\bibitem{Cossec-Dolgachev:book}
Fran{\c{c}}ois Cossec and Igor Dolgachev.
\newblock {\em Enriques surfaces. {I}}, volume~76 of {\em Progress in
  Mathematics}.
\newblock Birkh\"auser, 1989.

\bibitem{DahGui}
Fran\c{c}ois Dahmani, Vincent Guirardel, and Denis Osin.
\newblock Hyperbolic embeddings and rotating families in groups acting on
  hyperbolic spaces.
\newblock {\em arXiv:1111.7048}, 2011.

\bibitem{Danilov:1974}
Vladimir Danilov.
\newblock Non-simplicity of the group of unimodular automorphisms of an affine
  plane.
\newblock {\em Mat. Zametki}, 15:289--293, 1974.

\bibitem{Delzant:1996}
Thomas Delzant.
\newblock Sous-groupes distingu\'es et quotients des groupes hyperboliques.
\newblock {\em Duke Math. J.}, 83(3):661--682, 1996.

\bibitem{Deserti:2006bis}
Julie D{\'e}serti.
\newblock Groupe de {C}remona et dynamique complexe: une approche de la
  conjecture de {Z}immer.
\newblock {\em Int. Math. Res. Not.}, pages Art. ID 71701, 27, 2006.

\bibitem{Deserti:2006}
Julie D{\'e}serti.
\newblock Sur les automorphismes du groupe de {C}remona.
\newblock {\em Compos. Math.}, 142(6):1459--1478, 2006.

\bibitem{Deserti:2007}
Julie D{\'e}serti.
\newblock Le groupe de {C}remona est hopfien.
\newblock {\em C. R. Math. Acad. Sci. Paris}, 344(3):153--156, 2007.

\bibitem{Diller-Favre:2001}
Jeffrey Diller and Charles Favre.
\newblock Dynamics of bimeromorphic maps of surfaces.
\newblock {\em Amer. J. Math.}, 123(6):1135--1169, 2001.

\bibitem{Dolgachev:1984}
Igor Dolgachev.
\newblock On automorphisms of {E}nriques surfaces.
\newblock {\em Invent. Math.}, 76(1):163--177, 1984.

\bibitem{Dolgachev:1986}
Igor Dolgachev.
\newblock Infinite {C}oxeter groups and automorphisms of algebraic surfaces.
\newblock In {\em The {L}efschetz centennial conference, {P}art {I} ({M}exico
  {C}ity, 1984)}, volume~58 of {\em Contemp. Math.}, pages 91--106. Amer. Math.
  Soc., 1986.

\bibitem{Dolgachev:2008}
Igor Dolgachev.
\newblock Reflection groups in algebraic geometry.
\newblock {\em Bull. Amer. Math. Soc. (N.S.)}, 45(1):1--60, 2008.

\bibitem{Dolgachev-Ortland:1989}
Igor Dolgachev and David Ortland.
\newblock Point sets in projective spaces and theta functions.
\newblock {\em Ast\'erisque}, (165), 1988.

\bibitem{Dolgachev-Zhang:2001}
Igor Dolgachev and De-Qi Zhang.
\newblock Coble rational surfaces.
\newblock {\em Amer. J. Math.}, 123(1):79--114, 2001.

\bibitem{Enriques:1894}
Federigo Enriques.
\newblock {C}onferenze di geometria. {F}ondamenti di una geometria
  iperspaziale.
\newblock {\em lit., Bologna}, 1894-95.

\bibitem{ej}
Mikhail Ershov and Andrei Jaikin-Zapirain.
\newblock Property ({T}) for noncommutative universal lattices.
\newblock {\em Invent. Math.}, 179(332):303--347, 2010.

\bibitem{Favre:Bourbaki}
Charles Favre.
\newblock Le groupe de {C}remona et ses sous-groupes de type fini.
\newblock {\em Ast\'erisque}, (332):Exp. No. 998, vii, 11--43, 2010.
\newblock S{\'e}minaire Bourbaki. Volume 2008/2009.

\bibitem{Frohlich-Taylor:Book}
Albrecht Fr{\"o}hlich and Martin~J. Taylor.
\newblock {\em Algebraic number theory}.
\newblock Cambridge University Press, 1993.

\bibitem{Furter-Lamy:Preprint}
Jean-Philippe Furter and St{\'e}phane Lamy.
\newblock Normal subgroup generated by a plane polynomial automorphism.
\newblock {\em Transform. Groups}, 15(3):577--610, 2010.

\bibitem{Ghys-delaHarpe:Book}
{\'E}tienne Ghys and Pierre de~la Harpe.
\newblock {\em Sur les groupes hyperboliques d'apr\`es {M}ikhael {G}romov},
  volume~83 of {\em Progr. Math.}
\newblock Birkh\"auser Boston, 1990.

\bibitem{Gizatullin:1980}
Marat Gizatullin.
\newblock Rational {$G$}-surfaces.
\newblock {\em Izv. Akad. Nauk SSSR Ser. Mat.}, 44(1):110--144, 239, 1980.

\bibitem{Gizatullin:1994}
Marat Gizatullin.
\newblock The decomposition, inertia and ramification groups in birational
  geometry.
\newblock In {\em Algebraic geometry and its applications ({Y}aroslavl\cprime,
  1992)}, Aspects Math., E25, pages 39--45. Vieweg, Braunschweig, 1994.

\bibitem{Guedj:Panorama}
Vincent Guedj.
\newblock Propri{\'e}t{\'e}s ergodiques des applications rationnelles.
\newblock {\em Panorama de la Soci{\'e}t{\'e} math{\'e}matique deFrance, to
  appear}, 30:1--130, 2007.

\bibitem{Halphen:1882}
George Halphen.
\newblock Sur les courbes planes du sixi\`eme degr\'e \`a neuf points doubles.
\newblock {\em Bull. Soc. Math. France}, 10:162--172, 1882.

\bibitem{Hartshorne:book}
Robin Hartshorne.
\newblock {\em Algebraic geometry}.
\newblock Springer-Verlag, 1977.
\newblock Graduate Texts in Mathematics, No. 52.

\bibitem{HNN}
Graham Higman, B.~H. Neumann, and Hanna Neumann.
\newblock Embedding theorems for groups.
\newblock {\em J. London Math. Soc.}, 24:247--254, 1949.

\bibitem{Hubbard:Book}
John~Hamal Hubbard.
\newblock {\em Teichm\"uller theory and applications to geometry, topology, and
  dynamics. {V}ol. 1}.
\newblock Matrix Editions, Ithaca, NY, 2006.

\bibitem{Jung:1942}
Heinrich W.~E. Jung.
\newblock \"{U}ber ganze birationale {T}ransformationen der {E}bene.
\newblock {\em J. Reine Angew. Math.}, 184:161--174, 1942.

\bibitem{KSC:book}
J{\' a}nos Koll{\' a}r, Karen~E. Smith, and Alessio Corti.
\newblock {\em Rational and Nearly Rational Varieties}.
\newblock Cambridge University Press, 2004.

\bibitem{Lamy:2002}
St{\'e}phane Lamy.
\newblock Une preuve g\'eom\'etrique du th\'eor\`eme de {J}ung.
\newblock {\em Enseign. Math. (2)}, 48(3-4):291--315, 2002.

\bibitem{Lamy:HDR}
St{\'e}phane Lamy.
\newblock Groupes de transformations birationnelles de surfaces.
\newblock {\em M\'emoire d'habilitation}, 2010.

\bibitem{Lazarsfeld:Book}
Robert Lazarsfeld.
\newblock {\em Positivity in algebraic geometry. {I}}, volume~48.
\newblock Springer-Verlag, Berlin, 2004.
\newblock Classical setting: line bundles and linear series.

\bibitem{Lyndon-Schupp:Book}
Roger~C. Lyndon and Paul~E. Schupp.
\newblock {\em Combinatorial group theory}.
\newblock Classics in Mathematics. Springer-Verlag, Berlin, 2001.
\newblock Reprint of the 1977 edition.

\bibitem{Manin:Book}
Yuri~I. Manin.
\newblock {\em Cubic forms}, volume~4 of {\em North-Holland Mathematical
  Library}.
\newblock North-Holland Publishing Co., Amsterdam, second edition, 1986.

\bibitem{Mumford:1974}
David Mumford.
\newblock Algebraic geometry in mathematical developments arising from
  {H}ilbert problems.
\newblock In {\em ({P}roc. {S}ympos. {P}ure {M}ath., {V}ol. {XXVIII},
  {N}orthern {I}llinois {U}niv., {D}e {K}alb, {I}ll., 1974)}, pages 431--444.
  Amer. Math. Soc., Providence, R. I., 1976.

\bibitem{Nguyen:these}
Dat~Dang Nguyen.
\newblock Composantes irr{\'e}ductibles des transformations de {C}remona de
  degr{\'e}~$d$.
\newblock {\em Th{\`e}se de Math{\'e}matique, Univ. Nice, France}, 2009.

\bibitem{Ol}
Alexander~Yu. Ol{\cprime}shanski{\u\i}.
\newblock {${\rm SQ}$}-universality of hyperbolic groups.
\newblock {\em Mat. Sb.}, 186(8):119--132, 1995.

\bibitem{Osin:2010}
Denis Osin.
\newblock Small cancellations over relatively hyperbolic groups and embedding
  theorems.
\newblock {\em Ann. of Math. (2)}, 172(1):1--39, 2010.

\bibitem{Se}
Jean-Pierre Serre.
\newblock {\em Arbres, amalgames, {${\rm SL}\sb{2}$}}.
\newblock Soci\'et\'e Math\'ematique de France, Paris, 1977.
\newblock Avec un sommaire anglais, R\'edig\'e avec la collaboration de Hyman
  Bass, Ast\'erisque, No. 46.

\bibitem{Serre:Bourbaki}
Jean-Pierre Serre.
\newblock Le groupe de {C}remona et ses sous-groupes finis.
\newblock {\em Ast\'erisque}, (332):Exp. No. 1000, vii, 75--100, 2010.
\newblock S{\'e}minaire Bourbaki. Volume 2008/2009.

\end{thebibliography}

\end{document}